%% file: template-jems.tex
\documentclass{article}
\usepackage[journal=APDE,lang=british]{ems-journal-jems} 

\usepackage{textcomp}
\usepackage{mathrsfs}

\usepackage{chngcntr}
\usepackage{apptools}
\usepackage{verbatim}
\usepackage{tikz-cd}
\usepackage{subcaption}

\newtheorem{theorem}{Theorem}
\newtheorem{proposition}{Proposition}
\newtheorem{lemma}{Lemma}

\theoremstyle{definition}

\newtheorem{problem}{Problem}

\newtheorem{conjecture}{Conjecture}

\theoremstyle{remark}
\newtheorem{remark}{Remark}
\newtheorem{example}{Example}
\newtheorem{corollary}{Corollary}

\numberwithin{equation}{section}

\newcommand\restr[2]{{
		\left.\kern-\nulldelimiterspace 
		#1 
		\vphantom{\big|} 
		\right|_{#2} 
}}

\numberwithin{equation}{section}

\numberwithin{theorem}{section}
\numberwithin{proposition}{section}
\numberwithin{lemma}{section}
\numberwithin{remark}{section}
\numberwithin{exercise}{section}
\numberwithin{corollary}{section}
\numberwithin{problem}{section}
\numberwithin{example}{section}
\numberwithin{conjecture}{section}

\newcounter{desccount}

\begin{document}

\title{Variational description of uniform Lyapunov exponents via adapted metrics on exterior products}
\titlemark{Variational description of uniform Lyapunov exponents}


\emsauthor{1}{
	\givenname{Mikhail}
	\surname{Anikushin}
	\orcid{0000-0002-7781-8551}}{M.M.~Anikushin}

\Emsaffil{1}{
	\pretext{}
	\department{Department of Applied Cybernetics}
	\organisation{Faculty of Mathematics and Mechanics, St Petersburg University}
	\rorid{023znxa73}
	\address{Universitetskiy prospekt 28}
	\zip{198504}
	\city{Peterhof}
	\country{Russia}
	\posttext{}
	\affemail{demolishka@gmail.com}
	}
	

	


\classification[34K20, 37C45]{37L30}

\begin{abstract}
	In this work, we present a comprehensive study of the relationship among uniform Lyapunov exponents, the Liouville trace formula, and adapted metrics for cocycles in Hilbert spaces. First, we prove that uniform Lyapunov exponents can be approximated by constructing adapted metrics on exterior products. Next, we develop a general computational theory in an abstract setting, establish a generalized Liouville trace formula, and pose and discuss the symmetrization problem related to computations. Third, we discuss ergodic properties and upper semicontinuity in the context of subadditive families over a noncompact base. Furthermore, we use adapted metrics and the trace formula to obtain, for the first time, effective dimension estimates for a general class of delay equations. In particular, we illustrate this approach by deriving upper estimates for the Lyapunov dimension of global attractors in the Mackey--Glass equations and the periodically forced Suarez--Schopf delayed oscillator. As the delay value tends to infinity, the estimates appear to be asymptotically sharp.
\end{abstract}

\maketitle
\tableofcontents

\input{Introduction}
\input{CocyclesAndGrowth}
\input{ExteriorPowersAndLiouvilleFormula}
\input{ComputaionOfInfinitesimalExponents}
\input{LiouvFormula}
\input{ApplicationsDelayEqs}
\input{MackeyGlassEqs}
\input{SuarezSchopfOscill}

\appendix
\input{ErgVarPrinc}
\input{UpperSemiCont}
\input{Symmetrization}

\section*{Funding}
The reported study was funded by the Russian Science Foundation (Project 25-11-00147).

\section*{Data availability}
Data sharing not applicable to this article as no datasets were generated or analyzed during the current study.

\section*{Conflict of interest}
The author has no conflicts of interest to declare that are relevant
to the content of this article.







\end{document}

%% file: Introduction.tex
\section{Introduction}
This work is devoted to the study of linear cocycles in Hilbert spaces and their uniform Lyapunov exponents, which are concerned with the uniform growth of volumes under the action of cocycle mappings. In applications, the primary interest in these quantities is related to the concept of the (uniform) Lyapunov dimension and cases where the cocycle is given by the linearization of a semiflow (or nonlinear cocycle) over an invariant set. In this setting, the Lyapunov dimension bounds from above the fractal dimension of the invariant set (or its fibers), see \cite{ChepyzhovIlyin2004, KuzReit2020}. Moreover, under appropriate conditions, such bounds prevent the existence of certain high-dimensional invariant structures in the system \cite{LiMuldowney1995LowBounds, Smith1986HD} and allows to embed the invariant subset into a finite-dimensional space \cite{RobinsonBook2011, ZelikAttractors2022}. However, the Lyapunov dimension usually reflects not any purely geometric dimensions of the invariant set, but rather their possible expansions under perturbations and bifurcations of the system, see \cite{TuraevZelik2003} for an illustrative example. Furthermore, as such a quantity is upper semicontinuous with respect to perturbations of the cocycle (see Appendix \ref{SEC: UpperSemiContLyapDim}), it is more convenient for global analysis, providing a tool to control the purely geometric dimensions.

Usually, effective upper estimates for the Lyapunov dimension of cocycles in Hilbert spaces are obtained using the Liouville trace formula, which provides an exact description of the evolution of a particular $m$-dimensional volume. To derive uniform bounds, one has to compute or estimate from above consecutive sums $\beta_{1} + \ldots + \beta_{m}$ of the so-called \textit{trace numbers} $\beta_{1} \geq \beta_{2} \geq \ldots$, which depend on points from the base and represent infinitesimal analogues of singular values of the cocycle mappings. Being maximized or time-averaged over the base, such sums provide an upper bound for the sum of the first $m$ uniform Lyapunov exponents. It is often possible to express trace numbers as eigenvalues of self-adjoint operators and thus justify what we call the \textit{symmetrization procedure}, see Appendix \ref{SEC: SymmetrizationOfOperators}. In the context of standard metrics, this approach may produce physically relevant estimates for global attractors, see \cite{Temam1997, ZelikAttractors2022}, or even allow exact computations of the Lyapunov dimension, see \cite{Doeringetal1987}. However, for certain problems discussed below, it may also provide very rough estimates that are inappropriate.

There are well-known applications of the trace formula in adapted metrics that broaden the scope of its applications. In this direction, Leonov proposed to vary constant metrics within their conformal class using special Lyapunov-like functions, see \cite{LeonovBoi1992}. To date, this approach, known as the Leonov method \cite{Kuznetsov2016}, has led to exact computations of the Lyapunov dimension for global attractors in the Lorenz \cite{LeoKuzKorzhKusakin2016, KuzMokKuzKud2020} and Lorenz-like systems \cite{Leonov2016LorenzLike, LeoAlexKuzSMsystem2015} and invariant sets of the H\'{e}non map \cite{LeonovForHenonLor2001}. In Section \ref{SEC: LiouvilleTraceFormulaAndSymProc}, we develop it in a general form.

In \cite{Smith1986HD}, Smith obtained dimension estimates for ODEs in $\mathbb{R}^{n}$ developing an approach that involves adapted metrics and the Liouville trace formula. Specifically, these metrics are used to establish uniform lower bounds for all singular values of the cocycle mappings. Such bounds are then employed to estimate from above, for $m \in \{1,\ldots,n\}$, the growth of $m$-dimensional volumes via the growth of $n$-dimensional volumes, which can be computed using the classical Liouville formula involving the traces of matrices. For infinite-dimensional systems, this approach can be generalized using inertial manifolds, see \cite[Theorem 12]{Anikushin2022Semigroups}. This highlights some impracticality of the approach, as it requires computing inertial forms (vector fields on inertial manifolds), and also underscores its artificiality in general, since the existence of nontrivial exponential dichotomies is a more complex issue that goes beyond dimension estimates\footnote{In more common terms of the spectral theory, dimension estimates focus on estimating the spectral bound, while exponential dichotomies pertain to separating specific parts of the spectrum.}. However, there is an interesting case in infinite dimensions to which this approach can be applied. It involves delay equations in $\mathbb{R}^{n}$ with small delays, for which inertial manifolds always exist and it is possible to obtain a series expansion of the inertial form in powers of the delay parameter, see \cite{Chicone2004}. In particular, this allows us to transfer estimates from the ODE corresponding to zero delay to the system with small delays. We refer to \cite{Anikushin2020Geom, Zelik2014, ZelikAttractors2022} for discussions on theoretical issues concerned with inertial manifolds and \cite{Anikushin2020FreqDelay, Anikushin2020FreqParab, AnikushinRom2023SS, AnikushinAADyn2021} for applications.

There are known examples where applications of the Liouville trace formula in standard metrics yield inappropriate results, but with the aid of adapted metrics one may still obtain relevant estimates. Such examples include damped hyperbolic-like equations, for which the results of Turaev and Zelik  \cite{ZelikAttractors2022, TuraevZelik2003} and Ghidaglia \cite{Ghidaglia1988Schroedinger, Ghidaglia1988KortewegdeVries} should be mentioned.

Another example is provided by delay equations. However, there is a long history associated with this problem. Starting from the pioneering paper of Mallet-Paret  \cite{MalletParet1976}, most works concerning dimension estimates for such equations relied solely on the compactness properties of the corresponding cocycles, thereby yielding only qualitative conclusions regarding the finiteness of dimensions. This is reflected in classical monographs such as \cite{HaleLunel1993}, as well as in more recent ones \cite{CarvalhoLangaRobinson2012, Chueshov2015}. Even recent studies, such as the work of Hu and Caraballo \cite{HuCaraballo2024}, continue to develop this approach. In \cite{HuCaraballo2024}, the resulting estimates are called \textit{explicit}, but they are not effective because they involve characteristics related to exponential dichotomies of delay operators. Moreover, the use of exponential dichotomies is likely to produce incorrect asymptotic behavior with respect to the system parameters and, certainly, complicate the analysis of the estimates.

It seems that So and Wu \cite{SoWu1991} were the first to apply the Liouville trace formula to (reaction-diffusion) delay equations\footnote{This requires to pose such equations in the setting of a proper Hilbert space (most of works treat the equations in a Banach space of continuous functions).}, but they did not succeed in deriving effective estimates from it. Namely, at the end of the introduction in \cite{SoWu1991} they wrote: ``The verification of the hypotheses in Theorem 3.10 for reaction-diffusion systems with delays is a non-trivial task and will be reported in a future paper.'' Since then, no similar results have been published.

In our work \cite{Anikushin2022Semigroups}, we illustrated the problem by means of a scalar delay equation and demonstrated that applying the trace formula in the standard metric yields uninformative trace numbers, as no decay of volumes can be observed from them. Anticipating the results of the present paper, we note that in Section \ref{SEC: DimensionEstimatesViaLiouvDelayEqs} we provide a satisfactory resolution to this problem for a large class of delay equations in $\mathbb{R}^{n}$ with discrete delays by constructing constant adapted metrics that always allow effective observation of the decay of high-dimensional volumes. In the case of scalar equations with a single delay, as considered in Sections \ref{SEC: Mackey-GlassDimEst} and \ref{SEC: PertSuarezSchopfDimEstimate}, the resulting estimates are asymptotically sharp as the delay tends to infinity. To the best of our knowledge, this is the first time effective dimension estimates have been obtained for chaotic attractors of delay equations.

Nevertheless, let us continue to discuss other approaches.

In \cite{Anikushin2022Semigroups}, a comparison principle was developed for autonomous scalar equations with monotone feedback by combining the Ergodic Variational Principle for subadditive families (see Appendix \ref{APP: EVPNonCompact}) with the results on monotonicity properties of compound cocycles obtained by Mallet-Paret and Nussbaum \cite{MalletParretNussbaum2013}, as well as the Poincar\'{e}--Bendixson theory for such equations due to Mallet-Paret and Sell \cite{MalletParetSell1996}. This principle allows for the comparison of the largest uniform Lyapunov exponent of compound cocycles with the spectrum of certain stationary equations that can be analyzed directly or with the aid of numerical experiments. However, scalar equations, not to mention systems of such equations, with chaotic behavior fall outside the scope of this principle.

In our adjacent paper \cite{Anikushin2023Comp}, we developed an approach for a general class of delay equations in $\mathbb{R}^{n}$ to investigate uniform exponential stability --- or, more generally, the existence of gaps in the Sacker-Sell spectrum --- of compound cocycles generated by delay equations. In this context, the compound cocycle is treated as a perturbation of a stationary problem. Subsequently, so-called frequency inequalities are derived, which guarantee that certain dichotomy properties of the stationary problem are preserved under perturbations from a given class described by a quadratic constraint. These frequency inequalities are obtained by resolving the perturbation problem in terms of a properly posed (indefinite) quadratic infinite-horizon regulator problem, which, in turn, is solved via the Frequency Theorem developed in our earlier work \cite{Anikushin2020FreqDelay} (see also \cite{Anikushin2020FreqParab}). At the geometric level, these frequency inequalities ensure that a certain quadratic Lyapunov functional for a linear problem can also be used to study its nonautonomous perturbations. However, verifying the resulting frequency inequalities can only be done numerically. Applications to the exponential stability of twofold compound cocycles, which are related to the generalized Bendixson criterion for attractors \cite{LiMuldowney1995LowBounds}, are presented in our work \cite{AnikushinRomanov2023FreqConds} (joint with Romanov). It should be noted that this approach may substantially enhance the rigorous analytical results obtained from the comparison principle developed in \cite{Anikushin2022Semigroups} or from the findings in Section \ref{SEC: DimensionEstimatesViaLiouvDelayEqs} of the present paper. However, it requires numerical computations and is subject to limitations arising from the increasing complexity in the parameters $n$ and $m$.

From the above discussion, a natural problem arises that necessitates an in-depth study of the relationship among the trace formula, adapted metrics, and uniform Lyapunov exponents. This paper is dedicated to addressing this problem and is organized as follows.

In Section \ref{SEC: VariationalDescriptionOfLargUnifLyapExp}, we study continuous-time\footnote{For discrete-time cocycles, all related results are more elementary, and we only give some remarks in appropriate places.} linear cocycles\footnote{More precisely, what we call uniformly continuous linear cocycles.} in Banach spaces. For such cocycles, the evolution of length elements in a given metric (see the section for precise definitions) can be described via what we call the \textit{Growth Formula} --- a type of Newton--Leibniz formula that, in particular, implies the Liouville trace formula. This formula involves what we call \textit{infinitesimal growth exponents}. Maximizing these quantities provides an upper bound for the largest (uniform) Lyapunov exponent with respect to this metric. We further consider Lyapunov metrics as candidates for which the maximization procedure yields values arbitrarily close to their largest Lyapunov exponent, see Theorem \ref{TH: LengthsFamilyOfNorms}. Such metrics are not coercive in continuous time. However, for uniformly quasi-compact cocycles over a semiflow on a complete metric space that contracts under the semiflow to a compact attractor, we show that the largest Lyapunov exponent with respect to Lyapunov metrics bounds from above the largest Lyapunov exponent in the standard metric, see Theorem \ref{TH: LyapExpAdaptedMetricQuasiCom}. This provides a variational description of these quantities for noninvertible continuous-time cocycles via possibly noncoercive metrics, see Theorem \ref{TH: VariationalDescriptionLargestLyapunovExponent}.

In Section \ref{SEC: VariationalDescUnifLyapExpCompCoc}, we apply the above result to cocycles in Hilbert spaces. Here, the sum of the first $m$ Lyapunov exponents is given by the largest Lyapunov exponent of the \textit{$m$-fold multiplicative compound}, which is the multiplicative extension of the original cocycle to the $m$-fold exterior power of the Hilbert space. This provides a variational description of uniform Lyapunov exponents for cocycles in Hilbert spaces via adapted metrics defined on exterior powers, see Theorem \ref{TH: VariationalDescriptionUniformLyapunovExponents}. We further introduce the Lyapunov dimension and discuss several related computational results.

In Section \ref{SEC: ComputationOfInfExponents}, we develop a general theory for computing infinitesimal growth exponents of compound cocycles in associated metrics, using what we call \textit{curvature forms} and their \textit{mean representations}. Here, we establish the generalized Liouville trace formula (see Theorem \ref{TH: GeneralizedLiouvilleTraceFormula}) and provide a general upper estimate for singular values via trace numbers of mean representations (see Theorem \ref{TH: AveragedOmegaDEstimateViaTraceNumbers}). Finally, we connect the computation of trace numbers with the additive symmetrization of mean representations, which is further discussed in Appendix \ref{SEC: SymmetrizationOfOperators}.

In Section \ref{SEC: LiouvilleTraceFormulaAndSymProc}, we examine a specific case of the previous result for metrics within the conformal class of a constant metric (see Theorem \ref{TH: LiouvilleSymmetrizationSummary}), thereby deriving a general infinite-dimensional version of the Leonov method.

In Section \ref{SEC: DimensionEstimatesViaLiouvDelayEqs}, we apply this approach to cocycles arising from the linearization of a class of nonautonomous nonlinear delay equations in $\mathbb{R}^{n}$ posed in an appropriate Hilbert space. In Section \ref{SEC: LinearizationDelayEquations}, we discuss the well-posedness and linearization of such equations. Section \ref{SEC: DelayEqsSymmetrization} is devoted to the additive symmetrization of delay operators and provides a clear illustration of the symmetrization problem in infinite dimensions, where the outcome may significantly depend on the structure of the metrics. Next, we present upper estimates for the Lyapunov dimension over global attractors arising in the Mackey--Glass equations (see Section \ref{SEC: Mackey-GlassDimEst}) and in the periodically forced Suarez--Schopf delayed oscillator (see Section \ref{SEC: PertSuarezSchopfDimEstimate}). By numerically computing eigenvalues at appropriate equilibria, we demonstrate that these estimates are asymptotically sharp as the delay tends to infinity. In our joint paper with Romanov \cite{AnikushinRomanov2024EffEst}, we use these metrics to obtain similar results for the Nicholson blowflies model.

In Appendix \ref{APP: EVPNonCompact}, we provide an extension of the ergodic variational principle from \cite{Morris2013} in Theorem \ref{TH: ErgodicVarPrinciple} and discuss its important consequences.

In Appendix \ref{SEC: UpperSemiContLyapDim}, we establish the upper semicontinuity of the uniform growth exponent for subadditive families and deduce from it the corresponding property in the case of uniform volume growth characteristics (in particular, the Lyapunov dimension) for cocycles in Hilbert spaces, see Corollary \ref{COR: UpperSemiContUniformLyapExpAndLyapDim}. Although the result is elementary and folklore for some specialists, we include a discussion indicating that there are much more complicated proofs of its particular cases and even related misbeliefs presented in the literature.

In Appendix \ref{SEC: SymmetrizationOfOperators}, we discuss the trace numbers of operators in Hilbert spaces and their computation via additive symmetrization. We provide a simple answer to the problem of additive symmetrization in terms of the Friedrichs extensions in Theorem \ref{TH: AdditiveSymmetrizationAndFriedrichsExtension} and prove Theorem \ref{TH: ComputationSwedgeSpectralBound}, which relates the trace numbers of a general self-adjoint operator to its spectrum.

%% file: CocyclesAndGrowth.tex
\section{Infinitesimal growth exponents and Lyapunov metrics for cocycles in Banach spaces}
\label{SEC: VariationalDescriptionOfLargUnifLyapExp}

Let $\mathbb{T} \in \{ \mathbb{Z}_{+}, \mathbb{Z}, \mathbb{R}, \mathbb{R}_{+} \}$ be a time space\footnote{Here $\mathbb{R}_{+} = [0,+\infty)$ and $\mathbb{Z}_{+} = \mathbb{Z} \cap \mathbb{R}_{+}$.}, and let $\mathcal{Q}$ be a complete metric space. A family of mappings $\vartheta^{t}(\cdot) \colon \mathcal{Q} \to \mathcal{Q}$, where $t \in \mathbb{T}$, is called a \textit{dynamical system} on $\mathcal{Q}$ if it satisfies
\begin{description}[before=\let\makelabel\descriptionlabel]
	\item[\textbf{(DS1)}\refstepcounter{desccount}\label{DESC: DS1}] for all $q \in \mathcal{Q}$ and $t,s \in \mathbb{T}$, we have $\vartheta^{t+s}(q) = \vartheta^{t}( \vartheta^{s}(q))$ and $\vartheta^{0}(q) = q$;
	\item[\textbf{(DS2)}\refstepcounter{desccount}\label{DESC: DS2}] the mapping $\mathbb{T} \times \mathcal{Q} \ni (t,q) \mapsto \vartheta^{t}(q)$ is continuous.
\end{description}
For brevity, we use the notation $(\mathcal{Q},\vartheta)$ or simply $\vartheta$ to denote the dynamical system. In the case $\mathbb{T} = \mathbb{R}_{+}$ (resp. $\mathbb{T} = \mathbb{R}$) we call $\vartheta$ a \textit{semiflow} (resp. a \textit{flow}) on $\mathcal{Q}$.

Put $\mathbb{T}_{+} \coloneq \mathbb{T} \cap [0,+\infty)$. For a given Banach space $\mathbb{E}$, we call a family of mappings $\psi^{t}(q,\cdot) \colon \mathbb{E} \to \mathbb{E}$, where $t \in \mathbb{T}_{+}$ and $q \in \mathcal{Q}$, a \textit{cocycle} in $\mathbb{E}$ over $(\mathcal{Q},\vartheta)$ if 
\begin{description}[before=\let\makelabel\descriptionlabel]
	\item[\textbf{(CO1)}\refstepcounter{desccount}\label{DESC: CO1}] for all $v \in \mathbb{E}$, $q \in \mathcal{Q}$, and $t,s \in \mathbb{T}_{+}$, we have $\psi^{t+s}(q,v) = \psi^{t}(\vartheta^{s}(q),\psi^{s}(q,v))$ and $\psi^{0}(q,v) = v$;
	\item[\textbf{(CO2)}\refstepcounter{desccount}\label{DESC: CO2}] the mapping $\mathbb{T}_{+} \times \mathcal{Q} \times \mathbb{E} \ni (t,q,v) \mapsto \psi^{t}(q,v)$ is continuous.
\end{description}
For brevity, the cocycle will be denoted by $\psi$. If each mapping $\psi^{t}(q,\cdot)$ belongs to the space $\mathcal{L}(\mathbb{E})$ of bounded linear operators in $\mathbb{E}$, we say that the cocycle is \textit{linear}. In this paper, we are interested in linear cocycles arising after linearization of nonlinear cocycles (see Section \ref{SEC: LinearizationDelayEquations}). 

In what follows, linear cocycles will be denoted by $\Xi$. Moreover, if such $\Xi$ additionally satisfies\footnote{Here $\|\cdot\|_{\mathcal{L}(\mathbb{E})}$ denotes the operator norm in $\mathcal{L}(\mathbb{E})$ associated with the norm $\|\cdot\|_{\mathbb{E}}$ in $\mathbb{E}$.}
\begin{description}[before=\let\makelabel\descriptionlabel]
	\item[\textbf{(UC1)}\refstepcounter{desccount}\label{DESC: UC1}] for any $t \in \mathbb{T}_{+}$ the mapping $\mathcal{Q} \ni q \mapsto \Xi^{t}(q,\cdot) \in \mathcal{L}(\mathbb{E})$ is continuous in the operator norm; 
	\item[\textbf{(UC2)}\refstepcounter{desccount}\label{DESC: UC2}] the cocycle mappings are bounded uniformly in finite times:
	\begin{equation}
		\sup_{t \in \mathbb{T}_{+} \cap [0,1]} \sup_{q \in \mathcal{Q}}\| \Xi^{t}(q,\cdot) \|_{\mathcal{L}(\mathbb{E})} < \infty,
	\end{equation}
\end{description}
then $\Xi$ is called a \textit{uniformly continuous linear cocycle}. Clearly, for such cocycles, the condition \nameref{DESC: CO2} is equivalent to the fact that the operator $\Xi^{t}(q,\cdot)$ depends continuously on $(t,q)$ in the strong operator topology.

\begin{remark}
	All the results and definitions from the abstract part of our work (i.e., up to Section \ref{SEC: LiouvilleTraceFormulaAndSymProc}) can be more or less straightforwardly adapted for cocycles on Banach (or Hilbert) bundles. Here we are mainly interested in infinite-dimensional problems, for which such a setting is too redundant.
\end{remark}

Let us firstly establish the following lemma.
\begin{lemma}
	\label{LEM: LinearCocycleNormMappingRightLowerSemCon}
	Let $\Xi$ be a linear cocycle in $\mathbb{E}$ over a semiflow $(\mathcal{Q},\vartheta)$. Then for all $t \geq 0$ and $q \in \mathcal{Q}$, we have
	\begin{equation}
		\label{EQ: LinearCocycleRightLowerSemiCon}
		\liminf_{h \to 0+} \|\Xi^{t+h}(q,\cdot)\|_{\mathcal{L}(\mathbb{E})} \geq \|\Xi^{t}(q,\cdot)\|_{\mathcal{L}(\mathbb{E})}.
	\end{equation}
	In other words, the function
	\begin{equation}
		\label{EQ: LinearCocycleNormMapping}
		[0,+\infty) \ni t \mapsto \|\Xi^{t}(q,\cdot)\|_{\mathcal{L}(\mathbb{E})} \in [0,\infty)
	\end{equation}
	is right lower semicontinuous. In particular, it is Borel measurable.
\end{lemma}
\begin{proof}
	Let $\varepsilon>0$ be fixed. Then there exists $\xi \in \mathbb{E}$ with $\|\xi\|_{\mathbb{E}}=1$ such that $\|\Xi^{t}(q,\xi)\|_{\mathbb{E}} = \|\Xi^{t}(q,\cdot)\|_{\mathcal{L}(\mathbb{E})} + R_{\varepsilon}$, where $|R_{\varepsilon}|<\varepsilon$. From \nameref{DESC: CO1} we have
	\begin{equation}
		\begin{split}
			\|\Xi^{t+h}(q,\cdot)\|_{\mathcal{L}(\mathbb{E})} \geq \|\Xi^{t+h}(q,\xi)\|_{\mathbb{E}} = \|\Xi^{t}(\vartheta^{h}(q),\Xi^{h}(q,\xi))\|_{\mathbb{E}} =\\= \|\Xi^{t}(q,\xi)\|_{\mathbb{E}} + \widetilde{R}_{\varepsilon}(h) = \|\Xi^{t}(q,\cdot)\|_{\mathcal{L}(\mathbb{E})} + R_{\varepsilon} + \widetilde{R}_{\varepsilon}(h),
		\end{split}
	\end{equation}
	where $\widetilde{R}_{\varepsilon}(h) \to 0$ as $h \to 0+$ due to \nameref{DESC: CO2}. Taking it to the limit as $h \to 0+$ and then as $\varepsilon \to 0+$, we obtain \eqref{EQ: LinearCocycleRightLowerSemiCon}.
	
	Let $t_{0} \geq 0$ belong to the preimage of $(a,\infty)$ with $a \geq 0$ under the mapping \eqref{EQ: LinearCocycleNormMapping}. From \eqref{EQ: LinearCocycleRightLowerSemiCon} there exists $\varepsilon>0$ such that any $t \in [t_{0},\varepsilon)$ also belongs to the preimage. Then a standard argument shows that the preimage is given by a countable union of intervals. Thus, it is a Borel subset of $[0,+\infty)$. Moreover, if $\|\Xi^{t_{0}}(q,\cdot)\|_{\mathcal{L}(\mathbb{E})}=0$, then $\|\Xi^{t}(q,\cdot)\|_{\mathcal{L}(\mathbb{E})}=0$ for any $t \geq t_{0}$. This shows that the preimage of $0$ is Borel measurable. Taken together, this implies that the function \eqref{EQ: LinearCocycleNormMapping} is Borel measurable.
\end{proof}

Below, we will mainly deal with the most interesting case of continuous time $\mathbb{T} = \mathbb{R}_{+}$ and make remarks on the discrete case.

Let $\mathfrak{n}=\{\mathfrak{n}_{q}\}_{q \in \mathcal{Q}}$ be a family of nonzero seminorms in a Banach space $\mathbb{E}$, which, for brevity, will be called a \textit{semi-metric} over $\mathcal{Q}$. If each $\mathfrak{n}_{q}$ is a norm in $\mathbb{E}$, we say that $\mathfrak{n}$ is a \textit{metric} over $\mathcal{Q}$

We say that $\mathfrak{n}$ is \textit{uniformly weaker} than $\|\cdot\|_{\mathbb{E}}$ if for some $M^{+}_{\mathfrak{n}}>0$ we have
\begin{equation}
	\label{EQ: UniformWeaknessFamNorm}
	\sup_{q \in \mathcal{Q}}\mathfrak{n}_{q}(\xi) \leq M^{+}_{\mathfrak{n}} \cdot \| \xi \|_{\mathbb{E}} \qquad \text{for any} \quad \xi \in \mathbb{E}.
\end{equation}
Analogously, we say that $\mathfrak{n}$ is \textit{uniformly coercive} with respect to $\|\cdot\|_{\mathbb{E}}$ if for some $M^{-}_{\mathfrak{n}}>0$ we have
\begin{equation}
	\label{EQ: UniformCoerciveFamily}
	\inf_{q \in \mathcal{Q}}\mathfrak{n}_{q}(\xi) \geq M^{-}_{\mathfrak{n}} \cdot \| \xi \|_{\mathbb{E}} \qquad \text{for any} \quad \xi \in \mathbb{E}.
\end{equation}
If both \eqref{EQ: UniformWeaknessFamNorm} and \eqref{EQ: UniformCoerciveFamily} are satisfied, we say that $\mathfrak{n}$ is \textit{equivalent} to $\|\cdot\|_{\mathbb{E}}$. 

If the mapping $\mathcal{Q} \ni q \mapsto \mathfrak{n}_{q}$ is continuous in the sense of uniform topology on seminorms\footnote{It is induced by the distance $d( \mathfrak{m}_{1}, \mathfrak{m}_{2}) \coloneq \sup_{\|\xi\|_{\mathbb{E}}=1}|\mathfrak{m}_{1}(\xi)-\mathfrak{m}_{2}(\xi)|$, where $\mathfrak{m}_{1}$ and $\mathfrak{m}_{2}$ are seminorms in $\mathbb{E}$.}, we say that $\mathfrak{n}$ is \textit{continuous}.

Let $\Xi$ be a linear cocycle in $\mathbb{E}$ over a semiflow $(\mathcal{Q},\vartheta)$. We say that $\Xi$ \textit{admits infinitesimal growth exponents} with respect to $\mathfrak{n}$ if
\begin{description}[before=\let\makelabel\descriptionlabel]
	\item[\textbf{(GE1)}\label{DESC: GE1}] there exists a family of dense in $\mathbb{E}$ subspaces $\mathcal{N}_{\mathfrak{n}}(q)$, where $q \in \mathcal{Q}$, such that $\Xi^{t}(q,\mathcal{N}_{\mathfrak{n}}(q)) \subset \mathcal{N}_{\mathfrak{n}}(\vartheta^{t}(q))$ for all $t \geq 0$ and $q \in \mathcal{Q}$;
	\item[\textbf{(GE2)}\label{DESC: GE2}] for all $q \in \mathcal{Q}$, $\xi \in \mathcal{N}_{\mathfrak{n}}(q)$, and $T>0$, the mapping
	\begin{equation}
		\label{EQ: TheSeminormMappingOverCocycleTrajectory}
		[0,T] \ni t \mapsto \mathfrak{n}_{\vartheta^{t}(q)}( \Xi^{t}(q,\xi) )
	\end{equation}
    is absolutely continuous, i.e., it belongs to the Sobolev space $W^{1,1}(0,T)$;
    \item[\textbf{(GE3)}\label{DESC: GE3}] for all $q \in \mathcal{Q}$ and $\xi \in \mathcal{N}_{\mathfrak{n}}(q)$, there exists the right derivative
    \begin{equation}
    	\label{EQ: NormRightDerivativeLimitDPlusDef}
    	D^{+}_{\mathfrak{n}}(q;\xi) \coloneq \lim_{t \to 0+} \frac{ \mathfrak{n}_{\vartheta^{t}(q)}(\Xi^{t}(q,\xi)) - \mathfrak{n}_{q}(\xi)}{t}.
    \end{equation}
\end{description}

Under the above properties, for all $q \in \mathcal{Q}$ and $\xi \in \mathcal{N}_{\mathfrak{n}}(q)$ such that $\mathfrak{n}_{q}(\xi) \not= 0$, we introduce the \textit{infinitesimal growth exponent} of $\Xi$ at $\xi$ over $q$ with respect to $\mathfrak{n}$ by
\begin{equation}
	\label{EQ: InfinitesimalGrowthExponentSeminormDef}
	\alpha_{\mathfrak{n}}(q;\xi)\coloneq \frac{D^{+}_{\mathfrak{n}}(q;\xi)}{\mathfrak{n}_{q}(\xi)}.
\end{equation}

We immediately have the following theorem.
\begin{theorem}[Growth Formula]
	Suppose a linear cocycle $\Xi$ over a semiflow $(\mathcal{Q},\vartheta)$ admits infinitesimal growth exponents with respect to a semi-metric $\mathfrak{n}=\{\mathfrak{n}_{q}\}_{q \in \mathcal{Q}}$, i.e., \nameref{DESC: GE1}, \nameref{DESC: GE2} and \nameref{DESC: GE3} are satisfied. Then for all $q \in \mathcal{Q}$ and $\xi \in \mathcal{N}_{\mathfrak{n}}(q)$, we have
	\begin{equation}
		\label{EQ: SeminormsGrowthViaInfExponentsEq}
		\mathfrak{n}_{\vartheta^{t}(q)}( \Xi^{t}(q,\xi)) = \mathfrak{n}_{q}(\xi) \exp\left( \int_{0}^{t}\alpha_{\mathfrak{n}}(\vartheta^{s}(q);\Xi^{s}(q,\xi))ds \right)
	\end{equation}
    satisfied for all $t \in [0,T]$ and $T>0$ such that $\mathfrak{n}_{\vartheta^{t}(q)}(\Xi^{t}(q,\xi)) \not= 0 $ for any $t \in [0,T]$.
\end{theorem}
\begin{proof}
	From \nameref{DESC: GE2} and \nameref{DESC: GE3}, for almost all $t \in [0,T]$ we have
	\begin{equation}
		\frac{d}{dt}\left( \mathfrak{n}_{\vartheta^{t}(q)}( \Xi^{t}(q,\xi)) \right) = D^{+}_{\mathfrak{n}}(\vartheta^{t}(q);\Xi^{t}(q,\xi)).
	\end{equation}
   From this, \nameref{DESC: GE2}, and \eqref{EQ: InfinitesimalGrowthExponentSeminormDef}, for any $t \in [0,T]$ the Newton-Leibniz formula
    \begin{equation}
    	\label{EQ: InfGrowthExpsLogFormula}
    	\begin{split}
    		\ln\mathfrak{n}_{\vartheta^{t}(q)}(\Xi^{t}(q,\xi)) - \ln\mathfrak{n}_{q}(\xi) &= \int_{0}^{t} \frac{D^{+}_{\mathfrak{n}}(\vartheta^{s}(q);\Xi^{s}(q,\xi))}{\mathfrak{n}_{\vartheta^{s}(q)}(\Xi^{s}(q,\xi))}ds = \\ &= \int_{0}^{t}\alpha_{\mathfrak{n}}(\vartheta^{s}(q);\Xi^{s}(q,\xi))ds.
    	\end{split}
    \end{equation}
    is valid. Applying the exponential function yields \eqref{EQ: SeminormsGrowthViaInfExponentsEq}.
\end{proof}

\begin{remark}
	\label{REM: InfinitesimalGrowthExponentsDiscrete}
	For the case of discrete time $\mathbb{T} \in \{\mathbb{Z}_{+}, \mathbb{Z}\}$, for all $q \in \mathcal{Q}$ and $\xi \in \mathbb{E}$ such that $\mathfrak{n}_{q}(\xi) \not=0$, the growth exponent $\alpha_{\mathfrak{n}}(q;\xi)$ should be defined by
	\begin{equation}
		\alpha _{\mathfrak{n}}(q,\xi) \coloneq \ln \frac{\mathfrak{n}_{\vartheta^{1}(q)}(\Xi^{1}(q,\xi))}{\mathfrak{n}_{q}(\xi)}
	\end{equation}
    with the convention that $\ln 0 \coloneq -\infty$. Then we have an analog of \eqref{EQ: SeminormsGrowthViaInfExponentsEq} satisfied for appropriate $t \in \mathbb{Z}_{+}$ with the integral exchanged with the sum $\sum_{k=0}^{t-1}\alpha _{\mathfrak{n}}(\vartheta^{k}(q),\Xi^{k}(q,\xi))$, where the convention $e^{-\infty}\coloneq0$ is used.
\end{remark}

In what follows, we suppose that $\Xi$ admits infinitesimal growth exponents with respect to a metric $\mathfrak{n}$ over $\mathcal{Q}$, i.e., each $\mathfrak{n}_{q}$ is a norm in $\mathbb{E}$.

Let us consider the \textit{largest uniform Lyapunov exponent} $\lambda_{\mathfrak{n}}(\Xi)$ of $\Xi$ with respect to $\mathfrak{n}$:
\begin{equation}
	\label{EQ: LargetLyapunovExponentWRTMetric}
	\lambda_{\mathfrak{n}}(\Xi) \coloneq \lim_{t \to +\infty} \frac{\ln \sup_{q,\xi}\mathfrak{n}_{\vartheta^{t}(q)}(\Xi^{t}(q,\xi))}{t},
\end{equation}
where the supremum is taken over all $q \in \mathcal{Q}$ and $\xi \in \mathbb{E}$ with $\mathfrak{n}_{q}(\xi)=1$. Since $\mathfrak{n}$ is not uniformly equivalent to $\|\cdot\|_{\mathbb{E}}$, there may be problems related to the existence of the limit. We define $\bar{\lambda}_{\mathfrak{n}}(\Xi)$ (resp. $\underbar{$\lambda$}_{\mathfrak{n}}(\Xi)$) to be the limit superior (resp. inferior) of the corresponding expression. We say that $\mathfrak{n}$ is \textit{admissible} for $\Xi$ if the limit (finite or infinite) indeed exists, i.e., $\bar{\lambda}_{\mathfrak{n}}(\Xi) = \underbar{$\lambda$}_{\mathfrak{n}}(\Xi)$.

If the metric $\mathfrak{n}$ is standard, that is, $\mathfrak{n}_{q}(\cdot) = \| \cdot \|_{\mathbb{E}}$ for any $q \in \mathcal{Q}$, we use the notation $\lambda_{1}(\Xi)$ instead of $\lambda_{\mathfrak{n}}(\Xi)$ and call $\lambda_{1}(\Xi)$ the \textit{uniform Lyapunov exponent} of $\Xi$.

In view of \eqref{EQ: SeminormsGrowthViaInfExponentsEq}, it is natural to maximize each $\alpha_{\mathfrak{n}}(q;\xi)$ over all $q \in \mathcal{Q}$ and $\xi \in \mathcal{N}_{\mathfrak{n}}(q)$, thereby determining the value
\begin{equation}
	\label{EQ: AlphaPlusInfExpMaximizationDefinition}
	\alpha^{+}_{\mathfrak{n}}(\Xi) \coloneq \sup_{q \in \mathcal{Q}} \sup_{\xi \in \mathcal{N}_{\mathfrak{n}}(q)} \alpha_{\mathfrak{n}}(q;\xi).
\end{equation}
Such a process is called the \textit{maximization procedure} of infinitesimal growth exponents $\alpha_{\mathfrak{n}}(q;\xi)$, and it results in the \textit{maximized exponent} $\alpha^{+}_{\mathfrak{n}}(\Xi)$. Thanks to the density of $\mathcal{N}_{\mathfrak{n}}(q)$ from \nameref{DESC: GE1}, \eqref{EQ: SeminormsGrowthViaInfExponentsEq} immediately yields $\bar{\lambda}_{\mathfrak{n}}(\Xi) \leq \alpha^{+}_{\mathfrak{n}}(\Xi)$. We note that even in seemingly nice cases, $\alpha^{+}_{\mathfrak{n}}(\Xi)$ may be infinite (see Section \ref{SEC: DelayEqsSymmetrization}).

\begin{remark}
	\label{REM: AveragedExponents}
	Let $\alpha^{+}_{\mathfrak{n}}(q) = \alpha^{+}_{\mathfrak{n}}(q;\Xi)$ be given by the inner supremum from \eqref{EQ: AlphaPlusInfExpMaximizationDefinition}. We call $\alpha^{+}_{\mathfrak{n}}(q)$ the \textit{maximized exponent} over $q$. When we are working with particular metrics in applications, the estimate via the fiber maxima $\alpha^{+}_{\mathfrak{n}}(q)$:
	\begin{equation}
		\mathfrak{n}_{\vartheta^{t}(q)}(\Xi^{t}(q,\xi)) \leq \mathfrak{n}_{q}(\xi) \exp\left(\int_{0}^{t}\alpha^{+}_{\mathfrak{n}}(\vartheta^{s}(q))ds\right),
	\end{equation}
	when it makes sense (see Remark \ref{REM: LozinskiiNorms}), may be more convenient, as the \textit{averaged exponent}
	\begin{equation}
		\bar{\alpha}_{\mathfrak{n}}(\Xi) \coloneq \lim_{t \to +\infty}\frac{1}{t}\sup_{q \in \mathcal{Q}} \int_{0}^{t}\alpha^{+}_{\mathfrak{n}}(\vartheta^{s}(q))ds
	\end{equation}
	provides sharper estimates for $\lambda_{\mathfrak{n}}(\Xi)$, see \cite{LiMuldowney1996SIAMGlobStab} for a practical example. We say that $\bar{\alpha}_{\mathfrak{n}}(\Xi)$ is obtained via the \textit{averaging procedure} of infinitesimal growth exponents $\alpha_{\mathfrak{n}}(q;\xi)$.
\end{remark}

\begin{remark}
	\label{REM: LozinskiiNorms}
	It can be shown that $\Xi$ admits infinitesimal growth exponents with respect to the standard metric $\mathfrak{n}$, provided that the set
	\begin{equation}
		\label{EQ: IGEStandradMetricDomains}
		\mathcal{N}_{\mathfrak{n}}(q) = \{ \xi \in \mathbb{E} \ | \ [0,\infty) \ni t \mapsto \Xi^{t}(q,\xi) \in \mathbb{E} \text{ is $C^{1}$-differentiable} \}
	\end{equation}
	is dense in $\mathbb{E}$. Then for such a metric, the inequality $\lambda_{1}(\Xi) \leq \alpha^{+}_{\mathfrak{n}}(\Xi)$ can be called the \textit{Lozinskii estimate}. In fact, the Lozinskii estimate \cite{Lozinskii1958}, which immediately follows from \eqref{EQ: SeminormsGrowthViaInfExponentsEq}, is
	\begin{equation}
		\|\Xi^{t}(q,\xi)\|_{\mathbb{E}} \leq \|\xi\|_{\mathbb{E}} \exp\left( \int_{0}^{t}\alpha^{+}_{\mathfrak{n}}(\vartheta^{s}(q))ds \right),
	\end{equation}
    where $\alpha^{+}_{\mathfrak{n}}(q)$ is the fiber maximum over $q$ considered in Remark \ref{REM: AveragedExponents}. Note that in our abstract context, it is not even a priori known whether $s \mapsto \alpha^{+}_{\mathfrak{n}}(\vartheta^{s}(q))$ is measurable, not to mention its integrability.
    
    Consider the generator $A(q)$ of $\Xi$ over $q \in \mathcal{Q}$, which is an operator with domain $\mathcal{N}_{\mathfrak{n}}(q)$ as in \eqref{EQ: IGEStandradMetricDomains} given by
    \begin{equation}
    	\label{EQ: CocycleDerivativeLogarNorm}
    	A(q)\xi \coloneq \lim_{t \to 0+} \frac{\Xi^{t}(q,\xi) - \xi}{t} \qquad \text{for any} \quad \xi \in \mathcal{N}_{\mathfrak{n}}(q).
    \end{equation}
    Then the quantity $\alpha^{+}_{\mathfrak{n}}(q)$ should be called the \textit{Lozinskii} (or \textit{logarithmic}) norm of $A(q)$. 
    
    Note that $\alpha^{+}_{\mathfrak{n}}(q)$ is given by the supremum of pointwise derivatives. Usually, logarithmic norms are defined by taking the supremum first and then differentiating, see \cite{Lozinskii1958, KuzMokKuzKud2020, Davydovetal2024}. First, such a definition is inappropriate in infinite dimensions where $A(q)$ is unbounded. Second, following this definition, one omits the exact Growth Formula \eqref{EQ: SeminormsGrowthViaInfExponentsEq}. For bounded operators, both definitions coincide (this can be proven analogously to Theorem \ref{TH: RegularityOfTraceExponentAndCoincidence}). \qed
\end{remark}

We are interested in approximations of the largest uniform Lyapunov exponent $\lambda_{1}(\Xi)$ in the standard metric, via the quantity $\alpha^{+}_{\mathfrak{n}}(\Xi)$ computed in an adapted metric $\mathfrak{n}$. 

On the abstract level, natural candidates for adapted metrics are Lyapunov metrics. In this direction, we start with the following result.
\begin{theorem}
	\label{TH: LengthsFamilyOfNorms}
	Let $\Xi$ be a uniformly continuous linear cocycle over a semiflow $(\mathcal{Q},\vartheta)$. For any $T \in (0,\infty]$, $p \geq 1$, and $\nu > \lambda_{1}(\Xi)$, consider the Lyapunov metric $\mathfrak{n} = \{\mathfrak{n}_{q}(\cdot)\}_{q \in \mathcal{Q}}$, where the norms $\mathfrak{n}_{q}(\cdot)$ are given by
	\begin{equation}
		\label{EQ: FamilyOfNormsDef}
		\mathfrak{n}_{q}(\xi) \coloneq \left(\int_{0}^{T}\|e^{-\nu t}\Xi^{t}(q,\xi)\|^{p}_{\mathbb{E}}dt\right)^{1/p} \qquad \text{for all} \quad q \in \mathcal{Q} \quad \text{and} \quad \xi \in \mathbb{E}.
	\end{equation} 
    Then $\Xi$ admits infinitesimal growth exponents with respect to $\mathfrak{n}$, which for any $q \in \mathcal{Q}$ and any nonzero $\xi \in \mathcal{N}_{\mathfrak{n}}(q) = \mathbb{E}$ are given by
	\begin{equation}
		\label{EQ: LyapunovMetricApproxAdaptedGrowthExponentFormula}
		\alpha_{\mathfrak{n}}(q;\xi) = \nu + \frac{\|e^{-\nu T}\Xi^{T}(q,\xi)\|^{p}_{\mathbb{E}} - \|\xi\|^{p}_{\mathbb{E}}}{p \cdot \mathfrak{n}^{p}_{q}(\xi)}.
	\end{equation}
    In particular, $\mathfrak{n}$ is continuous and uniformly weaker than $\|\cdot\|_{\mathbb{E}}$, i.e., it satisfies \eqref{EQ: UniformWeaknessFamNorm} with some $M^{+}_{\mathfrak{n}} > 0$. Moreover, there exists $T_{0}=T_{0}(\nu)>0$ such that $\alpha^{+}_{\mathfrak{n}}(\Xi) < \nu$ for any $T \geq T_{0}$, in which case $\mathfrak{n}$ is admissible.
\end{theorem}
\begin{proof}
	Omitting the trivial case $\lambda_{1}(\Xi) = -\infty$ and using \nameref{DESC: UC2} for $\Xi$, for every $\varepsilon>0$ we get a constant $M_{\varepsilon}>0$ such that for all $\xi \in \mathbb{E}$, $q \in \mathcal{Q}$, and $t \geq 0$, the estimate
	\begin{equation}
		\label{EQ: LargLyapExpTheoremExpEstimateNu}
		\|e^{-\nu t}\Xi^{t}(q,\xi)\|_{\mathbb{E}} \leq M_{\varepsilon} e^{(\lambda_{1}(\Xi) + \varepsilon - \nu)t)} \|\xi\|_{\mathbb{E}}
	\end{equation}
    is valid. Since $\nu > \lambda_{1}(\Xi)$, we may take $\varepsilon>0$ such that $\lambda_{1}(\Xi) + \varepsilon - \nu < 0$. This implies that $\mathfrak{n}_{q}(\xi)$ is well-defined by \eqref{EQ: FamilyOfNormsDef}, and there exists a constant $M_{\mathfrak{n}}^{+}>0$ such that $\mathfrak{n}_{q}(\xi) \leq M_{\mathfrak{n}}^{+} \|\xi\|_{\mathbb{E}}$ for all $q \in \mathcal{Q}$ and $\xi \in \mathbb{E}$. Moreover, the mapping $\mathcal{Q} \ni q \mapsto \mathfrak{n}_{q}$ is continuous by the Dominated Convergence Theorem, which is applicable due to Lemma \ref{LEM: LinearCocycleNormMappingRightLowerSemCon} (measurability), \nameref{DESC: UC1} (pointwise convergence), and the estimate \eqref{EQ: LargLyapExpTheoremExpEstimateNu} (integrable majorant).
	
	From \eqref{EQ: FamilyOfNormsDef}, for any $t \geq 0$, we have
	\begin{equation}
		\label{EQ: LargLyapExpInLyapMetricIdentity1}
		\begin{split}
			e^{-p\nu t}\mathfrak{n}^{p}_{\vartheta^{t}(q)}(\Xi^{t}(q,\xi)) - \mathfrak{n}^{p}_{q}(\xi) = \\ = \int_{T}^{T+t}\|e^{-\nu s}\Xi^{s}(q,\xi)\|^{p}_{\mathbb{E}}ds - \int_{0}^{t}\|e^{-\nu s}\Xi^{s}(q,\xi)\|^{p}_{\mathbb{E}}ds.
		\end{split}
	\end{equation}
    From this, it is clear that $\mathfrak{n}$ satisfies \nameref{DESC: GE1} with $\mathcal{N}_{\mathfrak{n}}(q) = \mathbb{E}$ and \nameref{DESC: GE2} with the function $t \mapsto \mathfrak{n}_{\vartheta^{t}(q)}(\Xi^{t}(q,\xi))$ being $C^{1}$-differentiable. Moreover, dividing \eqref{EQ: LargLyapExpInLyapMetricIdentity1} by $t>0$ and taking it to the limit as $t \to 0+$, we get \nameref{DESC: GE3} satisfied such that \eqref{EQ: LyapunovMetricApproxAdaptedGrowthExponentFormula} holds.
    
    Let $T_{0} > 0 $ be such that in terms of \eqref{EQ: LargLyapExpTheoremExpEstimateNu} we have $M_{\varepsilon} e^{(\lambda_{1}(\Xi)+\varepsilon-\nu)T} \leq 1$ for any $T \geq T_{0}$. From \eqref{EQ: LyapunovMetricApproxAdaptedGrowthExponentFormula}, it is clear that for such $T$ we have $\alpha^{+}_{\mathfrak{n}}(\Xi) < \nu$. In particular, $\lambda_{\mathfrak{n}}(\Xi)$ is well-defined by the limit in \eqref{EQ: LargetLyapunovExponentWRTMetric}, where the fraction is bounded from above by $\nu$, thanks to \eqref{EQ: SeminormsGrowthViaInfExponentsEq}, and therefore we may appeal to Lemma \ref{LEM: FeketeLemma}. Consequently, $\mathfrak{n}$ is admissible.
\end{proof}

Let us make some remarks.
\begin{remark}
	\label{REM: InvertibleCocycleCoerciveMetric}
	Suppose that for some $\varepsilon>0$ we have a kind of uniform local invertibility:
	\begin{equation}
		\inf_{q \in \mathcal{Q}}\inf_{t \in [0,\varepsilon]}\| \Xi^{t}(q, \cdot) \|_{\mathcal{L}(\mathbb{E})} > 0. 
	\end{equation}
    Then, it is clear that the Lyapunov metric $\mathfrak{n}$ from Theorem \ref{TH: LengthsFamilyOfNorms} is uniformly coercive with respect to $\|\cdot\|_{\mathbb{E}}$. In particular, $\mathfrak{n}$ is equivalent to $\|\cdot\|_{\mathbb{E}}$, and, consequently, we have $\lambda_{\mathfrak{n}}(\Xi) = \lambda_{1}(\Xi)$. Therefore for all sufficiently large $T$ in \eqref{EQ: FamilyOfNormsDef}, we have
    \begin{equation}
    	\label{EQ: InequalityLarLyapExpsMetricAndClassicAlphaPlus}
    	\lambda_{1}(\Xi) = \lambda_{\mathfrak{n}}(\Xi) \leq \alpha^{+}_{\mathfrak{n}}(\Xi) \leq \lambda_{1}(\Xi) + (\nu-\lambda_{1}(\Xi)).
    \end{equation}
    This shows that for such systems, the maximization procedure yields the quantities $\alpha^{+}_{\mathfrak{n}}(\Xi)$ that can be made arbitrarily close to $\lambda_{1}(\Xi)$ by a proper choice of a uniformly equivalent metric $\mathfrak{n}$. Then \eqref{EQ: SeminormsGrowthViaInfExponentsEq} yields that 
\end{remark}

\begin{remark}
	\label{REM: HilbertSpaceAdaptedMetric}
	If $\mathbb{E} = \mathbb{H}$ is a Hilbert space with the norm $\|\cdot\|_{\mathbb{E}} = |\cdot|_{\mathbb{H}}$ induced by an inner product, the norms $\mathfrak{n}_{q}$ given by \eqref{EQ: FamilyOfNormsDef} with $p=2$ are also induced by inner products. In fact, we can relax the condition $\mathbb{E} = \mathbb{H}$ by requiring that $\mathbb{E} \subset \mathbb{H}$ and $\Xi$ is in a sense smoothing in $\mathbb{E}$ with respect to $|\cdot|_{\mathbb{H}}$, as in \eqref{EQ: SmoothingDelayCocycle}.
\end{remark}

\begin{remark}
	\label{REM: SmoothenOfAdaptedLyapMetrics}
	If it makes sense to speak about a kind of smooth dependence of the cocycle mappings $\Xi^{t}(q,\cdot)$ on $q \in \mathcal{Q}$, then the norms $\mathfrak{n}_{q}$ from \eqref{EQ: FamilyOfNormsDef} with finite $T$ may also inherit this smoothness.
\end{remark}

\begin{remark}
	\label{REM: LyapunovMetricsForLyapExpDiscreteCase}
	For the case of discrete time $\mathbb{T} \in \{ \mathbb{Z}_{+}, \mathbb{Z} \}$, an analog of \eqref{EQ: FamilyOfNormsDef} is given by (here $T \in \mathbb{Z}_{+}$ and $T \geq 1$)
	\begin{equation}
		\label{EQ: LyapunovMetricForLyapExpDiscreteCase}
		\mathfrak{n}_{q}(\xi) \coloneq \left(\sum_{k=0}^{T-1} \| e^{-\nu k} \Xi^{k}(q,\xi) \|^{p}_{\mathbb{E}}\right)^{1/p} \qquad \text{for any} \quad \xi \in \mathbb{E}.
	\end{equation}
    Clearly, $\mathfrak{n}$ is equivalent to $\|\cdot\|_{\mathbb{E}}$ due to the presence of the term in \eqref{EQ: LyapunovMetricForLyapExpDiscreteCase} corresponding to $k=0$. Thus, in the discrete-time case, we automatically have $\lambda_{\mathfrak{n}}(\Xi) = \lambda_{1}(\Xi)$ for such Lyapunov metrics $\mathfrak{n}$. One can also show that $\alpha^{+}_{\mathfrak{n}}(\Xi) < \nu$ for any sufficiently large $T$.
\end{remark}

Since the Lyapunov metric $\mathfrak{n}$ from \eqref{EQ: FamilyOfNormsDef} may not be uniformly coercive with respect to $\|\cdot\|_{\mathbb{E}}$ in infinite dimensions due to the presence of smoothing properties, it is worth studying whether the inequality $\lambda_{1}(\Xi) \leq \lambda_{\mathfrak{n}}(\Xi)$ can be established without the assumption of local invertibility as in Remark \ref{REM: InvertibleCocycleCoerciveMetric}. This, of course, should be concerned with the definition \eqref{EQ: FamilyOfNormsDef}, linking both $\mathfrak{n}$ and $\|\cdot \|_{\mathbb{E}}$ with the dynamics, but it seems that we also have to assume the uniform quasi-compactness of $\Xi$, which we are going to introduce.

Let $\alpha_{K}(\cdot)$ be the Kuratowski measure of noncompactness in $\mathbb{E}$, i.e., for a given bounded subset $\mathcal{B}$ of $\mathbb{E}$, the value $\alpha_{K}(\mathcal{B})$ is given by the infimum over all $r>0$ such that $\mathcal{B}$ can be covered by finitely many balls of radius $r$. We consider the \textit{uniform compactness exponent} $\alpha_{K}(\Xi)$ of $\Xi$ defined by
\begin{equation}
	\label{EQ: KuratowskiExponentOfCocycle}
	\alpha_{K}(\Xi) \coloneq \lim_{t \to +\infty}\frac{\ln \sup_{q \in \mathcal{Q}} \alpha_{K}(\Xi^{t}(q,\mathcal{B}_{1}(0))) }{t},
\end{equation}
where $\mathcal{B}_{1}(0)$ is the ball of radius $1$ in $\mathbb{E}$ centered at $0$. From the standard properties of $\alpha_{K}$, we obtain that the limit exists (see Lemma \ref{LEM: FeketeLemma}) and $\alpha_{K}(\Xi) \leq \lambda_{1}(\Xi)$. If the inequality is strict or both values are equal to $-\infty$, the cocycle $\Xi$ is called \textit{uniformly quasi-compact}. See Remark \ref{REM: AsymCompFiniteLD} for the interpretation of the condition $\alpha_{K}(\Xi) < 0$ often encountered in practice.

We have the following theorem.
\begin{theorem}
	\label{TH: LyapExpAdaptedMetricQuasiCom}
	Let $\Xi$ be a uniformly continuous linear cocycle in $\mathbb{E}$ over a semiflow $(\mathcal{Q},\vartheta)$ on a complete metric space $\mathcal{Q}$. Suppose that $\Xi$ is uniformly quasi-compact and there exists a $2^{\mathcal{Q}}$-attractor for $(\mathcal{Q},\vartheta)$ as in \eqref{EQ: SubsetA0AttractsQ}. Then, the Lyapunov metric $\mathfrak{n}$ given by \eqref{EQ: FamilyOfNormsDef} for $T \geq 1$ satisfies $\underbar{$\lambda$}_{\mathfrak{n}}(\Xi) \geq \lambda_{1}(\Xi)$.
\end{theorem}
\begin{proof}
	To explain the idea of the proof, let us suppose first that there exist $q_{0} \in \mathcal{Q}$ and $\xi_{0} \in \mathbb{E}$ with $\|\xi_{0}\|_{\mathbb{E}} = 1$ such that
	\begin{equation}
		\label{EQ: ProofLELyapMetric1}
		\lambda_{1}(\Xi) = \lim_{t \to +\infty} \frac{\ln \| \Xi^{t}(q_{0},\xi_{0})\|_{\mathbb{E}}}{t}.
	\end{equation}
    Then from \eqref{EQ: FamilyOfNormsDef}, the Mean Value Theorem, and \nameref{DESC: UC2}, we get a constant $C>0$ such that for any $t \geq 0$ we have
    \begin{equation}
    	\label{EQ: LyapMetricULEEstimate}
    	\begin{split}
    		\mathfrak{n}_{\vartheta^{t}(q_{0})}(\Xi^{t}(q_{0},\xi_{0})) &= e^{\nu t} \left(\int_{t}^{T+t}\|e^{-\nu s} \Xi^{s}(q_{0},\xi_{0}) \|^{p}_{\mathbb{E}}ds\right)^{1/p} \geq\\ \geq e^{\nu t} \left(\int_{t}^{t+1}\|e^{-\nu s} \Xi^{s}(q_{0},\xi_{0}) \|^{p}_{\mathbb{E}}ds\right)^{1/p} &= e^{\nu t} e^{-\nu s(t)} \| \Xi^{s(t)}(q_{0},\xi_{0}) \|_{\mathbb{E}} \geq \\ &\geq C \| \Xi^{t+1}(q_{0},\xi_{0}) \|_{\mathbb{E}}
    	\end{split}
    \end{equation}
    for some $s(t) \in [t,t+1]$. Dividing both sides of \eqref{EQ: LyapMetricULEEstimate} by $\mathfrak{n}_{q_{0}}(\xi_{0})$, then passing to the supremum on the left-hand side as given below, and finally taking the natural logarithms and dividing by $t > 0$, we obtain
    \begin{equation}
    	\frac{\ln \sup_{q, \xi}\mathfrak{n}_{\vartheta^{t}(q)}(\Xi^{t}(q,\xi))}{t} \geq \frac{\ln C - \ln \mathfrak{n}_{q_{0}}(\xi_{0})}{t} + 
    	\frac{t+1}{t} \cdot \frac{\ln\| \Xi^{t+1}(q_{0},\xi_{0})\|_{\mathbb{E}}}{t+1},
    \end{equation}
     where the supremum is taken over all $q \in \mathcal{Q}$ and $\xi \in \mathbb{E}$ such that $\mathfrak{n}_{q}(\xi) = 1$. Passing to the limit inferior in the above inequality as $t \to +\infty$ yields $\underbar{$\lambda$}_{\mathfrak{n}}(\Xi) \geq \lambda_{1}(\Xi)$.
     
     Now we are going to show the existence of such $q_{0}$ and $\xi_{0}$. If $\lambda_{1}(\Xi) = -\infty$, then any $q_{0}$ and nonzero $\xi_{0}$ satisfy \eqref{EQ: ProofLELyapMetric1}. So, we may assume that $\lambda_{1}(\Xi) > -\infty$.
     
     First, by the Concentration Principle (see Corollary \ref{COR: DVolumesUniformAsDVolumesOverMeasure}), which for $\lambda_{1}(\Xi)$ (i.e., $d=1$) does not require the Hilbert space structure, we obtain an ergodic invariant probability measure $\mu$ on the $2^{\mathcal{Q}}$-attractor $\mathcal{A}$ such that the relations
     \begin{equation}
     	\lambda_{1}(\Xi) = \lim_{t \to +\infty}\frac{\ln \sup_{q \in \mathcal{Q}}\| \Xi^{t}(q,\cdot) \|_{\mathcal{L}(\mathbb{E})}}{t} = \lim_{t \to +\infty}\frac{\ln \|\Xi^{t}(q,\cdot)\|_{\mathcal{L}(\mathbb{E})}}{t},
     \end{equation}
     hold for $\mu$-almost every point $q \in \mathcal{A}$.
     
     By applying the Multiplicative Ergodic Theorem (see, e.g., \cite[Theorem 2.2.2]{Thieullen1992}) to the measure $\mu$, we obtain
     \begin{equation}
     	\lambda_{1}(\Xi) = \lim_{t \to +\infty}\frac{\ln \|\Xi^{t}(q,\cdot)\|_{\mathcal{L}(\mathbb{E})}}{t} = \lim_{t \to +\infty}\frac{\ln \|\Xi^{t}(q,\xi)\|_{\mathbb{E}}}{t}
     \end{equation}
     for $\mu$-almost all $q \in \mathcal{A}$ and any $\xi$ taken from the complement in $\mathbb{E}$ to the second subspace from the Oseledets filtration\footnote{This is the space $F_{2}$ in terms of \cite{Thieullen1992}.} over $q$, which is nontrivial because $\alpha_{K}(\Xi) < \lambda_{1}(\Xi)$. It is sufficient to take as $q_{0}$ and $\xi_{0}$ any such points.
\end{proof}

Let $\mathfrak{N}(\Xi)$ be the set of all metrics $\mathfrak{n} = \{ \mathfrak{n}_{q}\}_{q \in \mathcal{Q}}$ in $\mathbb{E}$ over $\mathcal{Q}$ such that $\mathfrak{n}$ is continuous and uniformly weaker than $\|\cdot\|_{\mathbb{E}}$, see \eqref{EQ: UniformWeaknessFamNorm}; each $\mathfrak{n}_{q}$ is a norm on $\mathbb{E}$;  $\mathfrak{n}$ is admissible for $\Xi$, see below \eqref{EQ: LargetLyapunovExponentWRTMetric}; and $\lambda_{\mathfrak{n}}(\Xi) \geq \lambda_{1}(\Xi)$. If $\mathbb{E} = \mathbb{H}$ is a Hilbert space, we additionally require that each $\mathfrak{n}_{q}$ is induced by an inner product. By combining Theorems \ref{TH: LengthsFamilyOfNorms} and \ref{TH: LyapExpAdaptedMetricQuasiCom} (see Remark \ref{REM: HilbertSpaceAdaptedMetric}), we obtain the following variational description of $\lambda_{1}(\Xi)$.
\begin{theorem}
	\label{TH: VariationalDescriptionLargestLyapunovExponent}
	Under the conditions of Theorem \ref{TH: LyapExpAdaptedMetricQuasiCom}, for the largest uniform Lyapunov exponent $\lambda_{1}(\Xi)$ of $\Xi$, we have
	\begin{equation}
		\lambda_{1}(\Xi) = \inf_{ \mathfrak{n} \in \mathfrak{N}(\Xi)} \alpha^{+}_{\mathfrak{n}}(\Xi),
	\end{equation} 
    where $\alpha^{+}_{\mathfrak{n}}(\Xi)$ is given by \eqref{EQ: AlphaPlusInfExpMaximizationDefinition}.
\end{theorem}

%% file: ExteriorPowersAndLiouvilleFormula.tex
\section{Uniform Lyapunov exponents and Lyapunov dimension for cocycles in Hilbert spaces}
\label{SEC: VariationalDescUnifLyapExpCompCoc}

Throughout this section, we consider a uniformly continuous linear cocycle $\Xi$ over a semiflow $(\mathcal{Q},\vartheta)$ acting on a real or complex separable Hilbert space $\mathbb{H}$. Let $\langle\cdot,\cdot\rangle_{\mathbb{H}}$ denote the inner product in $\mathbb{H}$.

We are going to briefly introduce some basic constructions related to tensor and exterior products of Hilbert spaces. For more details, we refer to \cite{Anikushin2023Comp,Temam1997}.

Given $m \geq 1$, let $\mathbb{H}^{\odot m}$ denote the $m$-fold algebraic tensor product of $\mathbb{H}$. It is spanned by elements $\xi_{1} \otimes \cdots \otimes \xi_{m}$, called decomposable tensors, with $\xi_{1},\ldots,\xi_{m} \in \mathbb{H}$. Then the $m$-fold tensor product $\mathbb{H}^{\otimes m}$ of $\mathbb{H}$ is defined by the completion of $\mathbb{H}^{\odot m}$ in the norm $|\cdot|_{\mathbb{H}^{\otimes m}}$ induced by the inner product $\langle\cdot,\cdot\rangle_{\mathbb{H}^{\otimes m}}$ given by
\begin{equation}
	\label{EQ: TensorProductHm}
	\langle\xi_{1} \otimes \cdots \otimes \xi_{m}, \eta_{1} \otimes \cdots \otimes \eta_{m} \rangle_{\mathbb{H}^{\otimes m}} = \langle \xi_{1},\eta_{1}\rangle_{\mathbb{H}} \cdots \langle\xi_{m},\eta_{m}\rangle_{\mathbb{H}},
\end{equation}
where $\xi_{j},\eta_{j} \in \mathbb{H}$ for each $j \in \{1,\ldots,m\}$. Note that \eqref{EQ: TensorProductHm} is well defined and determines an inner product on $\mathbb{H}^{\odot m}$ due to the universal property of algebraic tensor products.

Let $\mathbb{S}_{m}$ be the symmetric group on $\{ 1,\ldots, m \}$. Define
\begin{equation}
	\label{EQ: WedgeProjectorDef}
	 \Pi^{\wedge}_{m}(\xi_{1} \otimes \cdots \otimes \xi_{m}) \coloneq \frac{1}{m!}\sum_{\sigma \in \mathbb{S}_{m}} (-1)^{\sigma} \xi_{\sigma(1)} \otimes \cdots \otimes \xi_{\sigma(m)}
\end{equation}
for all $\xi_{1},\ldots,\xi_{m} \in \mathbb{H}$. It can be shown that the above formula determines an orthogonal projector $\Pi^{\wedge}_{m}$ in $\mathbb{H}^{\otimes m}$.

Now, define the $m$-fold exterior power $\mathbb{H}^{\wedge m}$ of $\mathbb{H}$ by the range of $\Pi^{\wedge}_{m}$. For $\xi_{1},\ldots,\xi_{m} \in \mathbb{H}$, we set $\xi_{1} \wedge \cdots \wedge \xi_{m} \coloneq \Pi^{\wedge}_{m}(\xi_{1} \otimes \cdots \otimes \xi_{m})$. Such elements span a dense subspace in $\mathbb{H}^{\wedge m}$, and \eqref{EQ: TensorProductHm} induces the inner product $\langle\cdot,\cdot\rangle_{\mathbb{H}^{\wedge m}}$ in $\mathbb{H}^{\wedge m}$ given by
\begin{equation}
	\label{EQ: InnerProductWedgeHDefinition}
	\langle \xi_{1} \wedge \cdots \wedge \xi_{m} , \eta_{1} \wedge \cdots \wedge \eta_{m}\rangle_{\mathbb{H}^{\wedge m}} \coloneq \frac{1}{m!}\operatorname{det}[ \langle\xi_{k},\eta_{j}\rangle_{\mathbb{H}} ]_{k,j =1}^{m},
\end{equation}
where $\xi_{j},\eta_{j} \in \mathbb{H}$ for each $j \in \{1,\ldots,m \}$.

Suppose $\mathbb{H}_{1}$ and $\mathbb{H}_{2}$ are separable Hilbert spaces with norms $|\cdot|_{\mathbb{H}_{1}}$ and $|\cdot|_{\mathbb{H}_{2}}$, respectively. For each bounded linear operator $L \in \mathcal{L}(\mathbb{H}_{1};\mathbb{H}_{2})$, there is the associated \textit{$m$-fold multiplicative compound} $L^{\wedge m} \in \mathcal{L}(\mathbb{H}^{\wedge m}_{1};\mathbb{H}^{\wedge m}_{2})$ of $L$ given by
\begin{equation}
	L^{\wedge m}(\xi_{1}\wedge \cdots \wedge \xi_{m}) \coloneq L\xi_{1} \wedge \cdots \wedge L\xi_{m}
\end{equation}
for all $\xi_{1},\ldots,\xi_{m} \in \mathbb{H}_{1}$.

Recall that the singular values $\sigma_{1}(L) \geq \sigma_{2}(L) \geq \ldots$ of $L$ are defined by
\begin{equation}
	\label{EQ: SingularValueDef}
	\sigma_{k} \coloneq \sup_{\substack{\mathbb{L} \subset \mathbb{H}_{1} \\ \dim \mathbb{L} = k}}\inf_{\substack{\xi \in \mathbb{L} \\ |\xi|_{\mathbb{H}_{1}} = 1}} |L\xi|_{\mathbb{H}_{2}},
\end{equation}
where the supremum is taken over all $k$-dimensional subspaces $\mathbb{L}$ in $\mathbb{H}_{1}$ and $k=1,2,\ldots$. If $\mathbb{H}$ is finite-dimensional, we set $\sigma_{k}(L) \coloneq 0$ for any $k > \dim\mathbb{H}$. Note that the monotonicity of $\sigma_{k}$ in $k$ is clear from the definition. 

For any $d = m + \theta$ with integer $m \geq 0$ and $\theta \in [0,1]$, we define the \textit{function of singular values} of $L$ by
\begin{equation}
	\label{EQ: FuncSingularValueDef}
	 \omega_{d}(L) \coloneq \prod_{j=1}^{m} \sigma_{j}(L) \cdot \sigma^{\theta}_{m+1}(L) = (\omega_{m}(L))^{1-\theta} \cdot (\omega_{m+1}(L))^{\theta},
\end{equation}
where we set $\omega_{0}(L) \coloneq 1$ by definition.

We will need the following well-known lemma, which is essentially \cite[Proposition 1.4, Chapter V]{Temam1997}; however, here we present a more conceptual proof.
\begin{lemma}
	\label{LEM: TemamLemmaWedgeNormComputation}
	For each $L \in \mathcal{L}(\mathbb{H}_{1};\mathbb{H}_{2})$ and any $m=1,2,\ldots$, we have
	\begin{equation}
		\label{EQ: ExteriorNormSingularValuesIdentity}
		\| L^{\wedge m} \|_{\mathcal{L}(\mathbb{H}^{\wedge m}_{1};\mathbb{H}^{\wedge m}_{2})} = \sup_{\substack{e_{1}, \ldots, e_{m} \in \mathbb{H}_{1},\\ |e_{1} \wedge \cdots \wedge e_{m}|_{\mathbb{H}^{\wedge m}_{1}} = 1} } |L^{\wedge m}(e_{1} \wedge \cdots \wedge e_{m})|_{\mathbb{H}^{\wedge m}_{2}} = \Pi_{j=1}^{m} \sigma_{j}(L),
	\end{equation}
	where the supremum is taken over orthogonal vectors $e_{1},\ldots,e_{m} \in \mathbb{H}_{1}$.
\end{lemma}
\begin{proof}
	We first give a proof under the assumption that $L$ is compact. Then the adjoint $L^{*}$ of $L$ is also compact, and, consequently, $L^{*}L$ is a compact nonnegative definite self-adjoint operator in $\mathbb{H}_{1}$. From the Courant--Fischer--Weyl min-max principle, we obtain that the squared singular values of $L$ coincide with the arranged eigenvalues of $L^{*}L$. By the Hilbert--Schmidt theorem, there exists an orthonormal eigenbasis $\{e_{j} \}_{j \geq 1}$ such that $L^{*}Le_{j} = \sigma^{2}_{j}(L) e_{j}$ for any $j$. In particular, the vectors $e_{j_{1}} \wedge \cdots \wedge e_{j_{m}}/\sqrt{m!}$ taken over all indices $1 \leq j_{1} < \cdots < j_{m}$ form an orthonormal basis in $\mathbb{H}^{\wedge m}_{1}$. Consequently, for any Fourier coefficients $\{c_{j_{1}\ldots j_{m}}\}$ in this basis, we have (the sums are taken over all possible multi-indices $j_{1}\ldots j_{m}$)
	\begin{equation}
		\begin{split}
			(L^{*}L)^{\wedge m}\left( \sum_{j_{1}\ldots j_{m}} \frac{c_{j_{1}\ldots j_{m}}}{\sqrt{m!}} e_{j_{1}} \wedge \cdots \wedge e_{j_{m}}  \right) = \sum_{j_{1}\ldots j_{m}} \frac{c_{j_{1}\ldots j_{m}}}{\sqrt{m!}} L^{*}Le_{j_{1}} \wedge \cdots \wedge L^{*}Le_{j_{m}} =\\= \sum_{j_{1}\ldots j_{m}} \Pi_{k=1}^{m}\sigma^{2}_{j_{k}} \cdot \frac{c_{j_{1}\ldots j_{m}}}{\sqrt{m!}} e_{j_{1}} \wedge \cdots \wedge e_{j_{m}}.
		\end{split}
	\end{equation}
	From this and since $(L^{*}L)^{\wedge m} = (L^{*})^{\wedge m} L^{\wedge m}$ and $(L^{\wedge m})^{*} = (L^{*})^{\wedge m}$, we obtain
	\begin{equation}
		\| L^{\wedge m} \|^{2}_{\mathcal{L}(\mathbb{H}^{\wedge m}_{1};\mathbb{H}^{\wedge m}_{2})} = \| (L^{*} L)^{\wedge m} \|_{\mathcal{L}(\mathbb{H}^{\wedge m}_{1};\mathbb{H}^{\wedge m}_{2})} \leq \Pi_{j=1}^{m} \sigma^{2}_{j}(L).
	\end{equation}
	However, it is clear that the estimate is achieved on $e_{1} \wedge \cdots \wedge e_{m}$. This establishes \eqref{EQ: ExteriorNormSingularValuesIdentity} for compact operators $L$.
	
	For general $L$, we take orthonormal bases $\{e_{i}\}_{i \geq 1}$ in $\mathbb{H}_{1}$ and $\{f_{j}\}_{j \geq 1}$ in $\mathbb{H}_{2}$. For any $k=1,2,\ldots$, consider the projectors $\Pi^{e}_{k}$ and $\Pi^{f}_{k}$ onto $\operatorname{Span}\{ e_{1},\ldots,e_{k} \}$ and $\operatorname{Span}\{ f_{1},\ldots,f_{k} \}$, respectively. Since the truncated operators $\Pi^{f}_{k} L \Pi^{e}_{k}$ have finite-dimensional ranges, they compact and satisfy \eqref{EQ: ExteriorNormSingularValuesIdentity}. Then, it is straightforward to verify that in the limit as $k \to \infty$, all the quantities from \eqref{EQ: ExteriorNormSingularValuesIdentity} for the truncated operators converge to those of $L$ (this is the usual way to compute them in numerical experiments).
\end{proof}

An immediate corollary of Lemma \ref{LEM: TemamLemmaWedgeNormComputation} and the submultiplicativity of the operator norm is the following inequality (also known as the Horn inequality \cite{KuzReit2020}).
\begin{corollary}
	\label{COR: HornInequality}
	For any operators $L_{1} \in \mathcal{L}(\mathbb{H}_{1}, \mathbb{H}_{2})$ and $L_{2} \in \mathcal{L}(\mathbb{H}_{2};\mathbb{H}_{3})$, we have
	\begin{equation}
		\omega_{d}(L_{2}L_{1}) \leq \omega_{d}(L_{1}) \omega_{d}(L_{2}) \qquad \text{for any} \quad d \geq 0.
	\end{equation}
\end{corollary}

Now we define the \textit{$m$-fold multiplicative compound} $\Xi_{m}$ of $\Xi$ as the cocycle in $\mathbb{H}^{\wedge m}$ with the mappings $\Xi^{t}_{m}(q,\cdot)$ for all $q \in \mathcal{Q}$ and $t \geq 0$ given by the $m$-fold multiplicative compound of $\Xi^{t}(q,\cdot)$, i.e.,
\begin{equation}
	\label{EQ: CompoundCocycleDefMetricsLiov}
	\Xi^{t}_{m}(q,\xi_{1} \wedge \cdots \wedge \xi_{m}) \coloneq \Xi^{t}(q,\xi_{1}) \wedge \cdots \wedge \Xi^{t}(q,\xi_{m})
\end{equation}
for all $\xi_{1},\ldots,\xi_{m} \in \mathbb{H}$. By \cite[Proposition 3.1]{Anikushin2023Comp}, $\Xi_{m}$ is a uniformly continuous linear cocycle in $\mathbb{H}^{\wedge m}$ over $(\mathcal{Q},\vartheta)$.

For any $m=1,2,\ldots$, let us consider the largest uniform Lyapunov exponent $\lambda_{1}(\Xi_{m})$ of $\Xi_{m}$. Then the \textit{uniform Lyapunov exponents} $\lambda_{1}(\Xi), \lambda_{2}(\Xi), \ldots$ of $\Xi$ are defined by induction from the relations for each $m=1,2,\ldots$ as
\begin{equation}
	\label{EQ: UniformLyapExponentsFormulaDef}
	\lambda_{1}(\Xi) + \cdots + \lambda_{m}(\Xi) = \lambda_{1}(\Xi_{m}).
\end{equation}
Note that except $\lambda_{1}(\Xi)$, $\lambda_{j}(\Xi)$ is not necessarily a Lyapunov exponent over any ergodic measure. Moreover, these quantities are not necessarily monotone\footnote{There are, however, misunderstandings in the literature related to the monotonicity. See, e.g., \cite[Section 4]{Gelfert2003}.} in $j$. Both situations are illustrated in the following example.
\begin{example}
	\label{EX: ULENonMonotone}
	Suppose $\mathbb{H} = \mathbb{R}^{3}$, and let $\mathcal{Q}$ consist of three equilibria $q_{1}, q_{2}$, and $q_{3}$. Let local Lyapunov exponents at the equilibria be $\lambda_{1}(q_{1}) = 1$, $\lambda_{2}(q_{1}) = -1$, $\lambda_{3}(q_{1}) = -1$; $\lambda_{1}(q_{2}) = 1/2$, $\lambda_{2}(q_{2}) = 0$, $\lambda_{3}(q_{2}) = -1$; $\lambda_{1}(q_{3}) = \lambda_{2}(q_{3}) = \lambda_{3}(q_{3}) = 1/5$. Since $\lambda_{1}(\Xi)=\lambda_{1}(q_{1})$, $\lambda_{1}(\Xi_{2}) = \lambda_{1}(q_{2})+\lambda_{2}(q_{2})$, and $\lambda_{1}(\Xi_{3}) = \lambda_{1}(q_{3}) + \lambda_{2}(q_{3}) + \lambda_{3}(q_{3}) = 3/5$, we have $\lambda_{1}(\Xi) = 1$, $\lambda_{2}(\Xi) = -1/2$, and $\lambda_{3}(\Xi) = 1/10$. Using techniques from analysis, one can construct less artificial examples with $\mathcal{Q}$ as an attractor for a flow containing similar equilibria as the only ergodic components. Then, the local analysis is sufficient by the Ergodic Variational Principle (see Appendix \ref{APP: EVPNonCompact}).
\end{example}

In view of \eqref{EQ: UniformLyapExponentsFormulaDef}, uniform Lyapunov exponents of $\Xi$ are expressed via the largest uniform Lyapunov exponents $\lambda_{1}(\Xi_{m})$ of $m$-fold compound cocycles $\Xi_{m}$. By applying Theorem \ref{TH: VariationalDescriptionLargestLyapunovExponent} to $\Xi_{m}$, we obtain the following variational description of these quantities.
\begin{theorem}
	\label{TH: VariationalDescriptionUniformLyapunovExponents}
	Let $\Xi$ be a uniformly continuous linear cocycle in $\mathbb{H}$ over a semiflow $(\mathcal{Q},\vartheta)$, which admits a $2^{\mathcal{Q}}$-attractor as in \eqref{EQ: SubsetA0AttractsQ}. For an integer $m \geq 1$, suppose that $\Xi_{m}$ is uniformly quasi-compact. Then we have
	\begin{equation}
		\label{EQ: VariationalDescriptionMUniformLyapExps}
		\lambda_{1}(\Xi) + \cdots + \lambda_{m}(\Xi) = \lambda_{1}(\Xi_{m}) = \inf_{ \mathfrak{m} \in \mathfrak{N}(\Xi_{m})} \alpha^{+}_{\mathfrak{m}}(\Xi_{m}),
	\end{equation}
    where $\alpha^{+}_{\mathfrak{m}}(\Xi_{m})$ is given by \eqref{EQ: AlphaPlusInfExpMaximizationDefinition}, and $\mathfrak{N}(\Xi_{m})$ is described above Theorem \ref{TH: VariationalDescriptionLargestLyapunovExponent}.
\end{theorem}
\begin{remark}
	If $\Xi$ is eventually compact, i.e., there exists $\tau>0$ such that $\Xi^{\tau}(q,\cdot)$ is compact for every $q \in \mathcal{Q}$, then $\Xi_{m}$ is also eventually compact by \cite[Proposition 3.2]{Anikushin2023Comp} and, as a consequence, quasi-compact. This makes the characterization \eqref{EQ: VariationalDescriptionMUniformLyapExps} applicable to parabolic and delay equations.
\end{remark}
We postpone some related discussions concerning Theorem \ref{TH: VariationalDescriptionUniformLyapunovExponents} to the end of the section. Next, we introduce the Lyapunov dimension of a cocycle.

Consider the \textit{averaged function of singular values} of $\Xi$, which is given by\footnote{We emphasize that in this definition, each mapping of the cocycle is considered in the same space $\mathbb{H}$ endowed with the same inner product, although considering uniformly equivalent metrics over $\mathcal{Q}$ yields the same asymptotic quantity $\omega_{d}(\Xi)$, see \eqref{EQ: OmegaInGeneralMetric}. For estimating the corresponding quantities in varying metrics, see the discussion below \eqref{EQ: AveragedFSVViaMetricsProblem} and Section \ref{SEC: ComputationOfInfExponents}.}
\begin{equation}
	\label{EQ: AveragedFunctionOfSingularValues}
	\omega_{d}(\Xi) \coloneq \lim_{t \to +\infty} \left(\sup_{q \in \mathcal{Q}}\omega_{d}(\Xi^{t}(q,\cdot))\right)^{1/t},
\end{equation}
where the limit exists due to Corollary \ref{COR: HornInequality} and Lemma \ref{LEM: FeketeLemma}. From this, one can define the \textit{Lyapunov dimension} of $\Xi$ as
\begin{equation}
	\label{EQ: LyapunovDimensionDefinition}
	\dim_{\operatorname{L}}\Xi = \dim_{\operatorname{L}}(\Xi;\mathcal{Q}) \coloneq \inf\{ d \geq 0 \ | \ \omega_{d}(\Xi) < 1 \},
\end{equation}
where the infimum over an empty set is assumed to be $\infty$. By the monotonicity of singular values, $\omega_{d_{0}}(\Xi) < 1$ for some $d_{0} \geq 0$ implies that $\omega_{d}(\Xi) < 1$ for any $d \geq d_{0}$. This shows the correctness of the definition \eqref{EQ: LyapunovDimensionDefinition}.

From Lemma \ref{LEM: TemamLemmaWedgeNormComputation}, we obtain $\ln\omega_{m}(\Xi) = \lambda_{1}(\Xi_{m})$ for any $m=1,2,\ldots$. Moreover, for $d = m + \theta$ with $\theta \in [0,1]$, we have
\begin{equation}
	\label{EQ: FunctionOfSingularValuesViaExtProd}
	\ln\omega_{d}(\Xi) = \lim_{t \to +\infty} \frac{1}{t} \ln \left( \|\Xi^{t}_{m}(q,\cdot)\|^{1-\theta}_{\mathcal{L}(\mathbb{H}^{\wedge m})} \cdot \|\Xi^{t}_{m+1}(q,\cdot)\|^{\theta}_{\mathcal{L}(\mathbb{H}^{\wedge (m+1)})} \right).
\end{equation}
In particular, for any $d = m + \theta$ as above, we have the inequality
\begin{equation}
	\label{EQ: AveragedSingularMajorizedViaFractionalSumULE}
	\ln \omega_{d}(\Xi) \leq \sum_{j=1}^{m}\lambda_{j}(\Xi) + \theta \lambda_{m+1}(\Xi)
\end{equation}
which becomes equality for $\theta = 0$.

\begin{remark}
	If $\mathcal{Q}$ is a compact subset of $\mathbb{H}$, and $\Xi$ is a linearization cocycle of $\vartheta$ over $\mathcal{Q}$ (i.e., each mapping $\Xi^{t}(q,\cdot)$ is a quasi-differential of $\vartheta^{t}$ at $q \in \mathcal{Q}$), then the Lyapunov dimension of $\Xi$ bounds from above the fractal dimension of $\mathcal{Q}$, see \cite{ChepyzhovIlyin2004, ZelikAttractors2022}.
\end{remark}
\begin{remark}
	More precisely, $\dim_{\operatorname{L}}\Xi$ should be referred to as the \textit{uniform Lyapunov dimension} to distinguish it from the Lyapunov dimension over an ergodic measure. We also note that in \cite{EdenLocalEstimates1990}, $\dim_{\operatorname{L}}\Xi$ is called the Douady-Oesterl\'{e} dimension, while the term Lyapunov dimension corresponds to the quantity $\dim^{KY}_{\operatorname{L}}\Xi$ introduced below. Currently, our terminology is the most widely used \cite{KuzReit2020, ZelikAttractors2022}.
\end{remark}
\begin{remark}
	\label{REM: AsymCompFiniteLD}
	Recall the uniform compactness exponent $\alpha_{K}(\Xi)$ of $\Xi$ defined in \eqref{EQ: KuratowskiExponentOfCocycle}. If $\alpha_{K}(\Xi) < 0$, the cocycle is called \textit{uniformly asymptotically compact}. Supposing that $(\mathcal{Q},\vartheta)$ admits a $2^{\mathcal{Q}}$-attractor, from the Ergodic Variational Principle (see Section \ref{APP: EVPNonCompact}) and the Multiplicative Ergodic Theorem (see \cite{Thieullen1992}), it may be deduced that $\alpha_{K}(\Xi) < 0$ if and only if $\dim_{\operatorname{L}}\Xi < \infty$.
\end{remark}

Let us discuss a related variant of dimension defined in terms of uniform Lyapunov exponents. Consider the largest integer $m \geq 0$ (if it exists) such that $\lambda_{1}(\Xi) + \cdots + \lambda_{m}(\Xi) \geq 0$, where for $m=0$ the sum is taken to be zero. Then, it follows that $\lambda_{1}(\Xi) + \cdots + \lambda_{m+1}(\Xi) < 0$, and the \textit{uniform Kaplan--Yorke Lyapunov dimension} can be defined by
\begin{equation}
	\label{EQ: Kaplan-YorkeFormula}
	\dim^{KY}_{\operatorname{L}}\Xi = \dim^{KY}_{\operatorname{L}}(\Xi;\mathcal{Q}) \coloneq m + \frac{\sum_{j=1}^{m}\lambda_{j}(\Xi)}{|\lambda_{m+1}(\Xi)|}.
\end{equation}
If there is no such $m$, we set $\dim^{KY}_{\operatorname{L}}\Xi \coloneq \infty$. Without specifying a name, this quantity apparently first appeared in the work \cite{ConstantinFoias1985}, and, of course, it was motivated by the Kaplan--Yorke conjecture.

\begin{remark}
	If $\mathbb{H} = \mathbb{R}^{n}$, then we have $\lambda_{m}(\Xi) = -\infty$ for $m > n$ by definition. In particular, even if $\lambda_{1}(\Xi) + \cdots + \lambda_{n}(\Xi) \geq 0$, \eqref{EQ: Kaplan-YorkeFormula} gives $\dim^{KY}_{\operatorname{L}}\Xi = n$.
\end{remark}

\begin{remark}
	\label{REM: LyapDimEstApplications}
	Let $\Sigma(m) \coloneq \sum_{j=1}^{m}\lambda_{j}(\Xi)$ for any integer $m \geq 1$, $\Sigma(0)\coloneq 0$, and define $\Sigma(m+\theta) \coloneq \Sigma(m) + \theta\lambda_{m+1}$ for $\theta \in [0,1]$ and any integer $m \geq 0$. Then for $\lambda_{m+1} > -\infty$, the fractional term in \eqref{EQ: Kaplan-YorkeFormula} is given by $\theta^{*} \in [0,1]$ such that $\Sigma(m+\theta^{*}) = 0$. In applications, it often happens that we deal with an estimate $\Sigma(m) \leq \sigma^{+}(m)$ that holds for each integer $m \geq 1$ and some function $\sigma^{+}(\cdot)$. Suppose that $\sigma^{+}(m)$ can be defined for all real $m \geq 1$. Then it is clear that $\sigma^{+}(d^{*}) = 0$ for some $d^{*} \in [m,m+1]$ would imply $\dim^{KY}_{\operatorname{L}}\Xi \leq d^{*}$ provided that $\sigma^{+}(\cdot)$ is concave on $[m,m+1]$. By virtue of \eqref{EQ: KYLDboundsLD}, this observation makes the quantity $\dim^{KY}_{\operatorname{L}}\Xi$ more convenient to handle when exact computations are not required, but rather asymptotically sharp upper estimates (in certain parameters) are desired. 
\end{remark}

It is not difficult to see the following relation between the introduced quantities.
\begin{proposition}
	We always have the inequalities
	\begin{equation}
		\label{EQ: KYLDboundsLD}
		\dim^{KY}_{\operatorname{L}}\Xi - 1 < \dim_{\operatorname{L}}\Xi \leq \dim^{KY}_{\operatorname{L}}\Xi.
	\end{equation}
\end{proposition}
\begin{proof}
	Since $\ln \omega_{m}(\Xi) = \sum_{j=1}^{m}\lambda_{j}(\Xi)$ for any integer $m \geq 0$, we have that $\dim^{KY}_{\operatorname{L}}\Xi=\infty$ if and only if $\dim_{\operatorname{L}}\Xi = \infty$. Therefore, we may assume that both quantities are finite.
	
	Now suppose $m$ and $\theta^{*} \in [0,1)$ are such that $\dim^{KY}_{\operatorname{L}}\Xi = m + \theta^{*}$.  Then $\ln \omega_{m}(\Xi) = \sum_{j=1}^{m}\lambda_{j}(\Xi) \geq 0$ and $\ln \omega_{m+1}(\Xi) = \sum_{j=1}^{m+1}\lambda_{j}(\Xi) < 0$. Consequently, $\dim_{\operatorname{L}}\Xi \in [m, m+1)$. Furthermore, from \eqref{EQ: AveragedSingularMajorizedViaFractionalSumULE} with $d = m + \theta$ and any $\theta \in (\theta^{*}, 1]$, we obtain $\ln \omega_{d}(\Xi) < 0$. Thus, $\dim_{\operatorname{L}}\Xi \leq m + \theta$ for any $\theta > \theta^{*}$. Taking it to the limit as $\theta \to \theta^{*}$ yields \eqref{EQ: KYLDboundsLD}.
\end{proof}

We also have the following simple criterion for the coincidence of both quantities.
\begin{proposition}
	Suppose there exists $q \in \mathcal{Q}$ such that for some integer $m \geq 1$ we have
	\begin{equation}
		\begin{split}
			&\lim_{t \to +\infty}\frac{1}{t}\ln \| \Xi^{t}_{m}(q,\cdot) \|_{\mathcal{L}(\mathbb{H}^{\wedge m})} = \lambda_{1}(\Xi_{m}) \geq 0,\\
			&\lim_{t \to +\infty}\frac{1}{t}\ln \| \Xi^{t}_{m+1}(q,\cdot) \|_{\mathcal{L}(\mathbb{H}^{\wedge (m+1)})} = \lambda_{1}(\Xi_{m+1}) < 0.
		\end{split}
	\end{equation}
	Then
	\begin{equation}
		\dim_{\operatorname{L}}\Xi = \dim^{KY}_{\operatorname{L}}\Xi.
	\end{equation}
\end{proposition}
\begin{proof}
	Indeed, from \eqref{EQ: AveragedSingularMajorizedViaFractionalSumULE} and the hypotheses, for any $d = m + \theta$ with $\theta \in [0,1]$, we have
	\begin{equation}
		\begin{split}
			\sum_{j=1}^{m}\lambda_{j}(\Xi) + \theta \lambda_{m+1}(\Xi) \geq \omega_{d}(\Xi) \geq \\
			\geq \lim_{t \to +\infty}\frac{1}{t}\ln \left(\| \Xi^{t}_{m}(q,\cdot) \|^{1 - \theta}_{\mathcal{L}(\mathbb{H}^{\wedge m})} \cdot \| \Xi^{t}_{m+1}(q,\cdot) \|^{\theta}_{\mathcal{L}(\mathbb{H}^{\wedge (m+1)})}\right) = \sum_{j=1}^{m}\lambda_{j}(\Xi) + \theta \lambda_{m+1}(\Xi).
		\end{split}
	\end{equation}
	From this, the conclusion follows.
\end{proof}

\begin{remark}
	In terms of the ergodic characterization discussed in Appendix \ref{APP: EVPNonCompact}, the strict inequality $\dim_{\operatorname{L}}\Xi < \dim^{KY}_{\operatorname{L}}\Xi$ is likely to occur\footnote{In particular, the reversed inequality $\dim^{KY}_{\operatorname{L}}\Xi \leq \dim_{\operatorname{L}}\Xi$ stated in \cite[formula (21)]{Gelfert2003} is incorrect.} when there is no common maximizing measure for $\lambda_{1}(\Xi_{m})$ and $\lambda_{1}(\Xi_{m+1})$ (an example can be constructed similarly to Example \ref{EX: ULENonMonotone} in $\mathbb{H} = \mathbb{R}^{2}$ and with two equilibria). Such problems can be analyzed in some low-dimensional systems, such as the H\'{e}non map or the Lorenz system, where the derivative cocycle mappings have constant determinants, allowing the smallest Lyapunov exponent to be determined from the preceding ones. For the former, it suffices to find a maximizing measure (or even a single point) for $\lambda_{1}(\Xi)$, after which $\omega_{d}(\Xi)$ can be computed for any $d \in [0,2]$. For the latter, it is only necessary to find a maximizing measure for $\lambda_{1}(\Xi_{2})$. Although the existing computations of the Lyapunov dimension for such systems (see \cite{LeonovForHenonLor2001, LeoKuzKorzhKusakin2016}) do not explicitly use this observation, it may simplify the calculations.
\end{remark}

Modulo the conditions of Theorem \ref{TH: VariationalDescriptionUniformLyapunovExponents}, uniform Lyapunov exponents are computable via adapted metrics. Consequently, $\dim^{KY}_{\operatorname{L}}\Xi$ is also computable. Beyond the case where $\dim_{\operatorname{L}}\Xi$ coincides with $\dim^{KY}_{\operatorname{L}}\Xi$, it should be asked whether the former can also be computed in this manner. More precisely, in the case of discrete time, this question relates to whether
\begin{equation}
	\label{EQ: AveragedFSVViaMetricsProblem}
	\omega_{d}(\Xi) = \inf_{\mathfrak{n}}\sup_{q \in \mathcal{Q}} \omega^{(\mathfrak{n})}_{d}(\Xi^{1}(q,\cdot)),
\end{equation}
where the supremum is taken over all continuous and equivalent metrics $\mathfrak{n}$ in $\mathbb{H}$ over $\mathcal{Q}$ associated with families of inner products, and $\omega^{(\mathfrak{n})}_{d}(\Xi^{1}(q,\cdot))$ is the function of singular values of $\Xi^{1}(q,\cdot)$ considered as an operator from $(\mathbb{H}, \mathfrak{n}_{q})$ to $(\mathbb{H}, \mathfrak{n}_{\vartheta^{1}(q)})$.

In \cite{KawanPogromsky2021}, Kawan, Matveev, and Pogromsky, building on the geometric ideas of Bochi \cite{Bochi2018}, established the validity of \eqref{EQ: AveragedFSVViaMetricsProblem} for finite-dimensional invertible cocycles (both with discrete and continuous\footnote{For continuous time, the question relates to whether the estimate \eqref{EQ: OmegaEstimateViaMaxOfBetas} can be made arbitrarily close to $\ln \omega_{d}(\Xi)$ by varying $\mathfrak{n}$.} time) over a compact base $\mathcal{Q}$. Moreover, the corresponding approximating metrics --- smooth and equivalent to the standard metric --- can be chosen for all $d$ simultaneously. In particular, this result strengthens Theorem \ref{TH: VariationalDescriptionUniformLyapunovExponents} by demonstrating that one can vary over metrics $\mathfrak{m}$ on $\mathbb{H}^{\wedge m}$ over $\mathcal{Q}$ that are associated with metrics $\mathfrak{n}$ on $\mathbb{H}$ over $\mathcal{Q}$. 

Note also that the approximating metrics in \cite{KawanPogromsky2021} are defined as (unweighted) barycenters (in the space of metrics) taken over consecutive pullbacks of a given metric under the action of cocycle mappings, which requires invertibility. This construction, representing a time-averaging of spatial structures, is a counterpart to Lyapunov metrics, which are constructed via time-averaging along individual trajectories. We also note that metrics of this kind provide a modern perspective on the entropy theory of dynamical systems; in fact, this perspective was originally proposed by Shannon but remained unnoticed for a long time (see the survey of Vershik, Veprev, and Zatitskii \cite{VershikVeprevZatitskii2023}). Therefore, it is reasonable to call such metrics \textit{Shannon-like}.

\begin{remark}
	As we noted, the result of \cite{KawanPogromsky2021} is based on ideas from \cite{Bochi2018}, where similar results are established for a given ergodic measure $\mu$, see \cite[p.~17]{Bochi2018}. Here, instead of taking the supremum over $q \in \mathcal{Q}$ in \eqref{EQ: AveragedFSVViaMetricsProblem}, one should take the integral over $\mu$. Then, by minimizing over metrics, we obtain $\omega_{d}(\Xi;\mu)$, which is defined analogously to \eqref{EQ: AveragedFunctionOfSingularValues} by pointwise limits holding $\mu$-almost everywhere, thanks to Theorem \ref{TH: KingmanSETDiscreteTime}. Therefore, to perform related computations in practice, one must know the distribution of $\mu$.
\end{remark}
\begin{remark}
	\label{REM: MaximizingMeasuresEdenConjecture}
	According to Corollary \ref{COR: DVolumesUniformAsDVolumesOverMeasure}, $\omega_{d}(\Xi)$ can be realized as $\omega_{d}(\Xi;\mu)$ for an ergodic measure $\mu$, referred to as a \textit{maximizing measure} for $\omega_{d}(\Xi)$. In this context, we also say that $\omega_{d}(\Xi)$ is achieved by $\mu$. There is a conjecture in ergodic optimization \cite{Bochi2018} which states that, generically, maximizing measures for subadditive families are supported on periodic orbits with relatively small periods. For particular systems, the realization of $\omega_{d}(\Xi)$ with $d=\dim_{\operatorname{L}}\Xi$ on a periodic orbit or at an equilibrium is known as the Eden conjecture, originating from the studies of Eden \cite{EdenLocalEstimates1990}.
\end{remark}

Let us discuss the attainability of the infimum in \eqref{EQ: AveragedFSVViaMetricsProblem} and its continuous-time version. Known exact computations of the Lyapunov dimension for the H\'{e}non map \cite{LeonovForHenonLor2001, KuzReit2020}, the Lorenz and Lorenz-like systems \cite{KuzMokKuzKud2020, Leonov2016LorenzLike, LeoKuzKorzhKusakin2016}, and the Ginzburg-Landau equation \cite{Doeringetal1987} show that it may be sufficient to use only a single adapted metric that is smooth, coercive, and induced by a metric on $\mathbb{H}$. In all of these works, the maximizing measures (see Remark \ref{REM: MaximizingMeasuresEdenConjecture}) are concentrated at equilibria. Moreover, for some problems, the infimum may be attained not on a Riemannian metric but rather on a norm, as in the case of Barabanov norms \cite{Protasov2022}.

This suggests that variational problems should be considered within restricted classes of metrics, which are much easier to handle. Recent advances in this area focus on the numerical resolution of such problems via subgradient optimization \cite{LouzeiroetAlAdaptedMetricsComp2022, KawanHafsteinGiesl2021} and nonlinear constrained optimization \cite{AnikushinRomanov2025RobustEstimates}. In practice, the subgradient method is limited to geodesically convex subsets in the space of metrics, for which related formulas for its iterations can also be derived. To date, effective algorithms have only been developed for optimization within the class of metrics $e^{V(q)}|\cdot|_{\mathbb{H}}$, where $V(q)$ is a varying polynomial and $|\cdot|_{\mathbb{H}}$ is a varying Euclidean norm in $\mathbb{H} = \mathbb{R}^{n}$. As a general multiparametric class of metrics has nothing to do with the geodesic convexity, it is unlikely that the subgradient method can be used for optimization over general smooth families. In contrast, the iterative nonlinear programming method proposed in \cite{AnikushinRomanov2025RobustEstimates} does not have such limitations. In particular, it successfully justified that the positive equilibrium on the classical H\'{e}non attractor is maximizing, a result not achieved in \cite{LouzeiroetAlAdaptedMetricsComp2022, KawanHafsteinGiesl2021}. For further discussion, we refer the reader to \cite{AnikushinRomanov2025RobustEstimates}.

Thus, there is potential for more detailed studies of restricted variational problems and the development of approximation techniques.

Although, in theory, it may be sufficient to use only associated metrics on exterior powers, practical approaches often lead to non-associated metrics. For example, there are approaches involving applications of the Frequency Theorem \cite{Anikushin2020FreqDelay, Anikushin2020FreqParab}, which provides conditions (frequency inequalities) for the existence of quadratic Lyapunov-like functionals by resolving infinite-horizon indefinite quadratic regulator problems. For systems on exterior powers (infinitesimal generators of $\Xi_{m}$), such functionals correspond to non-associated metrics. In our work \cite{Anikushin2023Comp}, we laid theoretical foundations for related applications in the case of delay equations. Furthermore, in our joint paper with Romanov \cite{AnikushinRomanov2023FreqConds}, we developed approximation schemes to verify frequency inequalities through numerical experiments. In certain cases, these results can potentially be rigorously validated using interval arithmetic.

It is natural that, for continuous-time systems in infinite-dimensional spaces, Lyapunov metrics may fail to be coercive due to the presence of smoothing properties. This phenomenon is frequently encountered in the study of stability problems and inertial manifolds using quadratic Lyapunov-like functionals (see our works cited above and \cite{Anikushin2020Geom, Anikushin2022Semigroups, AnikushinAADyn2021}).

%% file: ComputaionOfInfinitesimalExponents.tex
\section{Infinitesimal growth exponents of compound cocycles in associated metrics}
\label{SEC: ComputationOfInfExponents}

Let $\Xi$ be a uniformly continuous linear cocycle in a separable real or complex Hilbert space $\mathbb{H}$ over a semiflow $(\mathcal{Q},\vartheta)$ on a complete metric space $\mathcal{Q}$. Suppose $\mathfrak{n}=\{ \mathfrak{n}_{q} \}_{q \in \mathcal{Q}}$ is a metric in $\mathbb{H}$ over $\mathcal{Q}$, and let $\Xi$ admit infinitesimal growth exponents with respect to $\mathfrak{n}$, i.e., \nameref{DESC: GE1}, \nameref{DESC: GE2}, and \nameref{DESC: GE3} are satisfied. We further suppose that each $\mathfrak{n}_{q}$ is a norm induced by an inner product $\langle \cdot, \cdot \rangle_{q}$ in $\mathbb{H}$, which is equivalent (not necessarily uniformly) to the standard one $\langle \cdot, \cdot \rangle_{\mathbb{H}}$. In this section, we aim to provide a basis for computation of infinitesimal growth exponents in associated with $\mathfrak{n}$ metrics $\mathfrak{n}^{\wedge m}$ on $\mathbb{H}^{\wedge m}$.

\begin{remark}
	For cocycles in Banach spaces, the development of appropriate theories of computation is a topic of recent research. Here, differentiation of norms leads to so-called weak pairings. See \cite{Davydovetal2024} for the case of constant metrics.
\end{remark}

Since a symmetric\footnote{We use the same term ``symmetric form'' for both symmetric bilinear and Hermitian (sesqulinear) forms in the real and complex cases, respectively. Moreover, we follow the convention that a Hermitian form is linear in the first argument and conjugate-linear in the second.} form can be expressed via its quadratic form, by \nameref{DESC: GE3}, the limit
\begin{equation}
	\label{EQ: CurvatureFormDefinition}
	B_{\mathfrak{n}}(q;\xi,\eta) \coloneq \frac{1}{2} \lim_{t \to 0+} \frac{\langle \Xi^{t}(q,\xi), \Xi^{t}(q,\eta) \rangle_{\vartheta^{t}(q)} - \langle \xi, \eta \rangle_{q}  }{t}
\end{equation}
exists for all $q \in \mathcal{Q}$ and $\xi,\eta \in \mathcal{N}(q)$, where $\mathcal{N}(q)$ is delivered by \nameref{DESC: GE1}. Clearly, $B_{\mathfrak{n}}(q;\xi,\eta)$ is linear in $\xi$ and conjugate-linear in $\eta$. We use the notation $B_{\mathfrak{n}}(q)$ for the corresponding symmetric form, to which we also refer as the \textit{curvature form} of the cocycle $\Xi$ in the metric $\mathfrak{n}$ over $q$. Taken together, these forms are denoted by $B_{\mathfrak{n}}$ and are called the \textit{curvature form} of $\Xi$ in the metric $\mathfrak{n}$.

We call a (possibly unbounded) operator $T_{\mathfrak{n}}(q)$ in $\mathbb{H}$ a \textit{representation} of $B_{\mathfrak{n}}(q)$ if its domain $\mathcal{D}(T_{\mathfrak{m}}(q))$ contains $\mathcal{N}_{\mathfrak{n}}(q)$ and
\begin{equation}
	B_{\mathfrak{n}}(q;\xi,\eta) = \langle \xi, T_{\mathfrak{n}}(q)\eta  \rangle_{q} \qquad \text{for all} \quad \xi, \eta \in \mathcal{N}(q).
\end{equation}
Since $B_{\mathfrak{n}}(q)$ is symmetric, $T_{\mathfrak{n}}(q)$ must be symmetric with respect to $\langle \cdot, \cdot \rangle_{q}$. In fact, in infinite dimensions, finding a representation of $B_{\mathfrak{n}}(q)$ on the entire $\mathcal{N}(q)$ is not always possible. This will be discussed later.

Given an operator $T$ in $\mathbb{H}$ with domain $\mathcal{D}(T)$, we define the \textit{algebraic $m$-fold additive compound} $T^{[\wedge m]}$ of $T$ as an operator on $\mathbb{H}^{\wedge m}$ with the domain $\mathcal{D}(T^{[\wedge m]})$ spanned by $\xi_{1} \wedge \cdots \wedge \xi_{m}$ with $\xi_{j} \in \mathcal{D}(T)$ for each $j \in \{1,\ldots,m\}$ and given by
\begin{equation}
	\label{EQ: AlgebraicAdditiveCompoundDef}
	T^{[\wedge m]}(\xi_{1} \wedge \cdots \wedge \xi_{m}) \coloneq \sum_{j=1}^{m} \xi_{1} \wedge \cdots \wedge T\xi_{j} \wedge \cdots \wedge \xi_{m}.
\end{equation}
From the universal property of algebraic tensor products and the antisymmetricity of the above formula, it is not difficult to see that $T^{[\wedge m]}$ is well-defined.

With $\mathfrak{n}$, we can associate a metric $\mathfrak{n}^{\wedge m}$ in $\mathbb{H}^{\wedge m}$ over $\mathcal{Q}$, which is induced by the associated with $\langle \cdot, \cdot \rangle_{q}$ inner products as in \eqref{EQ: InnerProductWedgeHDefinition}.

We start by establishing several algebraic statements.
\begin{theorem}[Representations in associated metrics]
	Let $\mathfrak{n}$ and $\mathfrak{n}^{\wedge m}$ be as above. Then $\Xi_{m}$ admits infinitesimal growth exponents with respect to $\mathfrak{n}^{\wedge m}$ with domains $\mathcal{N}_{\mathfrak{n}^{\wedge m}}(q)$ given by the algebraic $m$-fold exterior products $(\mathcal{N}_{\mathfrak{n}}(q))^{\wedge m}$ of $\mathcal{N}_{\mathfrak{n}}(q)$ for any $q \in \mathcal{Q}$. Moreover, for any representation $T_{\mathfrak{n}}(q)$ of $B_{\mathfrak{n}}(q)$, the operator $T^{[\wedge m]}_{\mathfrak{n}}(q) \coloneq $ $(T_{\mathfrak{n}}(q))^{[\wedge m]}$ is a representation of $B_{\mathfrak{n}^{\wedge m}}(q)$.
\end{theorem}
\begin{proof}
	By linearity, it is sufficient to establish the statements for decomposable tensors $\xi_{1} \wedge \cdots \wedge \xi_{m}$ and $\eta_{1} \wedge \cdots \wedge \eta_{m}$, where $\xi_{j}, \eta_{j} \in \mathcal{N}_{\mathfrak{n}}(q)$ for each $j \in \{1,\ldots,m\}$.
	
	From the definitions of $\mathfrak{n}^{\wedge m}$ and $\Xi_{m}$ and the Leibniz formula for the determinant (see below), it immediately follows that the function
	\begin{equation}
		t \mapsto \left\langle \Xi^{t}_{m}(q, \xi_{1} \wedge \cdots \wedge \xi_{m}),  \Xi^{t}_{m}(q, \eta_{1} \wedge \cdots \wedge \eta_{m}) \right\rangle_{\vartheta^{t}(q)}
	\end{equation}
	is absolutely continuous on finite intervals and differentiable from the right at $t=0$. By linearity, this also holds on the entire $(\mathcal{N}(q))^{\wedge m}$. Consequently, $\Xi_{m}$ admits infinitesimal growth exponents with respect to $\mathfrak{n}^{\wedge m}$.
	
	From the above, we have (here the derivative $\frac{d}{dt}$ is taken at $t=0$)
	\begin{equation}
		\label{EQ: RepresentationAssociatedMetric1}
		\begin{split}
			m! \cdot B_{\mathfrak{n}^{\wedge m}}(q; \xi_{1} \wedge \cdots \wedge \xi_{m}, \eta_{1} \wedge \cdots \wedge \eta_{m}) =\\
			= \frac{1}{2}\frac{d}{dt} \det\left[ \left\langle \Xi^{t}(q,\xi_{i}), \Xi^{t}(q,\eta_{j}) \right\rangle_{\vartheta^{t}(q)} \right]_{i,j=1}^{m} =\\
			= \frac{1}{2}\frac{d}{dt} \sum_{ \sigma \in \mathbb{S}_{m} } (-1)^{\sigma} \prod_{j=1}^{m}  \left\langle \Xi^{t}(q,\xi_{\sigma(j)}), \Xi^{t}(q,\eta_{j}) \right\rangle_{\vartheta^{t}(q)} =\\
			= \frac{1}{2}\sum_{ \sigma \in \mathbb{S}_{m} } (-1)^{\sigma} \sum_{j=1}^{m} \frac{d}{dt}\left[\left\langle \Xi^{t}(q,\xi_{\sigma(j)}), \Xi^{t}(q,\eta_{j}) \right\rangle_{\vartheta^{t}(q)} \right] \cdot \prod_{l \not= j} \left\langle \xi_{\sigma(l)}, \eta_{l} \right\rangle_{q} =\\
			= \sum_{j=1}^{m} \sum_{ \sigma \in \mathbb{S}_{m} } (-1)^{\sigma} B_{\mathfrak{n}}(q; \xi_{\sigma(j)}, \eta_{j}) \cdot \prod_{l\not= j} \left\langle \xi_{\sigma(l)}, \eta_{l} \right\rangle_{q}.
		\end{split}
	\end{equation}
	
	On the other hand, we have
	\begin{equation}
		\label{EQ: RepresentationAssociatedMetric2}
		\begin{split}
			m! \cdot \left\langle \xi_{1} \wedge \cdots \wedge \xi_{m}, T^{[\wedge m]}_{\mathfrak{n}}(q)(\eta_{1} \wedge \cdots \wedge \eta_{m}) \right\rangle_{q} =\\
			= m! \cdot \sum_{j=1}^{m} \left\langle \xi_{1} \wedge \cdots \wedge \xi_{m}, \eta_{1} \wedge \cdots \wedge T_{\mathfrak{n}}(q) \eta_{j} \wedge \cdots \wedge \eta_{m} \right\rangle_{q}  =\\
			=\sum_{j=1}^{m} \det \begin{pmatrix}
				\langle \xi_{1}, \eta_{1} \rangle_{q} && \cdots && \langle \xi_{1}, T_{\mathfrak{n}}\eta_{j}\rangle_{q} && \cdots && \langle \xi_{1}, \eta_{m} \rangle_{q}\\
				\vdots && \ &&\vdots && \ && \vdots\\
				\langle \xi_{m}, \eta_{1}\rangle_{q} && \cdots && \langle \xi_{m}, T_{\mathfrak{n}}\eta_{j}\rangle_{q} && \cdots && \langle \xi_{m}, \eta_{m} \rangle_{q}
			\end{pmatrix} =\\
			= \sum_{j=1}^{m} \sum_{ \sigma \in \mathbb{S}_{m} } (-1)^{\sigma} \langle \xi_{\sigma(j)}, T_{\mathfrak{n}}\eta_{j} \rangle_{q}  \cdot \prod_{l \not= j} \left\langle \xi_{\sigma(l)}, \eta_{l} \right\rangle_{q}.
		\end{split}
	\end{equation}
	From this, \eqref{EQ: RepresentationAssociatedMetric1}, and since $B_{\mathfrak{n}}(q; \xi_{\sigma(j)}, \eta_{j}) = \langle \xi_{\sigma(j)}, T_{\mathfrak{n}}\eta_{j} \rangle_{q}$, we obtain that $T^{[\wedge m]}_{\mathfrak{n}}(q)$ is a representation of $B_{\mathfrak{n}^{\wedge m}}$.
\end{proof}

Now note that
\begin{equation}
	\label{EQ: InfinitesimalGEviaCurvatureForm}
	\alpha_{\mathfrak{n}}(q; \xi) = \frac{B_{\mathfrak{n}}(q; \xi, \xi)}{\langle \xi, \xi \rangle_{q}} \qquad \text{for any} \quad 0\not=\xi \in \mathcal{N}_{\mathfrak{n}}(q).
\end{equation}
Thus, the optimization of $\alpha_{\mathfrak{n}}(q; \xi)$ over $\xi$ is related to the optimization of $B_{\mathfrak{n}}(q; \xi,\xi)$ on the unit ball with respect to $\langle \cdot, \cdot \rangle_{q}$.

We call an operator $A_{\mathfrak{n}}(q)$ in $\mathbb{H}$ a \textit{mean representation} of $B_{\mathfrak{n}}(q)$ if its domain $\mathcal{D}(A_{\mathfrak{n}}(q))$ contains $\mathcal{N}_{\mathfrak{n}}(q)$ and
\begin{equation}
	\label{EQ: SymmetricRepresentationDefinition}
	B_{\mathfrak{n}}(q;\xi,\xi) = \operatorname{Re}\langle \xi, A_{\mathfrak{n}}(q)\xi  \rangle_{q} \qquad \text{for any} \quad \xi \in \mathcal{N}_{\mathfrak{n}}(q).
\end{equation}
Equivalently, we have
\begin{equation}
	\label{EQ: SymmetricRepresentationDefinitionSymmetric}
	B_{\mathfrak{n}}(q;\xi,\eta) = \frac{1}{2} \left( \langle A_{\mathfrak{n}}(q)\xi, \eta \rangle_{q} + \langle \xi, A_{\mathfrak{n}}(q)\eta \rangle_{q}  \right) \qquad \text{for any} \quad \xi, \eta \in \mathcal{N}_{\mathfrak{n}}(q)
\end{equation}
that justifies the word ``mean'' in the name. In particular, any representation is a mean representation. However, even in finite dimensions, there are different mean representations.

If $A_{\mathfrak{n}}(q)$ is given for all $q \in \mathcal{Q}$, we denote the corresponding family by $A_{\mathfrak{n}}$ and call it a \textit{mean representation} of $B_{\mathfrak{n}}$.

\begin{theorem}[Mean representations in associated metrics]
	\label{TH: SymmetricRepresentationsInAssociatedMetrics}
	Let $A_{\mathfrak{n}}(q)$ be a mean representation of $B_{\mathfrak{n}}(q)$. Then $A^{[\wedge m]}_{\mathfrak{n}}(q) \coloneq (A_{\mathfrak{n}}(q))^{[\wedge m]}$ is a mean representation of $B_{\mathfrak{n}^{\wedge m}}(q)$. Furthermore, for all $\xi_{1},\ldots,\xi_{m} \in \mathcal{N}_{\mathfrak{n}}(q)$, we have
	\begin{equation}
		\label{EQ: SymmetricRepresentationTraceFormula}
		\langle A_{\mathfrak{n}^{\wedge m}}(q)(\xi_{1} \wedge \cdots \wedge \xi_{m}), \xi_{1} \wedge \cdots \wedge \xi_{m} \rangle_{q} = \mathfrak{n}^{\wedge m}_{q}(\xi_{1} \wedge \cdots \wedge \xi_{m})^{2} \cdot \operatorname{Tr}( \Pi \circ A_{\mathfrak{n}}(q) \circ \Pi), 
	\end{equation}
	where $\Pi$ is the orthogonal with respect to $\langle \cdot, \cdot \rangle_{q}$ projector in $\mathbb{H}$ onto $\operatorname{Span}\{ \xi_{1}, \ldots, \xi_{m} \}$ and $\operatorname{Tr}$ denotes the trace functional.
\end{theorem}
\begin{proof}
	We need to show an analog of \eqref{EQ: SymmetricRepresentationDefinitionSymmetric} for $n^{\wedge m}$. Due to linearity, it is sufficient to consider decomposable tensors $\xi_{1} \wedge \cdots \wedge \xi_{m}$ and $\eta_{1} \wedge \cdots \wedge \eta_{m}$. Similarly to \eqref{EQ: RepresentationAssociatedMetric2}, we have the relations
	\begin{equation}
		\label{EQ: SymmetricRepresentationTh1}
		\begin{split}
			m! \cdot \left\langle \xi_{1} \wedge \cdots \wedge \xi_{m}, A^{[\wedge m]}_{\mathfrak{n}}(q)(\eta_{1} \wedge \cdots \wedge \eta_{m}) \right\rangle_{q} =\\
			= \sum_{\sigma \in \mathbb{S}_{m}} (-1)^{\sigma} \sum_{j=1}^{m} \langle \xi_{\sigma(j)}, A_{\mathfrak{n}}\eta_{j} \rangle_{q}  \cdot \prod_{l \not= j} \left\langle \xi_{\sigma(l)}, \eta_{l} \right\rangle_{q}.
		\end{split}
	\end{equation}
	and (here $\sigma$ in the first appearance is arbitrary)
	\begin{equation}
		\label{EQ: SymmetricRepresentationTh2}
		\begin{split}
			m! \cdot \left\langle A^{[\wedge m]}_{\mathfrak{n}}(q)(\xi_{1} \wedge \cdots \wedge \xi_{m}), \eta_{1} \wedge \cdots \wedge \eta_{m} \right\rangle_{q} =\\
			= \sum_{\tau \in \mathbb{S}_{m}} (-1)^{\tau} \sum_{j=1}^{m} \langle A_{\mathfrak{n}}\xi_{j}, \eta_{\tau(j)} \rangle_{q}  \cdot \prod_{l \not= j} \left\langle \xi_{l}, \eta_{\tau(l)} \right\rangle_{q} =\\
			= \sum_{\tau \in \mathbb{S}_{m}} (-1)^{\tau} \sum_{j=1}^{m} \langle A_{\mathfrak{n}}\xi_{\sigma(j)}, \eta_{\tau(\sigma(j))} \rangle_{q}  \cdot \prod_{l \not= \sigma(j)} \left\langle \xi_{l}, \eta_{\tau(l)} \right\rangle_{q} =\\
			\overset{\tau = \sigma^{-1}}{=} \sum_{\sigma \in \mathbb{S}_{m}} (-1)^{\sigma} \sum_{j=1}^{m} \langle A_{\mathfrak{n}}\xi_{\sigma(j)}, \eta_{j} \rangle_{q}  \cdot \prod_{l \not= j} \left\langle \xi_{\sigma(l)}, \eta_{l} \right\rangle_{q}
			,
		\end{split}
	\end{equation}
	where we also used that $\prod_{l \not= \sigma(j)} \left\langle \xi_{l}, \eta_{\tau(l)} \right\rangle_{q} = \prod_{l \not= j} \left\langle \xi_{\sigma(l)}, \eta_{\tau(\sigma(l))} \right\rangle_{q}$. Now summing \eqref{EQ: SymmetricRepresentationTh1} and \eqref{EQ: SymmetricRepresentationTh2} and using \eqref{EQ: SymmetricRepresentationDefinitionSymmetric} along with \eqref{EQ: RepresentationAssociatedMetric1}, we obtain that $A^{[\wedge m]}_{\mathfrak{n}}(q)$ is a mean representation of $B_{\mathfrak{n}^{\wedge m}}(q)$.
	
	For \eqref{EQ: SymmetricRepresentationTraceFormula}, it is sufficient to consider orthonormal families of $\xi_{j}$. Putting these into \eqref{EQ: SymmetricRepresentationTh2} and noting that the product vanishes for all permutations except the identity, we obtain
	\begin{equation}
		\begin{split}
			m! \cdot \left\langle A^{[\wedge m]}_{\mathfrak{n}}(q)(\xi_{1} \wedge \cdots \wedge \xi_{m}), \xi_{1} \wedge \cdots \wedge \xi_{m} \right\rangle_{q} =\\
			= \sum_{j=1}^{m} \langle A_{\mathfrak{n}} \xi_{j}, \xi_{j} \rangle_{q} = \sum_{j=1}^{m} \langle \Pi \circ A_{\mathfrak{n}} \circ \Pi \xi_{j}, \xi_{j} \rangle_{q} = \operatorname{Tr}(\Pi \circ A_{\mathfrak{n}} \circ \Pi).
		\end{split}
	\end{equation}
	Since $1/m!$ is the squared norm of $\xi_{1} \wedge \cdots \wedge \xi_{m}$, this shows the required.
\end{proof}

From this, we immediately have the generalized Liouville trace formula.
\begin{theorem}
	\label{TH: GeneralizedLiouvilleTraceFormula}
	Let $\mathfrak{n}$ be as above, and suppose $A_{\mathfrak{n}}$ is a mean representation of $B_{\mathfrak{n}}$. Then for all $m=1,2,\ldots$, $q \in \mathcal{Q}$, $\xi_{1},\ldots,\xi_{m} \in \mathcal{N}_{\mathfrak{n}}(q)$, and $t \geq 0$ such that $\Xi^{t}_{m}(q,\xi_{1} \wedge \cdots \wedge \xi_{m}) \not= 0$, we have the \textit{Liouville trace formula}
	\begin{equation}
		\label{EQ: GeneralizedTraceFormula}
		\begin{split}
			\mathfrak{n}^{\wedge m}_{\vartheta^{t}(q)}(\Xi^{t}_{m}(q, \xi_{1} \wedge \cdots \wedge \xi_{m}) = \mathfrak{n}^{\wedge m}_{q}(\xi_{1} \wedge \cdots \wedge \xi_{m}) \times \\
			\times \exp\left(\int_{0}^{t} \operatorname{Re} \operatorname{Tr} (\Pi(s) \circ A_{\mathfrak{n}}(\vartheta^{s}(q)) \circ \Pi(s)) ds \right),
		\end{split}
	\end{equation}
	where $\Pi(s)=\Pi(s; q; \xi_{1},\ldots,\xi_{m})$ is the orthogonal with respect to $\langle \cdot, \cdot \rangle_{\vartheta^{s}(q)}$ projector in $\mathbb{H}$ onto $\operatorname{Span}\left\{ \Xi^{s}(q,\xi_{1}), \ldots, \Xi^{s}(q,\xi_{m}) \right\}$.
\end{theorem}
\begin{proof}
	We apply the Growth Formula \eqref{EQ: SeminormsGrowthViaInfExponentsEq} to $\Xi_{m}$ in the metric $\mathfrak{n}^{\wedge m}$ for the initial vector $\xi_{1} \wedge \cdots \wedge \xi_{m}$. Then the conclusion follows from \eqref{EQ: InfinitesimalGEviaCurvatureForm} and \eqref{EQ: SymmetricRepresentationDefinition} for $\mathfrak{n}^{\wedge m}$ and \eqref{EQ: SymmetricRepresentationTraceFormula}.
\end{proof}
\begin{remark}
	In the paper of Savostianov and Zelik \cite{SavostianovZelik2016}, the Liouville trace formula is considered for cocycles with bounded generators.
\end{remark}

For all $m=1,2,\ldots$ and $q \in \mathcal{Q}$, we define the \textit{$m$-trace exponent} of $\Xi$ over $q$ as the quantity
\begin{equation}
	\label{EQ: TraceExponentDefinition}
	\alpha^{tr}_{\mathfrak{n}^{\wedge m}}(q) = \alpha^{tr}_{\mathfrak{n}^{\wedge m}}(q;\Xi) \coloneq \sup_{\xi_{1}, \ldots, \xi_{m}} \alpha_{\mathfrak{n}^{\wedge m}}(q; \xi_{1} \wedge \cdots \wedge \xi_{m}),
\end{equation}
where the supremum is taken over all linearly independent vectors $\xi_{1},\ldots,\xi_{m} \in \mathcal{N}_{\mathfrak{n}}(q)$.

\begin{theorem}
	Let $\mathfrak{n}$ be as above, and suppose $A_{\mathfrak{n}}(q)$ is a mean representation of $B_{\mathfrak{n}}(q)$. Then
	\begin{equation}
		\label{EQ: TraceExponentViaTraceOfSymmetricRepresentation}
		\alpha^{tr}_{\mathfrak{n}^{\wedge m}}(q) = \sup_{\mathbb{L} \subset \mathcal{N}_{\mathfrak{n}}(q)} \operatorname{Re}\operatorname{Tr}( \Pi_{\mathbb{L}} \circ A_{\mathfrak{n}}(q) \circ \Pi_{\mathbb{L}} ),
	\end{equation}
	where the supremum is taken over all $m$-dimensional subspaces $\mathbb{L}$ of $\mathcal{N}_{\mathfrak{n}}(q)$, and $\Pi_{\mathbb{L}}$ is the orthogonal with respect to $\langle \cdot, \cdot \rangle_{q}$ projector in $\mathbb{H}$ onto $\mathbb{L}$. In other words,
	\begin{equation}
		\label{EQ: TraceExponentViaTraceNumbers}
		\alpha^{tr}_{\mathfrak{n}^{\wedge m}}(q) = \sum_{j=1}^{m} \beta_{j}(A_{\mathfrak{n}}(q);\mathcal{N}(q)),
	\end{equation}
	where $\beta_{1}(A_{\mathfrak{n}}(q); \mathcal{N}(q)) \geq \beta_{2}(A_{\mathfrak{n}}(q);\mathcal{N}(q)) \geq \ldots$ are the trace numbers of $A_{\mathfrak{n}}(q)$ as an operator defined on $\mathcal{N}(q)$ in the Hilbert space $\mathbb{H}$ endowed with $\langle \cdot, \cdot \rangle_{q}$, see \eqref{EQ: TraceNumbersDef}.
	
	Furthermore, if $s \mapsto \alpha^{tr}_{\mathfrak{n}^{\wedge m}}(\vartheta^{s}(q))$ is integrable on a segment $[0,T]$, then\footnote{By $\|\Xi^{t}(q,\cdot)\|_{\mathfrak{n}}$ we denote the norm of $\Xi^{t}(q,\cdot)$ as an operator from $(\mathbb{H},\mathfrak{n}_{q})$ to $(\mathbb{H},\mathfrak{n}_{\vartheta^{t}(q)})$, and similarly for its multiplicative compounds.}
	\begin{equation}
		\label{EQ: CocycleNormEstimateViaTraceExponent}
		\| \Xi^{t}_{m}(q, \cdot) \|_{\mathfrak{n}^{\wedge m}} \leq \exp\left( \int_{0}^{t} \alpha^{tr}_{\mathfrak{n}^{\wedge m}}(\vartheta^{s}(q))ds\right) \qquad \text{for any} \quad t \in [0,T].
	\end{equation}
\end{theorem}
\begin{proof}
	We immediately obtain \eqref{EQ: TraceExponentViaTraceOfSymmetricRepresentation} by using analogs of \eqref{EQ: InfinitesimalGEviaCurvatureForm} and \eqref{EQ: SymmetricRepresentationDefinition} for the metric $\mathfrak{n}^{\wedge m}$ and then applying Theorem \ref{TH: SymmetricRepresentationsInAssociatedMetrics} and \eqref{EQ: SymmetricRepresentationTraceFormula}. Then, \eqref{EQ: TraceExponentViaTraceNumbers} is just a reformulation of \eqref{EQ: TraceExponentViaTraceOfSymmetricRepresentation} in terms of the trace numbers defined in \eqref{EQ: TraceNumbersDef}.
	
	Finally, \eqref{EQ: CocycleNormEstimateViaTraceExponent} follows from the Liouville trace formula \eqref{EQ: GeneralizedTraceFormula} and Lemma \ref{LEM: TemamLemmaWedgeNormComputation}.
\end{proof}

For estimating the Lyapunov dimension defined by \eqref{EQ: LyapunovDimensionDefinition}, we need to deal with the averaged function of singular values $\omega_{d}(\Xi)$ for non-integer $d$, especially if we want to obtain sharp estimates. For the moment, we suppose that $\mathfrak{n}$ is equivalent to the standard metric in $\mathbb{H}$.

For $d = m + \theta$ with integer $m \geq 0$ and $\theta \in [0,1]$, the function of singular values of the cocycle mappings in the metric $\mathfrak{n}$ is given by
\begin{equation}
	\label{EQ: OmegaInGeneralMetric}
	\omega^{(\mathfrak{n})}_{d}(\Xi^{t}(q,\cdot)) = \| \Xi^{t}_{m}(q,\cdot) \|^{1-\theta}_{\mathfrak{n}^{\wedge m}} \cdot \| \Xi^{t}_{m+1}(q,\cdot) \|^{\theta}_{\mathfrak{n}^{\wedge (m+1)}},
\end{equation}
where for $m=0$ we set the left multiplier to be $1$. It is not difficult to show that
\begin{equation}
	\omega_{d}(\Xi) = \lim_{t \to +\infty} \left( \sup_{q \in \mathcal{Q}} \omega^{(\mathfrak{n})}_{d}(\Xi^{t}(q,\cdot))  \right)^{1/t}.
\end{equation}
In other words, $\omega_{d}(\Xi)$ is independent of the choice of an equivalent metric $\mathfrak{n}$.

\begin{theorem}
	\label{TH: AveragedOmegaDEstimateViaTraceNumbers}
	Let $\mathfrak{n}$, as above, be equivalent to the standard metric in $\mathbb{H}$, and consider $d = m + \theta$ with integer $m \geq 0$ and $\theta \in [0,1]$. Suppose $A_{\mathfrak{n}}$ is a mean representation of $B_{\mathfrak{n}}$. Then
	\begin{equation}
		\label{EQ: OmegaEstimateViaMaxOfBetas}
		\ln \omega_{d}(\Xi) \leq \sup_{q \in \mathcal{Q}} \left[ \sum_{j=1}^{m}\beta_{j}(A_{\mathfrak{n}}(q)) + \theta\beta_{m+1}(A_{\mathfrak{n}}(q)) \right].
	\end{equation}
	Furthermore, if for all $q \in \mathcal{Q}$ and $j \in \{1,\ldots,m+1\}$, the function $s \mapsto \beta_{j}(A_{\mathfrak{n}}(\vartheta^{s}(q)))$ is integrable on finite time intervals, then
	\begin{equation}
		\label{EQ: OmegaEstimateViaTimeAveragedMaxOfBetas}
		\begin{split}
			\ln\omega_{d}(\Xi)
			\leq \lim_{t \to +\infty}\frac{1}{t}\sup_{q \in \mathcal{Q}}\int_{0}^{t} \left[\sum_{j=1}^{m}\beta_{j}(A_{\mathfrak{n}}(\vartheta^{s}(q))) + \theta\beta_{m+1}(A_{\mathfrak{n}}(\vartheta^{s}(q)))\right] ds.
		\end{split}
	\end{equation}
\end{theorem}
\begin{proof}
	We estimate $\| \Xi^{t}_{m}(q,\cdot) \|^{1-\theta}_{\mathfrak{n}^{\wedge m}} \cdot \| \Xi^{t}_{m+1}(q,\cdot) \|^{\theta}_{\mathfrak{n}^{\wedge (m+1)}}$ by applying the Liouville trace formula \eqref{EQ: GeneralizedTraceFormula} to each multiplier and then using \eqref{EQ: TraceExponentViaTraceNumbers} and also \eqref{EQ: CocycleNormEstimateViaTraceExponent} (for the second estimate). This and
	\begin{equation}
		\ln \omega_{d}(\Xi) = \lim_{t \to +\infty}\frac{1}{t}\sup_{q \in \mathcal{Q}}\ln\left( \| \Xi^{t}_{m}(q,\cdot) \|^{1-\theta}_{\mathfrak{n}^{\wedge m}} \cdot \| \Xi^{t}_{m+1}(q,\cdot) \|^{\theta}_{\mathfrak{n}^{\wedge (m+1)}} \right)
	\end{equation}
	lead to both conclusions.
\end{proof}

It is clear that $\alpha^{tr}_{\mathfrak{n}^{\wedge m}}(q) \leq \alpha^{+}_{\mathfrak{n}^{\wedge m}}(q)$. However, for the converse inequality, which provides an infinitesimal analog of Lemma \ref{LEM: TemamLemmaWedgeNormComputation}, we need some regularity assumptions on $\alpha^{tr}_{\mathfrak{n}^{\wedge m}}(q)$ over trajectories of the base system as follows.
\begin{theorem}
	\label{TH: RegularityOfTraceExponentAndCoincidence}
	Let $\mathfrak{n}$ be as above. Then we have
	\begin{equation}
		\label{EQ: MaximizedExponentViaMaximizedTraceExponent}
		\alpha^{+}_{\mathfrak{n}^{\wedge m}}(\Xi_{m}) = \sup_{q \in \mathcal{Q}} \alpha^{tr}_{\mathfrak{n}^{\wedge m}}(q).
	\end{equation}
	Furthermore, if for some $q \in \mathcal{Q}$, the function $s \mapsto \alpha^{tr}_{\mathfrak{n}^{\wedge m}}(\vartheta^{s}(q))$ is integrable on $[0,\varepsilon]$ for some $\varepsilon>0$ and satisfies
	\begin{equation}
		\label{EQ: TraceExponentRegularityCond}
		\liminf_{t \to 0+}\frac{1}{t}\int_{0}^{t}\alpha^{tr}_{\mathfrak{n}^{\wedge m}}(\vartheta^{s}(q))ds \leq \alpha^{tr}_{\mathfrak{n}^{\wedge m}}(q),
	\end{equation}
	then $\alpha^{tr}_{\mathfrak{n}^{\wedge m}}(q) = \alpha^{+}_{\mathfrak{n}^{\wedge m}}(q)$.
\end{theorem}
\begin{proof}
	For any $\mathcal{P} \in \mathcal{N}_{\mathfrak{n}^{\wedge m}}(q)$ with $\mathfrak{n}^{\wedge m}_{q}(\mathcal{P}) = 1$, from the Growth Formula \eqref{EQ: SeminormsGrowthViaInfExponentsEq}, we have
	\begin{equation}
	\label{EQ: MaximizedTraceExponentEstimatesGrowth}
	\begin{split}
		\exp\left( \int_{0}^{t} \alpha_{\mathfrak{n}^{\wedge m}}(\vartheta^{s}(q); \Xi^{s}_{m}(q,\mathcal{P}))ds \right) = \mathfrak{n}^{\wedge m}_{\vartheta^{t}(q)}(\Xi^{t}_{\mathfrak{m}}(q,\mathcal{P})) \leq\\
		\leq \| \Xi^{t}_{m}(q, \cdot) \|_{\mathfrak{n}^{\wedge m}} \leq \exp\left(t \sup_{q \in \mathcal{Q}} \alpha^{tr}_{\mathfrak{n}^{\wedge m}}(q) \right).
	\end{split}
	\end{equation}
	By taking logarithms, dividing by $t$ and letting $t$ go to zero, we obtain
	\begin{equation}
		\alpha_{\mathfrak{n}^{\wedge m}}(q; \mathcal{P}) \leq \sup_{q \in \mathcal{Q}} \alpha^{tr}_{\mathfrak{n}^{\wedge m}}(q).
	\end{equation}
	Now taking the supremum over $\mathcal{P}$ and then over $q \in \mathcal{Q}$ gives \eqref{EQ: MaximizedExponentViaMaximizedTraceExponent}.
	
	Within \eqref{EQ: TraceExponentRegularityCond}, we act as in \eqref{EQ: MaximizedTraceExponentEstimatesGrowth} except for the last inequality, where we now use the integral of $\alpha^{tr}_{\mathfrak{n}^{\wedge m}}(\vartheta^{s}(q))$ over $[0,t]$ with $t \in (0,\varepsilon]$. This gives
	\begin{equation}
		\begin{split}
			\int_{0}^{t} \alpha_{\mathfrak{n}^{\wedge m}}(\vartheta^{s}(q); \Xi^{s}_{m}(q,\mathcal{P}) ) \leq \int_{0}^{t} \alpha^{tr}_{\mathfrak{n}^{\wedge m}}(\vartheta^{s}(q))ds.
		\end{split}
	\end{equation}
	Dividing both sides by $t$, then taking the limit inferior as $t$ goes to zero, and using \eqref{EQ: TraceExponentRegularityCond}, we obtain $\alpha_{\mathfrak{n}^{\wedge m}}(q; \mathcal{P}) \leq \alpha^{tr}_{\mathfrak{n}^{\wedge m}}(q)$. Now, taking the supremum over all $\mathcal{P}$ yields $\alpha^{+}_{\mathfrak{n}^{\wedge m}}(q) \leq \alpha^{tr}_{\mathfrak{n}^{\wedge m}}(q)$.
\end{proof}

\begin{remark}
	Mean representations of curvature forms essentially arise from differentiation by directly applying the definition \eqref{EQ: CurvatureFormDefinition}. Consequently, \eqref{EQ: SymmetricRepresentationDefinitionSymmetric} simply reflects the differentiation rule for symmetric forms. Although these operators often appear without explicit names, they serve as a foundation for deriving dimension estimates through particular cases of the Liouville trace formula \eqref{EQ: GeneralizedTraceFormula} in numerous works. There are two main reasons for this lack of naming. First, most studies in infinite-dimensional settings involve constant standard metrics, where the operators correspond to those arising in linearized equations, and the associated variational problems can be analyzed straightforwardly due to sufficient regularity\footnote{In many cases, computing the exact $m$-trace exponents is challenging; typically, they are estimated from above by treating $A_{\mathfrak{n}}(q)$ as a perturbation of a constant operator $A$, which is usually self-adjoint, such as the Laplace \cite{Temam1997} or Stokes operators \cite{ChepyzhovIlyin2004}.}. Second, on finite-dimensional manifolds with general smooth metrics, mean representations are expressed in terms of covariant derivatives compatible with the metric. In both contexts, there are no significant complications. However, in infinite dimensions—even for constant but varying metrics—the resolution of variational problems may critically depend on the metric (see Section \ref{SEC: DelayEqsSymmetrization}). Therefore, it becomes necessary to explicitly identify these challenges.
\end{remark}

In light of the above results, it remains to discuss the computation of the $m$-trace exponents $\alpha^{tr}_{\mathfrak{n}^{m}}(q)$. According to \eqref{EQ: TraceExponentViaTraceNumbers}, this problem reduces to computing the trace numbers $\beta_{k}(A_{\mathfrak{n}}(q))$ of a mean representation $A_{\mathfrak{n}}(q)$ of $B_{\mathfrak{n}}(q)$. Note that the supremum in \eqref{EQ: TraceExponentDefinition} need not be attained on vectors from the domain $\mathcal{N}_{\mathfrak{n}}(q)$ nor from the domain $\mathcal{D}(A_{\mathfrak{n}}(q))$.

In finite dimensions, the trace numbers of $A_{\mathfrak{n}}(q)$ correspond exactly to the arranged eigenvalues of the additive symmetrization $S_{0}(A_{\mathfrak{n}}(q)) \coloneq (A_{\mathfrak{n}}(q) + A^{*}_{\mathfrak{n}}(q))/2$, where the adjoint is taken with respect to $\langle \cdot, \cdot \rangle_{q}$. In infinite dimensions, there is also a strong relationship with these operators; however, many new phenomena may arise. We refer to Appendix \ref{SEC: SymmetrizationOfOperators} for discussions on additive symmetrization and to Section \ref{SEC: DelayEqsSymmetrization} for an illustrative example of these phenomena using delay operators.

%% file: LiouvFormula.tex
\section{Infinitesimal growth exponents in the conformal class of a constant metric}
\label{SEC: LiouvilleTraceFormulaAndSymProc}
We continue to work with a uniformly continuous linear cocycle $\Xi$ in a separable real or complex Hilbert space $\mathbb{H}$ over a semiflow $(\mathcal{Q},\vartheta)$ on a complete metric space $\mathcal{Q}$.

Here we aim to describe infinitesimal exponents $\alpha_{\mathfrak{n}^{\wedge m}}(q; \xi)$ for metrics $\mathfrak{n}$ over $\mathcal{Q}$ belonging to the conformal class of the standard metric in $\mathbb{H}$ and constructed with the aid of Lyapunov-like functions. In applications, this constitutes the core of the Leonov method \cite{LeonovBoi1992, Kuznetsov2016}. 

Let us consider a bounded scalar function $V \colon \mathcal{Q} \to \mathbb{R}$ such that
\begin{description}[before=\let\makelabel\descriptionlabel]
	\item[\textbf{(V1)}\label{DESC: LyapFuncV1}] for any $T>0$, the mapping $[0,T] \ni t \mapsto V(\vartheta^{t}(q)) \in \mathbb{R}$ is absolutely continuous;
	\item[\textbf{(V2)}\label{DESC: LyapFuncV2}] for any $q \in \mathcal{Q}$, there exists the right derivative of $V$ along the trajectory of $q$, i.e.,
	\begin{equation}
		\dot{V}(q) \coloneq \lim_{t \to 0+}\frac{V(\vartheta^{t}(q))-V(q)}{t}.
	\end{equation}
\end{description}

For a fixed integer $m \geq 1$, consider the metric $\mathfrak{n}=\{ \mathfrak{n}_{q} \}_{q \in \mathcal{Q}}$ in $\mathbb{H}$ which is given for any $q \in \mathcal{Q}$ by $\mathfrak{n}_{q}(\cdot) \coloneq e^{\frac{1}{m}V(q)}|\cdot|_{\mathbb{H}}$. Here $|\cdot|_{\mathbb{H}}$ is the norm in $\mathbb{H}$ induced by the standard inner product $\langle \cdot,\cdot \rangle_{\mathbb{H}}$. This metric induces a metric $\mathfrak{n}^{\wedge m} = \{ \mathfrak{n}^{\wedge m}_{q} \}_{q \in \mathcal{Q}}$ on $\mathbb{H}^{\wedge m}$ defined for all $q \in \mathcal{Q}$ and $\xi_{1},\ldots,\xi_{m} \in \mathbb{H}$, by
\begin{equation}
	\label{EQ: AdaptedMetricOnExteriorProductViaLyapFunc}
	\begin{split}
		\mathfrak{n}^{\wedge m}_{q}(\xi_{1} \wedge \cdots \wedge \xi_{m}) \coloneq& \frac{1}{\sqrt{m!}}\left(\operatorname{det}[ e^{\frac{2}{m}V(q)}\langle\xi_{i},\xi_{j}\rangle_{\mathbb{H}} ]_{i,j =1}^{m}\right)^{1/2} =\\=& e^{V(q)} |\xi_{1} \wedge \cdots \wedge \xi_{m}|_{\mathbb{H}^{\wedge m}}.
	\end{split}
\end{equation}
Since $V$ is bounded, the just introduced metric $\mathfrak{n}^{\wedge m}$ is uniformly equivalent to $|\cdot|_{\mathbb{H}^{\wedge m}}$.

To describe infinitesimal growth exponents, we will assume some differentiability of trajectories. Namely, let us assume that for any $q \in \mathcal{Q}$, the set
\begin{equation}
	\label{EQ: DomainAqDefinitionCocycleLiovForm}
	\mathcal{D}(A(q)) \coloneq \{ \xi \in \mathbb{H} \ | \ [0,\infty) \ni t \mapsto \Xi^{t}(q,\xi) \text{ is } C^{1}\text{-differentiable} \}
\end{equation}
is dense in $\mathbb{H}$ (this is easily verified in practice). To justify the notation, note that on $\mathcal{D}(A(q))$ there is an operator $A(q)$ given by
\begin{equation}
	\label{EQ: CocycleGeneratorDefinitionInfExps}
	A(q)\xi \coloneq \lim_{t \to 0+} \frac{\Xi^{t}(q,\xi) - \xi}{t} \qquad \text{for any} \quad \xi \in \mathcal{D}(A(q)).
\end{equation}
By definition and the cocycle property, we have $\Xi^{t}(q,\mathcal{D}(A(q))) \subset \mathcal{D}(A(\vartheta^{t}(q)))$ for all $q \in \mathcal{Q}$ and $t \geq 0$.

By straightforward differentiation, it is not difficult to see that $\Xi$ admits infinitesimal growth exponents with respect to $\mathfrak{n}$ with domains $\mathcal{N}_{\mathfrak{n}}(q) = \mathcal{D}(A(q))$, in which the curvature form $B_{\mathfrak{n}}$ admits a mean representation $A_{\mathfrak{n}}(q)$ given by
\begin{equation}
	\label{EQ: MeanRepresentationConformalClass}
	A_{\mathfrak{n}}(q) = A(q) + \frac{1}{m} \dot{V}(q) \operatorname{Id}_{\mathbb{H}} \qquad \text{for any} \quad q \in \mathcal{Q}.
\end{equation}

By Theorem \ref{TH: SymmetricRepresentationsInAssociatedMetrics} and \eqref{EQ: MeanRepresentationConformalClass}, for any linearly independent vectors $\xi_{1},\ldots,\xi_{m} \in \mathcal{D}(A(q))$, we have
\begin{equation}
	\alpha_{\mathfrak{n}^{\wedge m}}(q;\xi_{1} \wedge \cdots \wedge \xi_{m}) = \operatorname{Re}\operatorname{Tr}\left( \Pi \circ A(q) \circ \Pi \right) + \dot{V}(q),
\end{equation}
where $\Pi$ is the orthogonal projector in $\mathbb{H}$ onto $\operatorname{Span}\{\xi_{1}, \ldots, \xi_{m}\}$.

In terms of the trace numbers $\beta_{1}(A(q)) \geq \beta_{2}(A(q)) \geq \ldots$ of $A(q)$, we have
\begin{equation}
	\label{EQ: LeonovContextTraceExponentsViaTraceNumbers}
	\alpha^{tr}_{\mathfrak{n}^{\wedge m}}(q) = \dot{V}(q) + \sum_{j=1}^{m} \beta_{j}(A(q)) \qquad \text{for any} \quad q \in \mathcal{Q}.
\end{equation}

We will assume that each $A(q)$ admits a symmetrization as follows.
\begin{description}[before=\let\makelabel\descriptionlabel]
	\item[\textbf{(SYM)}\label{DESC: SYM}] For any $q \in \mathcal{Q}$, the operator $A(q)$ defined by \eqref{EQ: CocycleGeneratorDefinitionInfExps} admits a symmetrization $S(q)$, i.e., $S(q)$ is a self-adjoint operator in $\mathbb{H}$ satisfying $\beta_{k}(A(q)) = \beta_{k}(S(q))$ for any $k=1,2,\ldots$.
\end{description}
Under \textbf{(SYM)}, the \textit{symmetrization procedure} is applicable in \eqref{EQ: LeonovContextTraceExponentsViaTraceNumbers} and yields
\begin{equation}
	\label{EQ: LeonovContextTraceExponentsViaTraceNumbersSelfAdj}
	\alpha^{tr}_{\mathfrak{n}^{\wedge m}}(q) = \dot{V}(q) + \sum_{j=1}^{m}\beta_{j}(S(q)).
\end{equation}
By Theorem \ref{TH: ComputationSwedgeSpectralBound}, the trace numbers of $S(q)$ coincide with its characteristic numbers defined in \eqref{EQ: CharacteristicNumbersS}.

We collect the above considerations into the following theorem.
\begin{theorem}
	\label{TH: LiouvilleSymmetrizationSummary}
	Let \nameref{DESC: LyapFuncV1}, \nameref{DESC: LyapFuncV2}, and \nameref{DESC: SYM} be satisfied. For any $q \in \mathcal{Q}$, consider the largest $m$ characteristic numbers $\alpha_{1}(q) \geq \alpha_{2}(q) \geq \cdots \geq \alpha_{m}(q)$ of $S(q)$. Then for the metric $\mathfrak{n}^{\wedge m}$ given by \eqref{EQ: AdaptedMetricOnExteriorProductViaLyapFunc}, we have
	\begin{equation}
		\label{EQ: TraceMaximizationComputationFormula}
		\alpha^{tr}_{\mathfrak{n}^{\wedge m}}(q) = \dot{V}(q) + \sum_{j=1}^{m} \alpha_{j}(q).
	\end{equation}
	In particular,
	\begin{equation}
		\label{EQ: ComputationMaxExponent}
		\begin{split}
			\alpha^{+}_{\mathfrak{n}^{\wedge m}}(\Xi_{m}) = \sup_{q \in \mathcal{Q}} \left[\dot{V}(q) + \sum_{j=1}^{m} \alpha_{j}(q) \right],
		\end{split}
	\end{equation}
	and for any $d = m + \theta$ with integer $m \geq 0$ and $\theta \in [0,1]$, we have
	\begin{equation}
		\label{EQ: LeonovContextOmegaDViaAlphas}
		\ln\omega_{d}(\Xi) \leq \sup_{q \in \mathcal{Q}}\left[ \sum_{j=1}^{m}\alpha_{j}(q) + \theta \alpha_{m+1}(q)\right].
	\end{equation}
\end{theorem}
\begin{proof}
	Indeed, \eqref{EQ: TraceMaximizationComputationFormula} follows from \eqref{EQ: LeonovContextTraceExponentsViaTraceNumbersSelfAdj} and Theorem \ref{TH: ComputationSwedgeSpectralBound}. Then, \eqref{EQ: ComputationMaxExponent} follows from \eqref{EQ: TraceExponentViaTraceNumbers} and \eqref{EQ: MaximizedExponentViaMaximizedTraceExponent}, and \eqref{EQ: LeonovContextOmegaDViaAlphas} is a particular case of \eqref{EQ: OmegaEstimateViaMaxOfBetas}.
\end{proof}

\begin{remark}
	Recall that if the numbers $\alpha_{j}(q)$ behave regularly enough over the trajectories of $(\mathcal{Q},\vartheta)$ as in Theorem \ref{TH: RegularityOfTraceExponentAndCoincidence}, we have that $\alpha^{+}_{\mathfrak{n}^{\wedge m}}(q) = \alpha^{tr}_{\mathfrak{n}^{\wedge m}}(q)$, and the sharper estimate \eqref{EQ: OmegaEstimateViaTimeAveragedMaxOfBetas} is valid.
\end{remark}

In the next section, we are going to illustrate the theory by means of delay equations in $\mathbb{R}^{n}$.

%% file: ApplicationsDelayEqs.tex
\section{Adapted metrics for delay equations in $\mathbb{R}^{n}$}
\label{SEC: DimensionEstimatesViaLiouvDelayEqs}
\subsection{Linearization of delay equations}
\label{SEC: LinearizationDelayEquations}
Let us consider the class of nonlinear nonautonomous delay equations over a semiflow $(\mathcal{Q},\vartheta)$ on a complete metric space $\mathcal{Q}$, which are described over $q \in \mathcal{Q}$ by
\begin{equation}
	\label{EQ: DelayRnNonlinearNonautonomous}
	\dot{x}(t) = \widetilde{A}x_{t} + \widetilde{B}F(\vartheta^{t}(q),Cx_{t}) + \widetilde{W}(\vartheta^{t}(q)).
\end{equation}
Here $\tau>0$ is fixed, $x(\cdot) \colon [-\tau,T] \to \mathbb{R}^{n}$ for some $T>0$, and $x_{t}(\theta) \coloneq x(t+\theta)$ for $\theta \in [-\tau,0]$ denotes the $\tau$-history segment of $x(\cdot)$ at $t \in [0,T]$; $\widetilde{A} \colon C([-\tau,0];\mathbb{R}^{n}) \to \mathbb{R}^{n}$ and $C \colon C([-\tau,0];\mathbb{R}^{n}) \to \mathbb{R}^{r_{2}}$ are bounded linear operators; $\widetilde{B}$ is an $n \times r_{1}$-matrix; $\widetilde{W} \colon \mathcal{Q} \to \mathbb{R}^{n}$ is a bounded continuous function; and $F \colon \mathcal{Q} \times \mathbb{R}^{r_{2}} \to \mathbb{R}^{r_{1}}$ is a $C^{1}$-differentiable in the second argument continuous mapping satisfying for some $\Lambda > 0$ the Lipschitz condition
\begin{equation}
	|F(q,y_{1})-F(q,y_{2})|_{\mathbb{R}^{r_{1}}} \leq \Lambda |y_{1}-y_{2}|_{\mathbb{R}^{r_{2}}} \qquad \text{for all} \quad q \in \mathcal{Q} \quad \text{and} \quad y_{1},y_{2} \in \mathbb{R}^{r_{2}}.
\end{equation}

Let us write \eqref{EQ: DelayRnNonlinearNonautonomous} in the form of an evolution equation. For this, consider the Hilbert space $\mathbb{H} = \mathbb{R}^{n} \times L_{2}(-\tau,0;\mathbb{R}^{n})$ and define an unbounded operator $A$ in $\mathbb{H}$ by
\begin{equation}
	\label{EQ: OperatorADelayEqsDef}
	(x,\phi) \overset{A}{\mapsto} \left(\widetilde{A}\phi,\frac{d}{d\theta}\phi \right),
\end{equation}
where $(x,\phi) \in \mathcal{D}(A) \coloneq \{ (x,\phi) \in \mathbb{H} \ | \ \phi(\cdot) \in W^{1,2}(-\tau,0;\mathbb{R}^{n}) \text{ and } \phi(0) = x \}$, and $\frac{d}{d\theta}$ denotes the derivative in the Sobolev space $W^{1,2}(-\tau,0;\mathbb{R}^{n})$. We call $A$ the \textit{delay operator} associated with $\widetilde{A}$. It can be shown that $A$ generates a $C_{0}$-semigroup $G$ in $\mathbb{H}$.

Let $\mathbb{E} \coloneq C([-\tau,0];\mathbb{R}^{n})$ be embedded into $\mathbb{H}$ as $\phi \mapsto (\phi(0),\phi)$ for any $\phi \in \mathbb{E}$. It is convenient to identify the elements of $\mathbb{E}$ and their images in $\mathbb{H}$ under the embedding. In particular, we use the same notation for the operator induced by $C$ from \eqref{EQ: DelayRnNonlinearNonautonomous}, i.e., we set $C(x,\phi) \coloneq C\phi$ for $(x,\phi) \in \mathbb{E}$.

Now define a bounded linear operator $B \colon \mathbb{R}^{r_{1}} \to \mathbb{H}$ by $B\eta \coloneq (\widetilde{B} \eta,0)$ for any $\eta \in \mathbb{R}^{r_{1}}$. Furthermore, define $W \colon \mathcal{Q} \to \mathbb{H}$ by $W(q) \coloneq (\widetilde{W}(q),0)$ for any $q \in \mathcal{Q}$. Then \eqref{EQ: DelayRnNonlinearNonautonomous} can be treated as an abstract nonautonomous evolution equation in $\mathbb{H}$ over $(\mathcal{Q},\vartheta)$:
\begin{equation}
	\label{EQ: DelayLinearCocAbsract}
	\dot{v}(t) = Av(t) + BF(\vartheta^{t}(q),Cv(t)) + W(\vartheta^{t}(q)).
\end{equation}
By \cite[Theorem 1]{Anikushin2022Semigroups}, for any $v_{0} \in \mathbb{H}$ and $T>0$, there exists a unique generalized solution $v(\cdot) = v(\cdot;q,v_{0}) = (x(\cdot),\phi(\cdot))$ to \eqref{EQ: DelayLinearCocAbsract} on $[0,T]$ satisfying $v(0)=v_{0}$. More precisely, $v(\cdot)$ is a continuous $\mathbb{H}$-valued function such that $\phi(t) = x_{t}$ and
\begin{equation}
	\label{EQ: DelayVariationOfConstantsSingle}
	v(t) = G(t)v_{0} + \int_{0}^{t}G(t-s)\left[BF(\vartheta^{s}(q),Cx_{s}) + W(\vartheta^{s}(q)) \right]ds
\end{equation}
are satisfied for any $t \in [0,T]$. Note that it is possible to interpret the function $[0,T] \ni s \mapsto Cx_{s} \in \mathbb{R}^{r_{2}}$ for $x(\cdot) \in L_{2}(-\tau,T;\mathbb{R}^{n})$ as an element of $L_{2}(0,T;\mathbb{R}^{r})$, and so \eqref{EQ: DelayVariationOfConstantsSingle} is well-defined, see \cite{Anikushin2023Comp,Anikushin2020FreqDelay,Anikushin2022Semigroups}. This formula can be used to show that the mappings
\begin{equation}
	\label{EQ: NonlinCocycleDelayEqsDefinition}
	\psi^{t}(q,v_{0}) \coloneq v(t;q,v_{0}) \qquad \text{for all} \quad t \geq 0, q \in \mathcal{Q}, \ \text{and} \ v_{0} \in \mathbb{H}
\end{equation}
determine a cocycle $\psi$ in $\mathbb{H}$ over the semiflow $(\mathcal{Q},\vartheta)$.

Generalized solutions are related to classical solutions in the space $\mathbb{E} = C([-\tau,0];\mathbb{R}^{n})$ as follows. Recall that $\mathbb{E}$ is embedded into $\mathbb{H}$ as $\phi \mapsto (\phi(0),\phi)$ for $\phi \in \mathbb{E}$, and we identify such $\phi$ with its image $(\phi(0),\phi)$ in $\mathbb{H}$. Then the restriction of $\psi$ to $\mathbb{E}$ is the cocycle in $\mathbb{E}$ generated by classical solutions (see \cite{HaleLunel1993}), i.e., $\psi^{t}(q,v_{0}) = (x(0),x_{t})$ for $t \geq 0$, where $x(\cdot) \colon [-\tau,+\infty) \to \mathbb{R}^{n}$ is a continuous function that is $C^{1}$-differentiable on $[0,\infty)$, satisfies \eqref{EQ: DelayRnNonlinearNonautonomous} for $t \geq 0$, and $x(\theta) = \phi_{0}(\theta)$ for $\theta \in [-\tau,0]$, where $v_{0} = (\phi(0), \phi_{0}) \in \mathbb{E}$. Furthermore, $\psi^{\tau}(q,\mathbb{H}) \subset \mathbb{E}$, and for any $T \geq \tau$, there exists a constant $L_{T}>0$ such that the smoothing estimate
\begin{equation}
	\label{EQ: DelayCocNonlinearSmoothingEst}
	\| \psi^{t}(q,v_{1}) - \psi^{t}(q,v_{2}) \|_{\mathbb{E}} \leq L_{T} |v_{1}-v_{2}|_{\mathbb{H}} \qquad \text{for any} \quad t \in [\tau,T]
\end{equation}
is valid for all $v_{1},v_{2} \in \mathbb{H}$ and $q \in \mathcal{Q}$, see \cite[Theorem 1]{Anikushin2022Semigroups}.

With each cocycle $\psi$ there is the associated skew-product semiflow $\pi$ in $ \mathcal{Q} \times \mathbb{H}$ given by 
\begin{equation}
	\label{EQ: DelaySkewProductSemiflowAssCoc}
	\pi^{t}(q,v) = (\vartheta^{t}(q),\psi^{t}(q,v)) \qquad \text{for} \quad t \geq 0, q \in \mathcal{Q}, \ \text{and} \ v \in \mathbb{H}.
\end{equation}
From the above, it follows that any invariant subset $\mathcal{P} \subset \mathcal{Q} \times \mathbb{H}$, i.e., such that $\pi^{t}(\mathcal{P}) = \mathcal{P}$ for any $t \geq 0$, must satisfy $\mathcal{P} \subset \mathcal{Q} \times \mathbb{E}$.

For $q \in \mathcal{Q}$ and $v_{0} \in \mathbb{E}$ (for simplicity), the linearized in the fiber over $q$ equation along the trajectory of $v_{0}$ with respect to $\psi$ is given by
\begin{equation}
	\label{EQ: LinearizedEquationDelayInRn}
	\dot{\xi}(t) = A\xi(t) + BF'(\vartheta^{t}(q),C\psi^{t}(q,v_{0}))C\xi(t),
\end{equation}
where $F'$ denotes the derivative of $F(q,y)$ with respect to $y$.

Similarly to the above considerations, it can be shown that for any $\wp = (q,v_{0}) \in \mathcal{Q} \times \mathbb{E}$ and $\xi_{0} \in \mathbb{H}$, there exists a unique generalized solution $\xi(t)=\xi(t;q,v_{0},\xi_{0})$ to \eqref{EQ: LinearizedEquationDelayInRn} with $\xi(0) = \xi_{0}$, and the mappings
\begin{equation}
	\label{EQ: DelayLinearizationCocycleDef}
	\Xi^{t}(\wp,\xi_{0}) \coloneq \xi(t;q,v_{0},\xi_{0}) \qquad \text{for} \quad t \geq 0, \wp=(q,v_{0}) \in \mathcal{Q} \times \mathbb{E}, \ \text{and} \ \xi_{0} \in \mathbb{H}
\end{equation}
determine a uniformly continuous linear cocycle $\Xi$ in $\mathbb{H}$ over the semiflow $(\mathcal{Q} \times \mathbb{E}, \pi)$. Analogously to \eqref{EQ: DelayCocNonlinearSmoothingEst}, we have that $\Xi^{\tau}(\wp,\mathbb{H}) \subset \mathbb{E}$, and for any $T \geq \tau$, there exists a constant $L_{T}>0$ such that
\begin{equation}
	\label{EQ: SmoothingDelayCocycle}
	\| \Xi^{t}(\wp,\xi) \|_{\mathbb{E}} \leq L_{T} | \xi |_{\mathbb{H}} \qquad \text{for any} \quad t \in [\tau, T], \wp \in \mathcal{Q} \times \mathbb{E}, \ \text{and} \ \xi \in \mathbb{H}.
\end{equation}

It can be shown using \eqref{EQ: DelayVariationOfConstantsSingle}, see \cite[Theorem 2]{Anikushin2022Semigroups} for the case of semiflows, that for all $q \in \mathcal{Q}$, $v_{0} \in \mathbb{E}$, $T>0$, and for any bounded subset $\mathcal{B}$ of $\mathbb{H}$, we have
\begin{equation}
	\lim_{h \to 0}\frac{|\psi^{t}(q,v_{0} + h \xi) - \psi^{t}(q,v_{0}) - h\Xi^{t}(\wp,\xi) |_{\mathbb{H}}}{h} = 0,
\end{equation}
where $\wp = (q, v_{0})$, and the limit is uniform in $t \in [0,T]$ and $\xi \in \mathcal{B}$. Thus, $\Xi^{t}(\wp,\cdot)$ represents the Fr\'{e}chet differential of $\psi^{t}(q,\cdot) \colon \mathbb{H} \to \mathbb{H}$ at $v_{0} \in \mathbb{E}$, which continuously depends on $\wp \in \mathcal{Q} \times \mathbb{E}$. This allows us to relate the Lyapunov dimension of $\Xi$ and the fractal dimension of compact invariant subsets $\mathcal{P} \subset \mathcal{Q} \times \mathbb{E}$ (more precisely, their fibers), see \cite{ChepyzhovIlyin2004, KuzReit2020}.

Below, we study the uniform Lyapunov exponents and the Lyapunov dimension of the cocycle $\Xi$ over the restricted semiflow $(\mathcal{P},\pi)$, where $\mathcal{P}$ is positively invariant under $\pi$.

\subsection{Additive symmetrization of delay operators}
\label{SEC: DelayEqsSymmetrization}
By \cite[Theorem 1]{Anikushin2022Semigroups}, we have that for $\xi_{0} \in \mathcal{D}(A)$, where $\mathcal{D}(A)$ is defined below \eqref{EQ: OperatorADelayEqsDef}, the corresponding generalized solution $\xi(t)=\xi(t;q,v_{0},\xi_{0})$ to \eqref{EQ: LinearizedEquationDelayInRn} is a classical solution, i.e., $\xi(\cdot) \in C^{1}([0,+\infty);\mathbb{H}) \cap C([0,+\infty);\mathcal{D}(A))$. In terms of \eqref{EQ: DomainAqDefinitionCocycleLiovForm}, this gives that $\mathcal{D}(A(\wp)) \supset \mathcal{D}(A)$ and $\Xi^{t}(\wp,\mathcal{D}(A)) \subset \mathcal{D}(A)$ for all $t \geq 0$ and $\wp \in \mathcal{Q} \times \mathbb{E}$. Furthermore, from \eqref{EQ: LinearizedEquationDelayInRn}, it is seen that for any $\wp=(q,v_{0}) \in \mathcal{Q} \times \mathbb{E}$ and $\xi_{0} \in \mathcal{D}(A)$, we have
\begin{equation}
	\label{EQ: DelaySymmetrizationOperatorAqFormula}
	A(\wp)\xi_{0} = A\xi_{0} + BF'(q,Cv_{0})C\xi_{0}.
\end{equation}
This operator has the same form as $A$ in \eqref{EQ: OperatorADelayEqsDef}. In particular, $A(\wp)$ considered on $\mathcal{D}(A)$ is a closed operator and generates a $C_{0}$-semigroup in $\mathbb{H}$.

As a model for $A(\wp)$, let us consider a bounded linear operator $\widetilde{L} \colon C([-\tau,0];\mathbb{R}^{n}) \to \mathbb{R}^{n}$ and the associated with it delay operator $L$ as in \eqref{EQ: OperatorADelayEqsDef}, i.e.,
\begin{equation}
	\label{EQ: OperatorLSymmetrizationDelayEqsDef}
	(x,\phi) \overset{L}{\mapsto} \left(\widetilde{L}\phi,\frac{d}{d\theta}\phi \right),
\end{equation}
where $(x,\phi) \in \mathcal{D}(L) = \{ (x,\phi) \in \mathbb{H} \ | \ \phi(\cdot) \in W^{1,2}(-\tau,0;\mathbb{R}^{n}) \text{ and } \phi(0) = x \}$.

We suppose that $\widetilde{L}$ has the form
\begin{equation}
	\label{EQ: OperatorLDelaySymMeasureRepresentation}
	\widetilde{L}\phi = L_{0} \phi(0) + L_{-\tau}\phi(-\tau) + \sum_{j=1}^{J}L_{-\tau_{j}}\phi(-\tau_{j}) + \sum_{i=1}^{I}\int_{-\tau}^{0}M_{i}(\theta)\phi(\theta)d\theta
\end{equation}
for any $\phi \in C([-\tau,0];\mathbb{R}^{n})$. Here $L_{0}$, $L_{-\tau}$, and $L_{-\tau_{j}} \not= 0 $ are $n\times n$-matrices with $\tau_{j} \in (0,\tau)$ for each $j \in \{1,\ldots,J\}$ being distinct values, which are increasing in $j$, and $M_{i}(\cdot)$ is a $L_{2}$-summable $(n\times n)$-matrix-valued function for any $i \in \{1,\ldots,I\}$.

It is convenient to set $\tau_{0}\coloneq0$ and $\tau_{J+1}\coloneq\tau$. So, we obtain a partition of $[-\tau,0]$ as
\begin{equation}
	\label{EQ: DelayOperatorPartition}
	-\tau = -\tau_{J+1} < -\tau_{J} < \cdots < -\tau_{1} < -\tau_{0} = 0.
\end{equation}

Let $\rho$ be a positive function on $[-\tau,0]$, which may possibly have discontinuities at $-\tau_{j}$ for $j \in \{1,\ldots,J\}$, such that\footnote{This should be read as $\rho(\cdot) \in C^{1}((-\tau_{j+1},-\tau_{j});\mathbb{R})$ and that $\rho(\cdot)$ has natural extensions to the endpoints $-\tau_{j+1}$ and $-\tau_{j}$, so it becomes $C^{1}$-differentiable on the closed interval.} $\rho(\cdot) \in C^{1}([-\tau_{j+1},-\tau_{j}];\mathbb{R})$ for each $j \in \{0,\ldots,J\}$ and $\inf_{\theta \in [-\tau,0]}\rho(\theta) > 0$. Next, let $\langle \cdot, \cdot \rangle_{\mathbb{R}^{n}}$ be a fixed inner product in $\mathbb{R}^{n}$. We endow $\mathbb{H}$ with the inner product
\begin{equation}
	\label{EQ: DelayEqsSymWeightedInnerProduct}
	\langle (x,\phi),(y,\psi) \rangle_{\rho} \coloneq \langle x, y \rangle_{\mathbb{R}^{n}} + \int_{-\tau}^{0}\rho(\theta)\langle \phi(\theta), \psi(\theta) \rangle_{\mathbb{R}^{n}}d\theta,
\end{equation}
where $(x,\phi), (y,\psi) \in \mathbb{H}$.

Let $\mathcal{T}$ be the partition of $[-\tau,0]$ corresponding to \eqref{EQ: DelayOperatorPartition}. By $W^{1,2}_{\mathcal{T}}(-\tau,0;\mathbb{R}^{n})$ we denote the space of $\psi \in L_{2}(-\tau,0;\mathbb{R}^{n})$ such that $\psi \in W^{1,2}(-\tau_{j+1},-\tau_{j};\mathbb{R}^{n})$ for each $j \in \{0,\ldots,J\}$. For such $\psi$ and any $j \in \{1,\ldots,J\}$, we define
\begin{equation}
	\label{EQ: DelayEquationsInMetricsDeltaJOperatorDef}
	\Delta_{j}(\psi) \coloneq \psi(-\tau_{j}+0) - \psi(-\tau_{j}-0),
\end{equation}
where $ \psi(-\tau_{j}+0)$ and $\psi(-\tau_{j}-0)$ denote the right-hand and left-hand limits at $-\tau_{j}$, respectively. Note that $\rho \in W^{1,2}_{\mathcal{T}}(-\tau,0;\mathbb{R})$, so the same notation is applicable to it. 

There is the differentiation operator $\frac{d}{d\theta}$ taking $W^{1,2}_{\mathcal{T}}(-\tau,0;\mathbb{R}^{n})$ into $L_{2}(-\tau,0;\mathbb{R}^{n})$ by
\begin{equation}
	\label{EQ: DifferentiationOperatorSobolevOverPartition}
	\restr{\left(\frac{d}{d\theta}\psi\right)}{(-\tau_{j+1},-\tau_{j})} \coloneq \frac{d}{d\theta}\restr{\psi}{(-\tau_{j+1},-\tau_{j})} \qquad \text{for} \quad j \in \{0,\ldots,J\},
\end{equation}
where $\frac{d}{d\theta}$ on the right-hand side is the differentiation operator in the usual Sobolev space $W^{1,2}(-\tau_{j+1},-\tau_{j};\mathbb{R}^{n})$. From this, we endow $W^{1,2}_{\mathcal{T}}(-\tau,0;\mathbb{R}^{n})$ with the norm
\begin{equation}
	\| \psi \|^{2}_{W^{1,2}_{\mathcal{T}}} \coloneq \| \psi \|^{2}_{L_{2}(-\tau,0;\mathbb{R}^{n})} + \left\| \frac{d}{d\theta} \psi \right\|^{2}_{L_{2}(-\tau,0;\mathbb{R}^{n})}
\end{equation}
that makes it a Hilbert space. Clearly, then $\frac{d}{d\theta}$ defined in \eqref{EQ: DifferentiationOperatorSobolevOverPartition} is a bounded linear operator. Sometimes it is convenient to write $\psi'$ instead of $\frac{d}{d\theta} \psi$.

Given $j \in \{1,\ldots,J\}$, we define an operator $\mathfrak{D}^{\rho}_{j}$ from $W^{1,2}_{\mathcal{T}}(-\tau,0;\mathbb{R}^{n})$ to $\mathbb{R}^{n}$ by
\begin{equation}
	\label{EQ: DelayAdjointDifferenceOperator}
	\mathfrak{D}^{\rho}_{j}(\psi) \coloneq \Delta_{j}(\rho \psi) = \rho(-\tau_{j}+0)\psi(-\tau_{j}+0) - \rho(-\tau_{j}-0)\psi(-\tau_{j}-0).
\end{equation}
Since $W^{1,2}(-\tau_{j+1},-\tau_{j};\mathbb{R}^{n})$ naturally embeds into $C([-\tau_{j+1},-\tau_{j}];\mathbb{R}^{n})$, the operator $\mathfrak{D}^{\rho}_{j}(\cdot)$ is bounded.

In the following theorem, the adjoint $L^{*}$ of $L$ with respect to \eqref{EQ: DelayEqsSymWeightedInnerProduct} is described.
\begin{theorem}
	\label{TH: SymmetrizationDelayOperator}
	Suppose $L$ is given by \eqref{EQ: OperatorLSymmetrizationDelayEqsDef} and satisfies \eqref{EQ: OperatorLDelaySymMeasureRepresentation}. Then the adjoint $L^{*}$ of $L$ with respect to the inner product $\langle \cdot, \cdot \rangle_{\rho}$ from \eqref{EQ: DelayEqsSymWeightedInnerProduct} is given by\footnote{Here $\rho'$ is the derivative of $\rho$ in the space $W^{1,2}_{\mathcal{T}}(-\tau,0;\mathbb{R})$.}
	\begin{equation}
		\label{EQ: AdjointDelayOperatorAction}
		(y,\psi) \overset{L^{*}}{\mapsto} \left( L^{*}_{0}y + \rho(0) \psi(0) , -\frac{d}{d\theta}\psi - \frac{\rho'(\cdot)}{\rho(\cdot)} \psi + \frac{\sum_{i=1}^{I}M^{*}_{i}(\cdot)}{\rho(\cdot)}y \right)
	\end{equation}
	and defined for $(y,\psi) \in \mathcal{D}(L^{*})$, where
	\begin{equation}
		\mathcal{D}(L^{*}) = \{ (y,\psi) \in \mathbb{H} \ | \ \psi \in W^{1,2}_{\mathcal{T}}(-\tau,0;\mathbb{R}^{n}) \text{ and } \eqref{EQ: DelayAdjointBoundaryConditions} \text{ is satisfied}  \}.
	\end{equation}
	Here, the adjoints of matrices are taken with respect to the inner product $\langle \cdot, \cdot \rangle_{\mathbb{R}^{n}}$ in $\mathbb{R}^{n}$ from \eqref{EQ: DelayEqsSymWeightedInnerProduct}, and the differentiation $\frac{d}{d\theta}$ is understood according to \eqref{EQ: DifferentiationOperatorSobolevOverPartition}.
\end{theorem}
\begin{proof}
	Let us consider $(x,\phi) \in \mathcal{D}(L)$ and $(y,\psi) \in \mathbb{H}$ such that $\psi(\cdot) \in W^{1,2}_{\mathcal{T}}(-\tau,0;\mathbb{R}^{n})$. Integrating by parts yields
	\begin{equation}
		\label{EQ: ComputingAdjointDelayOperator1}
		\begin{split}
			\int_{-\tau}^{0}\rho(\theta)\langle \phi'(\theta),\psi(\theta) \rangle_{\mathbb{R}^{n}} d\theta = &\sum_{j=0}^{J} \langle \phi(-\tau_{j}), \rho(-\tau_{j}-0) \psi(-\tau_{j}-0)\rangle_{\mathbb{R}^{n}} \\ -  &\sum_{j=0}^{J}\langle \phi(-\tau_{j+1}), \rho(-\tau_{j+1}+0) \psi(-\tau_{j+1}+0)\rangle_{\mathbb{R}^{n}} \\ -
			&\sum_{j=0}^{J}\int_{-\tau_{j+1}}^{-\tau_{j}}\rho(\theta) \left\langle \phi(\theta), \psi'(\theta) + \frac{\rho'(\theta)}{\rho(\theta)}\psi(\theta) \right\rangle_{\mathbb{R}^{n}}d\theta.
		\end{split}
	\end{equation}
	From this, taking into account that for $j \in \{0,\ldots,J-1\}$
	\begin{equation}
		\label{EQ: ComputingAdjointDelayDefectRelation}
		\begin{split}
			\rho(-\tau_{j+1}+0) \psi(-\tau_{j+1}+0) = \rho(-\tau_{j+1}-0) \psi(-\tau_{j+1}-0) + \mathfrak{D}^{\rho}_{j+1}(\psi),
		\end{split}
	\end{equation}
	where $\mathfrak{D}^{\rho}_{j}(\psi)$ is given by \eqref{EQ: DelayAdjointDifferenceOperator}, we deduce that
	\begin{equation}
		\label{EQ: ComputingAdjointDelayOperator2}
		\begin{split}
			\sum_{j=0}^{J} \langle \phi(-\tau_{j}), \rho(-\tau_{j}-0) \psi(-\tau_{j}-0)\rangle_{\mathbb{R}^{n}} \\ -\sum_{j=0}^{J}\langle \phi(-\tau_{j+1}), \rho(-\tau_{j+1}+0) \psi(-\tau_{j+1}+0)\rangle_{\mathbb{R}^{n}} = \\ = \langle \phi(0) , \rho(0) \psi(0) \rangle_{\mathbb{R}^{n}} - \langle \phi(-\tau) , \rho(-\tau) \psi(-\tau) \rangle_{\mathbb{R}^{n}} -
			\sum_{j=1}^{J} \langle \phi(-\tau_{j}), \mathfrak{D}^{\rho}_{j}(\psi) \rangle_{\mathbb{R}^{n}}.
		\end{split}		 
	\end{equation}
	
	Moreover, we clearly have
	\begin{equation}
		\label{EQ: ComputingAdjointDelayOperator3}
		\left\langle \sum_{i=1}^{I}\int_{-\tau}^{0}M_{i}(\theta)\phi(\theta)d\theta, y \right\rangle_{\mathbb{R}^{n}} = \int_{-\tau}^{0}\left\langle \phi(\theta), \sum_{i=1}^{I} M^{*}_{i}(\theta)y \right\rangle_{\mathbb{R}^{n}}d\theta.
	\end{equation}
	
	By combining  \eqref{EQ: ComputingAdjointDelayOperator1}, \eqref{EQ: ComputingAdjointDelayOperator2}, \eqref{EQ: ComputingAdjointDelayOperator3}, and \eqref{EQ: OperatorLDelaySymMeasureRepresentation}, we obtain 
	\begin{equation}
		\label{EQ: OperatorLDelayBilinearForm}
		\begin{split}
			\left\langle L(x,\phi), (y,\psi) \right\rangle_{\rho} = \langle \widetilde{L}\phi, y \rangle_{\mathbb{R}^{n}} + \int_{-\tau}^{0}\rho(\theta)\langle \phi'(\theta),\psi(\theta) \rangle_{\mathbb{R}^{n}} d\theta = \\ = \int_{-\tau}^{0}\rho(\theta)\left\langle \phi(\theta), \frac{\sum_{i=1}^{I}M^{*}_{i}(\theta)}{\rho(\theta)}y \right\rangle_{\mathbb{R}^{n}}d\theta + \langle \phi(0), L^{*}_{0}y +\rho(0)\psi(0) \rangle_{\mathbb{R}^{n}} \\ - \int_{-\tau}^{0}\rho(\theta) \left\langle \phi(\theta), \psi'(\theta) + \frac{\rho'(\theta)}{\rho(\theta)}\psi(\theta) \right\rangle_{\mathbb{R}^{n}}d\theta \\ + \langle \phi(-\tau), L^{*}_{-\tau}y - \rho(-\tau)\psi(-\tau)\rangle_{\mathbb{R}^{n}} + \sum_{j=1}^{J}\langle \phi(-\tau_{j}), L^{*}_{-\tau_{j}}y - \mathfrak{D}^{\rho}_{j}(\psi)\rangle_{\mathbb{R}^{n}}.
		\end{split}
	\end{equation}
	By definition, $(y,\psi) \in \mathcal{D}(L^{*})$ if and only if the above functional of $(x,\phi) \in \mathcal{D}(L)$ is bounded in the norm of $\mathbb{H}$. Since the problematic terms are located only in the last row of \eqref{EQ: OperatorLDelayBilinearForm}, it is not difficult to see that the conditions
	\begin{equation}
		\label{EQ: DelayAdjointBoundaryConditions}
		L^{*}_{-\tau}y - \rho(-\tau)\psi(-\tau) = 0 \quad \text{and} \quad L^{*}_{-\tau_{j}}y - \mathfrak{D}^{\rho}_{j}(\psi) = 0 \qquad \text{for} \quad j \in \{1,\ldots, J\}
	\end{equation}
	are necessary and sufficient for the inclusion $(y,\psi) \in \mathcal{D}(L^{*})$. From this and \eqref{EQ: OperatorLDelayBilinearForm}, it is also clear that $L^{*}(y,\psi)$ is given by \eqref{EQ: AdjointDelayOperatorAction}.
	
	Clearly, the space $\mathcal{D}$ of all $(y,\psi)$ as above is dense in $\mathbb{H}$, and we have just shown that $\mathcal{D} \subset \mathcal{D}(L^{*})$. Moreover, it is also clear from \eqref{EQ: AdjointDelayOperatorAction} that $\mathcal{D}$ is complete in the graph norm of $L^{*}$. From this and \cite[Proposition 1.7, Chapter II]{EngelNagel2000}, to show that\footnote{A more direct approach for this is to establish that for a sufficiently large real $\lambda$, the equation $L^{*}v - \lambda v = w$ is uniquely solvable for any $w \in \mathbb{H}$ with $v \in \mathcal{D}$.} $\mathcal{D} = \mathcal{D}(L^{*})$, it is sufficient to establish that $\mathcal{D}$ is invariant with respect to the $C_{0}$-semigroup generated by $L^{*}$.
	
	Now consider the classical solution $(y(t),\psi(t))$, where $t \geq 0$, to the homogeneous Cauchy problem corresponding to $L^{*}$ with initial data $(y(0),\psi(0)) = (y_{0},\psi_{0}) \in \mathcal{D}$. We shall establish that it satisfies $(y(t),\psi(t)) \in \mathcal{D}$ for any $t \geq 0$ and
	\begin{equation}
		\label{EQ: AdjointDelayCompSystemEqs1}
		\begin{split}
			\dot{y}(t) &= L^{*}_{0}y(t) + \rho(0) \psi(t)(0),\\
			\dot{\psi}(t) &= -\frac{d}{d\theta} \psi(t) - \frac{\rho'(\cdot)}{\rho(\cdot)}\psi(t) + \frac{\sum_{i=1}^{I}M^{*}_{i}(\cdot)}{\rho(\cdot)}y(t).
		\end{split}
	\end{equation}
	
	Let $T^{-}_{\mathcal{T}}$ be the $C_{0}$-semigroup of right translations on $L_{2}(-\tau,0;\mathbb{R}^{n})$ according to the partition $\mathcal{T}$, i.e.,
	\begin{equation}
		(T^{-}_{\mathcal{T}}(t)\psi)(\theta) \coloneq \begin{cases}
			\psi(\theta - t) &\text{if} \quad \theta \in (-\tau_{j+1},-\tau_{j}) \quad \text{and} \quad \theta-t \in (-\tau_{j+1},-\tau_{j}),\\
			0 &\text{otherwise},
		\end{cases}
	\end{equation}
	where the first condition is taken over $j \in \{0,\ldots,J\}$. Standard arguments show that its generator $A^{-}_{\mathcal{T}}$ is given by the operator $-\frac{d}{d\theta}$ defined on the domain
	\begin{equation}
		\label{EQ: DelayAdjointCompAtranslatesDef}
		\mathcal{D}(A^{-}_{\mathcal{T}}) \coloneq \{ \psi \in W^{1,2}_{\mathcal{T}}(-\tau,0;\mathbb{R}^{n}) \mid \psi(-\tau_{j}+0) = 0 \text{ for } j \in \{1,\ldots,J+1\}\}.
	\end{equation}
	
	Consider the $L_{2}(-\tau,0;\mathbb{R}^{n})$-valued function of $t \geq 0$ defined by
	\begin{equation}
		\mathcal{I}_{y}(t) \coloneq \int_{0}^{t}T^{-}_{\mathcal{T}}(t-s)\sum_{i=1}^{I}M^{*}_{i}(\cdot)y(s)ds
	\end{equation}
	For $y(\cdot) \in C^{1}([0,\infty); \mathbb{R}^{n})$, the sum $\sum_{i=1}^{I}M^{*}_{i}(\cdot)y(t)$ as an $L_{2}(-\tau,0;\mathbb{R}^{n})$-valued function of $t \in [0,\infty)$ is $C^{1}$-differentiable. Then by \cite[Theorem 6.5, Chapter I]{Krein1971}, $\mathcal{I}_{y}(\cdot)$ is a classical solution to the inhomogeneous problem associated with $A^{-}_{\mathcal{T}}$. Namely,
	\begin{equation}
		\label{EQ: ComputingAdjointIySolution}
		\mathcal{I}_{y}(\cdot) \in C([0,\infty); \mathcal{D}(A^{-}_{\mathcal{T}})) \cap C^{1}([0,\infty); L_{2}(-\tau,0;\mathbb{R}^{n})),
	\end{equation}
	and $\frac{d}{dt}\mathcal{I}_{y}(t) = A^{-}_{\mathcal{T}}\mathcal{I}_{y}(t) + \sum_{i=1}^{I}M^{*}_{i}(\cdot)y(t)$ for any $t \geq 0$. In particular, \eqref{EQ: ComputingAdjointIySolution} and \eqref{EQ: DelayAdjointCompAtranslatesDef} give for any $t \geq 0$ that (recall $\Delta_{j}$ defined in \eqref{EQ: DelayEquationsInMetricsDeltaJOperatorDef})
	\begin{equation}
		\label{EQ: DelayAdjointCompBoundaryCondsInhomogeneous}
		\Delta_{j}(\mathcal{I}_{y}(t)) = - \mathcal{I}_{y}(t)(-\tau_{j}-0) \qquad \text{for} \quad j \in \{1,\ldots,J\},
	\end{equation}
	and $\mathcal{I}_{y}(t)(-\tau) = 0$.
	
	It is convenient to search for the function $\psi(t)$ in \eqref{EQ: AdjointDelayCompSystemEqs1} for $t \geq 0$ by assuming that its form is defined by
	\begin{equation}
		\label{EQ: AdjointDelayComputationEtaIntroduction}
		\eta_{j}(t-\theta - \tau_{j}) + \mathcal{I}_{y}(t)(\theta) = \rho(\theta)\psi(t)(\theta) \qquad \text{for} \quad \theta \in (-\tau_{j+1},-\tau_{j}) 
	\end{equation}
	with some $\eta_{j}(\cdot) \in W^{1,2}_{loc}(0,\infty;\mathbb{R}^{n})$ for each $j \in \{0,\ldots,J\}$. For such $\eta_{j}$, \cite[Lemma 3.4]{BatkaiPiazzera2005} implies that the mapping $[0,\infty) \ni t \mapsto \eta^{-}_{j, t}\coloneq\eta_{j}(t-\cdot) \in L_{2}(-\tau_{j+1}+\tau_{j},0;\mathbb{R}^{n})$ is $C^{1}$-differentiable and satisfies
	\begin{equation}
		\frac{d}{dt}\eta^{-}_{j,t} = -\frac{d}{d\theta}\eta^{-}_{j,t} \qquad \text{for any} \quad t \geq 0.
	\end{equation}
	
	From the above, it is clear that if we have found such $\eta_{j}$ and $y(\cdot) \in C^{1}([0,\infty);\mathbb{R}^{n})$, then
	\begin{equation}
		\psi(\cdot) \in C([0,\infty);W^{1,2}_{\mathcal{T}}(-\tau,0;\mathbb{R}^{n})) \cap C^{1}([0,\infty);L_{2}(-\tau,0;\mathbb{R}^{n})),
	\end{equation}
	and $\psi(t)$ satisfies the second equation in \eqref{EQ: AdjointDelayCompSystemEqs1} for any $t \geq 0$.
	
	Moreover, from \eqref{EQ: AdjointDelayComputationEtaIntroduction} with $\theta = 0$, the first equation in \eqref{EQ: AdjointDelayCompSystemEqs1} reads as
	\begin{equation}
		\label{EQ: DelayAdjCompYequation}
		\dot{y}(t) = L^{*}_{0}y(t) + \eta_{0}(t) + \mathcal{I}_{y}(t)(0). 
	\end{equation}
	
	Now let us express the boundary conditions \eqref{EQ: DelayAdjointBoundaryConditions} in terms of $\eta_{j}(\cdot)$ and $y(\cdot)$. From $\mathcal{I}_{y}(t)(-\tau) = 0$ and \eqref{EQ: AdjointDelayComputationEtaIntroduction}, we obtain that the first identity from \eqref{EQ: DelayAdjointBoundaryConditions} is equivalent to
	\begin{equation}
		\label{EQ: DelayAdjCompBoundConds1}
		L^{*}_{-\tau}y(t) - \eta_{J}(t+\tau-\tau_{J}) = 0.
	\end{equation}
	
	For convenience, we set $\eta(t)(\theta) \coloneq \rho(\theta)\psi(t,\theta) - \mathcal{I}_{y}(t)(\theta)$, i.e., $\eta(t)(\theta) = \eta_{j}(t-\theta+\tau_{j})$ for $\theta \in (-\tau_{j+1},-\tau_{j})$ and $j \in \{0,\ldots,J\}$. Then
	\begin{equation}
		\mathfrak{D}^{\rho}_{j}(\psi(t)) = \Delta_{j}(\rho \psi(t)) = \Delta_{j}(\eta(t)) + \Delta_{j}(\mathcal{I}_{y}(t)).
	\end{equation}
	From this and \eqref{EQ: DelayAdjointCompBoundaryCondsInhomogeneous}, the remaining series of conditions from \eqref{EQ: DelayAdjointBoundaryConditions} reads as
	\begin{equation}
		\label{EQ: DelayAdjCompBoundConds2}
		\eta_{j-1}(t+\tau_{j}-\tau_{j-1}) = L^{*}_{-\tau_{j}}y(t) + \eta_{j}(t) +\mathcal{I}_{y}(t)(-\tau_{j}-0)
	\end{equation}
	for $j \in \{1,\ldots,J\}$.
	
	So, it is required to construct appropriate $\eta_{j}(\cdot)$ and $y(\cdot)$ satisfying \eqref{EQ: DelayAdjCompYequation}, \eqref{EQ: DelayAdjCompBoundConds1}, and \eqref{EQ: DelayAdjCompBoundConds2}.
	
	First, from \eqref{EQ: AdjointDelayComputationEtaIntroduction} with $t = 0$, we obtain $\eta_{j}(-\theta - \tau_{j}) = \rho(\theta) \psi_{0}(\theta)$ for any $\theta \in (-\tau_{j+1},-\tau_{j})$ that determines $\eta_{j}$ on $(0,\tau_{j+1}-\tau_{j})$ as an element of $W^{1,2}(0,\tau_{j+1}-\tau_{j};\mathbb{R}^{n})$. 
	
	From this and the identity $\mathcal{I}_{y}(t)(0) = \mathcal{I}_{y}(\tau_{1})(0)$ for all $t \geq \tau_{1}$, it is seen that \eqref{EQ: DelayAdjCompYequation} transforms into
	\begin{equation}
		\label{EQ: AdjointDelayComputationYEquation2}
		\dot{y}(t) = \begin{cases}
			L^{*}_{0}y(t) + \rho(-t) \psi_{0}(-t) + \mathcal{I}_{y}(t)(0) \qquad &\text{if} \quad t \in [0,\tau_{1}],\\
			L^{*}_{0}y(t) + \eta_{0}(t) + \mathcal{I}_{y}(\tau_{1})(0) \qquad &\text{if} \quad t > \tau_{1}.
		\end{cases}
	\end{equation}
	We claim that, by the fixed point argument, there exists a unique $C^{1}$-differentiable solution $y(\cdot)$ to the above equation on $[0,\tau_{1}]$ such that $y(0)=y_{0}$. This can be shown as follows. Let $\mathcal{R} \colon L_{2}(-\tau,0;\mathbb{R}^{n}) \to L_{2}(-\tau_{1},0;\mathbb{R}^{n})$ be the operator that restricts functions to $(-\tau_{1},0)$ and $\mathcal{E} \colon L_{2}(-\tau_{1},0;\mathbb{R}^{n}) \to L_{2}(-\tau,0;\mathbb{R}^{n})$ be the operator that extends functions by zero outside $(-\tau_{1},0)$. Clearly, $\mathcal{R}T^{-}_{\mathcal{T}}\mathcal{E}$ is the right translation semigroup $T^{-}_{1}$ in $L_{2}(-\tau_{1},0;\mathbb{R}^{n})$ and
	\begin{equation}
		\label{EQ: RestrictionIyToRightInterval}
		\mathcal{R}\mathcal{I}_{y}(t) = \int_{0}^{t}T^{-}_{1}(t-s)\mathcal{R}\sum_{i=1}^{I}M^{*}_{i}(\cdot)y(s)ds.
	\end{equation}
	From this, it is seen that for $t \in [0,\tau_{1}]$ we have
	\begin{equation}
		\label{EQ: DelayDualMeasurement0Operator}
		\mathcal{I}_{y}(t)(0) = \mathcal{R}\mathcal{I}_{y}(t)(0) = \int_{0}^{t}\sum_{i=1}^{I}M^{*}_{i}(s-t)y(s)ds = \int_{-t}^{0}\sum_{i=1}^{I}M^{*}_{i}(s)y(s+t)ds,
	\end{equation}
	and it is possible to interpret the operator $y(\cdot) \mapsto (\mathcal{I}_{y})(\cdot)(0)$ as a bounded linear operator in $L_{2}(0,\varepsilon;\mathbb{R}^{n})$ with the norm independent of $\varepsilon \in [0,\tau_{1}]$. Then, the existence of a unique solution $y(\cdot)$ for \eqref{EQ: AdjointDelayComputationYEquation2} in $C([0,\tau_{1}];\mathbb{R}^{n})$ can be established in the standard way. For such $y(\cdot)$, from \eqref{EQ: DelayDualMeasurement0Operator} we obtain that $\mathcal{I}_{y}(t)(0)$ is a continuous function of $t \in [0,\tau_{1}]$. Then the $C^{1}$-differentiability of the constructed $y(\cdot)$ follows from the integral equation. By substituting such $y(\cdot)$ into \eqref{EQ: DelayDualMeasurement0Operator} again and applying the Lebesgue differentiation theorem, we see that $(\mathcal{I}_{y})(\cdot)(0)$ is an element of $W^{1,2}(0,\tau_{1};\mathbb{R}^{n})$. By the same reasoning, $(\mathcal{I}_{y})(\cdot)(-\tau_{j}-0)$ is also an element of that Sobolev space for any $j \in \{1,\ldots,J\}$.
	
	Now let us show that the constructed $y(\cdot)$ allows us to extend each $\eta_{j}$ to the interval $[\tau_{j+1}-\tau_{j}, \tau_{j+1}-\tau_{j} + \tau_{1}]$. We do this by induction as follows.
	\begin{enumerate}
		\item[1)] From \eqref{EQ: DelayAdjCompBoundConds1} with $t \in [0,\tau_{1}]$, we determine $\eta_{J}$ as an element of $C^{1}([\tau-\tau_{J},\tau - \tau_{J} +\tau_{1}];\mathbb{R}^{n})$. Since $(y_{0},\psi_{0}) \in \mathcal{D}$, we have
		\begin{equation}
			\eta_{J}(\tau-\tau_{J} - 0) = \rho(-\tau)\psi_{0}(-\tau) = L^{*}_{-\tau}y_{0} = \eta_{J}(\tau-\tau_{J}+0).
		\end{equation}
		Thereby, $\eta_{J}$ is determined as an element of $W^{1,2}(0,\tau - \tau_{J} +\tau_{1};\mathbb{R}^{n})$.
		\item[2)] For $j \in \{1,\ldots,J\}$, suppose $\eta_{j}$ is already determined as an element from $W^{1,2}(0,\tau_{j+1}-\tau_{j}+\tau_{1};\mathbb{R}^{n})$. Consider \eqref{EQ: DelayAdjCompBoundConds2} with such $j$ and $t \in [0,\tau_{1}]$. This determines $\eta_{j-1}$ as an element of $W^{1,2}(\tau_{j}-\tau_{j-1},\tau_{j}-\tau_{j-1}+\tau_{1};\mathbb{R}^{n})$. Moreover, since $(y_{0},\psi_{0}) \in \mathcal{D}$, we have $L^{*}_{-\tau_{j}}y_{0} = \mathfrak{D}^{\rho}_{j}(\psi_{0})$. From this and $\mathcal{I}_{y}(0) = 0$, putting $t=0$ in the determining relation gives 
		\begin{equation}
			\begin{split}
				\eta_{j-1}(\tau_{j}-\tau_{j-1}+0) = \mathfrak{D}^{\rho}_{j}(\psi_{0}) + \rho(-\tau_{j}-0)\psi_{0}(-\tau_{j}-0) = \\ = \rho(-\tau_{j}+0)\psi_{0}(-\tau_{j}+0) = \eta_{j-1}(\tau_{j}-\tau_{j-1}-0).
			\end{split}
		\end{equation}
		Consequently, $\eta_{j-1}$ is determined as an element of $W^{1,2}(0,\tau_{j} - \tau_{j-1} +\tau_{1};\mathbb{R}^{n})$.
	\end{enumerate}
	
	As a result, we determined $\eta_{0}$ on $[0,2\tau_{1}]$. From this, we can return to the second equation from \eqref{EQ: AdjointDelayComputationYEquation2} and consider it on $[\tau_{1},2\tau_{1}]$. Note that the right-hand sides agree at $t=\tau_{1}$, so the resulting $y(\cdot)$ will be continuously differentiable on $[0,2\tau_{1}]$. Then we may repeat the previous procedure, extending $\eta_{j}$ and so on. Since $y(\cdot)$ is $C^{1}$-differentiable, the agreement of extensions is more trivial in this case, as it immediately follows from \eqref{EQ: DelayAdjCompBoundConds1} and \eqref{EQ: DelayAdjCompBoundConds2}.
\end{proof}

From Theorem \ref{TH: SymmetrizationDelayOperator}, it is clear that $\mathcal{D}(L) \cap \mathcal{D}(L^{*})$ is dense in $\mathbb{H}$ if and only if $\Delta_{j}(\rho) \not= 0$ for any $j \in \{1,\ldots,J\}$. Indeed, for $(x,\phi) \in \mathcal{D}(L) \cap \mathcal{D}(L^{*})$ we have $\phi \in W^{1,2}(-\tau,0;\mathbb{R}^{n})$ and, consequently, $\mathfrak{D}^{\rho}_{j}(\phi) = \Delta_{j}(\rho) \phi(-\tau_{j})$. If $\Delta_{j}(\rho) = 0$, then the boundary condition $L^{*}_{-\tau_{j}}x - \mathfrak{D}^{\rho}_{j}(\phi) = 0$ reads as $L^{*}_{-\tau_{j}}x = 0$. Thus, one cannot approximate elements $(y,\psi)$ of the space $\mathbb{H}$ with $L^{*}_{-\tau_{j}}y \not= 0$ by the elements from $\mathcal{D}(L) \cap \mathcal{D}(L^{*})$. On the other hand, if $\Delta_{j}(\rho) \not= 0$ for each $j \in \{1,\ldots,J\}$, then there are no such limits, and the boundary conditions restrict $\phi$ only at a finite number of points, which does not affect the possibility of approximation in the $L_{2}$-norm.

It can also be shown that in the degenerate case where $\Delta_{j}(\rho) = 0$ for some $j \in \{1,\ldots,J\}$, we do not have any densely defined symmetric extensions of $L+L^{*}$. In the nondegenerate case, the following theorem establishes that there always exists a unique additive symmetrization.

\begin{theorem}
	\label{TH: DelayOperatorSymmetrizationExistence}
	Let $\mathbb{H}$ be endowed with the inner product from \eqref{EQ: DelayEqsSymWeightedInnerProduct}, and suppose that $\Delta_{j}(\rho) \not= 0$ for each $j \in \{1,\ldots,J\}$. Then the operator $L$ given by \eqref{EQ: OperatorLSymmetrizationDelayEqsDef} and satisfying \eqref{EQ: OperatorLDelaySymMeasureRepresentation} admits a unique additive symmetrization $S_{L}$, which is a bounded self-adjoint operator in $\mathbb{H}$ given by
	\begin{equation}
		\label{EQ: DelayOperatorAddSymmExistence}
		\begin{split}
			S_{L}(x,\phi) = \frac{1}{2} (y,\psi), \text{ where }\\
			y = \left(L_{0}+L^{*}_{0} + \rho(0)I_{n} + \frac{L_{-\tau}L^{*}_{-\tau}}{\rho(-\tau)} +\sum_{j=1}^{J} \frac{L_{-\tau_{j}}L^{*}_{-\tau_{j}}}{\Delta_{j}(\rho)}\right)x + \sum_{i=1}^{I}\int_{-\tau}^{0}M_{i}(\theta)\phi(\theta)d\theta,\\
			\psi(\theta) = -\frac{\rho'(\theta)}{\rho(\theta)} \phi(\theta) + \frac{\sum_{i=1}^{I}M^{*}_{i}(\theta)}{\rho(\theta)}x \qquad \text{for} \quad \theta \in (-\tau,0),
		\end{split}
	\end{equation}
	where $I_{n}$ is the unit $n\times n$-matrix, and the adjoints of matrices are taken with respect to the inner product $\langle \cdot, \cdot \rangle_{\mathbb{R}^{n}}$ in $\mathbb{R}^{n}$ from \eqref{EQ: DelayEqsSymWeightedInnerProduct}.
\end{theorem}
\begin{proof}
	Combining \eqref{EQ: OperatorLDelaySymMeasureRepresentation} and \eqref{EQ: AdjointDelayOperatorAction} yields that for $(x,\phi) \in \mathcal{D}(L) \cap \mathcal{D}(L^{*})$, the image $(y,\psi) = (L+L^{*})(x,\phi)$ is given by \eqref{EQ: DelayOperatorAddSymmExistence}, since we have $L^{*}_{-\tau} x = \rho(-\tau) \phi(-\tau)$ and $L^{*}_{-\tau_{j}} x = \Delta_{j}(\rho)\phi(-\tau_{j})$ for each $j \in \{1,\ldots,J\}$. Moreover, $\mathcal{D}(L) \cap \mathcal{D}(L^{*})$ is dense in $\mathbb{H}$, and the formula \eqref{EQ: DelayOperatorAddSymmExistence} defines a bounded operator. From this, the conclusion follows immediately.
\end{proof}

Now we are going to investigate conditions for the additive symmetrization $S_{L}$ from Theorem \ref{TH: DelayOperatorSymmetrizationExistence} to be proper. For what follows, we consider $\rho(\theta)$ given by
\begin{equation}
	\label{EQ: DelayWeightFunctionWithDisc}
	\rho(\theta) = e^{\varkappa_{j} \theta} \qquad \text{for} \quad \theta \in (-\tau_{j+1},-\tau_{j}) \quad \text{and} \quad j \in \{0, \ldots, J\}
\end{equation}
with some real parameters $\varkappa_{0},\ldots,\varkappa_{J}$. Clearly, $\Delta_{j}(\rho) = e^{-\varkappa_{j-1} \tau_{j}} - e^{-\varkappa_{j}\tau_{j}}$. Moreover, $\Delta_{j}(\rho) > 0$ if and only if $\varkappa_{j-1} < \varkappa_{j}$.

\begin{theorem}
	\label{TH: TraceNumbersAdmitSymmetrizationDelayOperator}
	Let $\mathbb{H}$ be endowed with the inner product from \eqref{EQ: DelayEqsSymWeightedInnerProduct} with $\rho(\theta)$ as in \eqref{EQ: DelayWeightFunctionWithDisc}, which additionally satisfies $\Delta_{j}(\rho) > 0$ for each $j \in \{1,\ldots,J\}$. Then for the operator $L$ given by \eqref{EQ: OperatorLSymmetrizationDelayEqsDef} and satisfying \eqref{EQ: OperatorLDelaySymMeasureRepresentation}, the trace numbers $\beta_{1}(L), \beta_{2}(L), \ldots$ (see \eqref{EQ: TraceNumbersDef}) satisfy for any $k=1,2,\ldots$ (note that $\mathbb{H}$ is real)
	\begin{equation}
		\label{EQ: TraceNumbersDelayOperator}
		\sum_{j=1}^{k}\beta_{j}(L) = \sup_{\substack{\mathbb{L} \subset \mathcal{D}(L) \cap \mathcal{D}(L^{*})\\\dim\mathbb{L}=k}}\operatorname{Tr}\left(\Pi_{\mathbb{L}} \circ L \circ \Pi_{\mathbb{L}}\right),
	\end{equation}
	where the supremum is taken over linear subspaces $\mathbb{L}$, $\Pi_{\mathbb{L}}$ denotes the orthogonal projector onto $\mathbb{L}$, $L^{*}$ is given by Theorem \ref{TH: SymmetrizationDelayOperator}, and $\operatorname{Tr}$ denotes the trace functional. In particular, the additive symmetrization $S_{L}$ from Theorem \ref{TH: DelayOperatorSymmetrizationExistence} is proper.
\end{theorem}
\begin{proof}
	First, we prove the statement for $k=1$. Let $e = (x,\phi) \in \mathcal{D}(L)$ be such that $|e|_{\mathbb{H}} = 1$. Let $|\cdot|_{\mathbb{R}^{n}}$ be the norm induced by $\langle \cdot, \cdot \rangle_{\mathbb{R}^{n}}$. From \eqref{EQ: DelayWeightFunctionWithDisc}, we obtain
	\begin{equation}
		\begin{split}
			2\int_{-\tau}^{0}\rho(\theta) \left\langle \phi(\theta), \phi'(\theta) \right\rangle_{\mathbb{R}^{n}}d\theta = \sum_{j=0}^{J}2\int_{-\tau_{j+1}}^{-\tau_{j}}e^{\varkappa_{j} \theta} \left\langle \phi(\theta), \phi'(\theta) \right\rangle_{\mathbb{R}^{n}}d\theta = \\ \sum_{j=0}^{J} \left( e^{-\varkappa_{j}\tau_{j}}|\phi(-\tau_{j})|^{2}_{\mathbb{R}^{n}} - e^{-\varkappa_{j}\tau_{j+1}} |\phi(-\tau_{j+1})|^{2}_{\mathbb{R}^{n}} \right) - \sum_{j=0}^{J}\varkappa_{j}\int_{-\tau_{j+1}}^{-\tau_{j}}e^{\varkappa_{j} \theta}|\phi(\theta)|^{2}_{\mathbb{R}^{n}}d\theta \\= |\phi(0)|^{2}_{\mathbb{R}^{n}} - e^{-\varkappa_{J}\tau}|\phi(-\tau)|^{2}_{\mathbb{R}^{n}} - \sum_{j=1}^{J} \Delta_{j}(\rho)|\phi(-\tau_{j})|^{2}_{\mathbb{R}^{n}} \\- \sum_{j=0}^{J}\varkappa_{j}\int_{-\tau_{j+1}}^{-\tau_{j}}e^{\varkappa_{j} \theta}|\phi(\theta)|^{2}_{\mathbb{R}^{n}}d\theta.
		\end{split}
	\end{equation}
	
	Substituting the above, $\phi(0)=x$, and \eqref{EQ: DelayWeightFunctionWithDisc} into \eqref{EQ: OperatorLDelayBilinearForm}, we obtain
	\begin{equation}
		\label{EQ: QuadraticFormDelayOperatorDecomposition}
		\begin{split}
			\langle Le,e \rangle_{\rho} = \left\langle \phi(-\tau), L^{*}_{-\tau}\phi(0) - \frac{1}{2}e^{-\varkappa_{J} \tau}\phi(-\tau)\right\rangle_{\mathbb{R}^{n}} \\+ \sum_{j=1}^{J} \left\langle \phi(-\tau_{j}), L^{*}_{-\tau_{j}}\phi(0) - \frac{1}{2}\Delta_{j}(\rho)\phi(-\tau_{j}) \right\rangle_{\mathbb{R}^{n}} + \cdots,
		\end{split}
	\end{equation}
	where $\cdots$ denotes the terms that are continuous in the norm of $\mathbb{H}$.
	
	Consider the quadratic functional
	\begin{equation}
		\label{EQ: DelaySymmQuadFunc1}
		\mathbb{R}^{n} \ni y \mapsto \left\langle y, L^{*}_{-\tau}\phi(0) - \frac{1}{2}e^{-\varkappa_{J} \tau}y\right\rangle_{\mathbb{R}^{n}}.
	\end{equation}
	It is clear that it admits a unique maximum at $y=y_{J+1}$ satisfying $L^{*}_{-\tau} \phi(0) = e^{-\varkappa_{J}\tau}y_{J+1}$.
	
	Similarly, for $j \in \{1,\ldots,J\}$, from $\Delta_{j}(\rho) > 0$, we conclude that the quadratic functional
	\begin{equation}
		\label{EQ: DelaySymmQuadFunc2}
		\mathbb{R}^{n} \ni y \mapsto \left\langle y, L^{*}_{-\tau_{j}}\phi(0) - \frac{1}{2}\Delta_{j}(\rho)y\right\rangle_{\mathbb{R}^{n}}
	\end{equation}
	admits a unique maximum at $y=y_{j}$ satisfying $L^{*}_{-\tau_{j}}\phi(0) = \Delta_{j}(\rho)y_{j}$.
	
	For any sufficiently small $\varepsilon>0$, we set $\phi_{\varepsilon}(\cdot) \in W^{1,2}(-\tau,0;\mathbb{R}^{n})$ identical to $\phi(\theta)$ for any $\theta \in [-\tau,0]$, avoiding the $\varepsilon$-neighborhoods of $-\tau_{j}$ for each $j \in \{1,\ldots,J+1\}$; $\phi_{\varepsilon}(-\tau_{j}) \coloneq y_{j}$ for each $j \in \{1,\ldots,J+1\}$; and $\phi_{\varepsilon}(\cdot)$ is the simplest piecewise affine function for the remaining $\theta$. Next, we define $e_{\varepsilon}$ as the normalized with respect to $\langle \cdot, \cdot \rangle_{\rho}$ vector $(x,\phi_{\varepsilon})$. Clearly, $e_{\varepsilon} \in \mathcal{D}(L) \cap \mathcal{D}(L^{*})$, and $e_{\varepsilon} \to e = (x,\phi)$ in $\mathbb{H}$ as $\varepsilon \to 0+$. From this and \eqref{EQ: QuadraticFormDelayOperatorDecomposition}, we obtain
	\begin{equation}
		\label{EQ: TraceNumbersSymmetrizationDelayOperatorConstruction}
		\langle Le,e \rangle_{\rho} \leq \lim_{\varepsilon \to 0+} \langle Le_{\varepsilon},e_{\varepsilon} \rangle_{\rho},
	\end{equation}
	or even $\langle Le,e \rangle_{\rho} < \langle Le_{\varepsilon},e_{\varepsilon} \rangle_{\rho}$ for all sufficiently small $\varepsilon>0$, provided that $e \notin \mathcal{D}(L^{*})$. This shows the statement for $k=1$.
	
	For general $k \geq 1$, we take a $k$-dimensional subspace $\mathbb{L} \subset \mathcal{D}(L)$ and vectors $e_{1},\ldots,e_{k} \in \mathbb{L}$ forming an orthonormal basis. For all sufficiently small $\varepsilon > 0$ and $l \in \{1,\ldots,k\}$, let $\widetilde{e}^{\varepsilon}_{l}$ be the vector $e_{\varepsilon}$ constructed above \eqref{EQ: TraceNumbersSymmetrizationDelayOperatorConstruction} for $e \coloneq e_{l}$. Finally, we define $e^{\varepsilon}_{1}, \ldots, e^{\varepsilon}_{k}$ as the vectors obtained from the Gram-Schmidt process (with normalization) applied to $\widetilde{e}^{\varepsilon}_{1}, \ldots, \widetilde{e}^{\varepsilon}_{k}$. From this and \eqref{EQ: TraceNumbersSymmetrizationDelayOperatorConstruction}, we obtain 
	\begin{equation}
		\begin{split}
			\operatorname{Tr}\left(\Pi_{\mathbb{L}} \circ L \circ \Pi_{\mathbb{L}}\right) = \sum_{l=1}^{k}\langle Le_{l},e_{l} \rangle_{\rho} \\ \leq \lim_{\varepsilon \to 0+}\sum_{l=1}^{k}\langle L\widetilde{e}^{\varepsilon}_{l}, \widetilde{e}^{\varepsilon}_{l} \rangle_{\rho} = \lim_{\varepsilon \to 0+}\sum_{l=1}^{k}\langle Le^{\varepsilon}_{l}, e^{\varepsilon}_{l} \rangle_{\rho}.
		\end{split}
	\end{equation}
	Thereby, \eqref{EQ: TraceNumbersDelayOperator} is established. From this and since $S_{L}$ is bounded, the relation $\beta_{k}(L) = \beta_{k}(S_{L})$ is satisfied for any $k=1,2,\ldots$.The proof is finished.
\end{proof}

\begin{remark}
	Let us emphasize that \eqref{EQ: TraceNumbersDelayOperator} is concerned with the surprising fact that the conditions satisfied by the unique maxima of the quadratic functionals \eqref{EQ: DelaySymmQuadFunc1} and \eqref{EQ: DelaySymmQuadFunc2} are related to the boundary conditions \eqref{EQ: DelayAdjointBoundaryConditions} for the adjoint operator $L^{*}$.
\end{remark}

\begin{remark}
	\label{REM: DelayOperatorsNoProperSymm}
	From the proof of Theorem \ref{TH: TraceNumbersAdmitSymmetrizationDelayOperator}, it is clear that if the condition $\Delta_{j}(\rho) > 0$  is violated for some $j \in \{1,\ldots,J\}$, then $\beta_{k}(L) = +\infty$ for any $k=1,2,\ldots$. This happens because the corresponding quadratic functional \eqref{EQ: DelaySymmQuadFunc2} becomes unbounded from above (with the additional requirement $L^{*}_{-\tau_{j}}\phi(0) \not= 0$ in the case $\Delta_{j}(\rho) = 0$). If, additionally, $\Delta_{j}(\rho) \not= 0$ for each $j \in \{1,\ldots,J\}$, this implies that the additive symmetrization $S_{L}$ is not proper, since each $\beta_{k}(S_{L})$ is finite due to the boundedness of $S_{L}$.
\end{remark}

Now consider \eqref{EQ: OperatorLDelaySymMeasureRepresentation} with $I=0$. Then, under the hypotheses of Theorem \ref{TH: TraceNumbersAdmitSymmetrizationDelayOperator}, there exists a unique proper additive symmetrization $S_{L}$ of $L$. From \eqref{EQ: DelayOperatorAddSymmExistence}, it is clear that the spectrum of $2S_{L}$ is constituted by the eigenvalues of the matrix 
\begin{equation}
	\label{EQ: MatrixDelaySymmetrization}
	L_{0}+L^{*}_{0} + e^{\varkappa_{J}\tau}L_{-\tau}L^{*}_{-\tau} +\sum_{j=1}^{J} \frac{L_{-\tau_{j}}L^{*}_{-\tau_{j}}}{e^{-\varkappa_{j-1}\tau_{j}} - e^{-\varkappa_{j}\tau_{j}}} + I_{n},
\end{equation}
corresponding to the spectral subspace $\mathbb{R}^{n} \times \{0\}$, and by the eigenvalues $-\varkappa_{j}$ taken over $j \in \{0,\ldots,J\}$. Each $-\varkappa_{j}$ has infinite multiplicity, and its spectral subspace consists of $(0,\phi) \in \mathbb{H}$ such that $\phi$ vanishes outside of $(-\tau_{j+1},-\tau_{j})$. Taken together, the spectral subspaces of $-\varkappa_{j}$ span the subspace $\{0\} \times L_{2}(-\tau,0;\mathbb{R}^{n})$. 

Given $m \geq 1$, let $\alpha_{1}(S_{L}) \geq \alpha_{2}(S_{L}) \geq \cdots \geq \alpha_{m}(S_{L})$ be the largest $m$ eigenvalues of $S_{L}$ counting multiplicities. Then, it is clear that the sum $\alpha_{1}(S_{L}) + \cdots + \alpha_{m}(S_{L})$ is negative if $m$ is taken sufficiently large and $\varkappa_{0} > 0$.

Combining the above considerations with Theorem \ref{TH: LiouvilleSymmetrizationSummary} gives the following. 
\begin{theorem}
	\label{TH: TheoremLyapExponentsEstimateDelayEqs}
	Let $\mathcal{P} \subset \mathcal{Q} \times \mathbb{E}$ be a closed subset that is positively invariant under $\pi$ from \eqref{EQ: DelaySkewProductSemiflowAssCoc}, and let $V \colon \mathcal{P} \to \mathbb{R}$ be a bounded function satisfying \nameref{DESC: LyapFuncV1} and \nameref{DESC: LyapFuncV2} with respect to the semiflow $(\mathcal{P},\pi)$. Suppose that for any $\wp \in \mathcal{P}$, the operators $A(\wp)$ from \eqref{EQ: DelaySymmetrizationOperatorAqFormula} have the form as $L$ in \eqref{EQ: OperatorLSymmetrizationDelayEqsDef} and \eqref{EQ: OperatorLDelaySymMeasureRepresentation} with $I=0$ and the corresponding $(n\times n)$-matrices $L_{0} = L_{0}(\wp)$, $L_{-\tau} = L_{-\tau}(\wp)$, and $L_{-\tau_{j}} = L_{-\tau_{j}}(\wp)$ for $j \in \{1,\ldots,J\}$. Let $\varkappa_{0} < \varkappa_{1} < \cdots < \varkappa_{J}$ be given, and let $\lambda_{1}(\wp) \geq \cdots \geq \lambda_{n}(\wp)$ be the eigenvalues (counting multiplicities) of the symmetric matrix (here $M^{T}$ denotes the transpose\footnote{Instead of taking transposes, one may also consider the adjoint matrices with respect to any fixed inner product in $\mathbb{R}^{n}$ as in \eqref{EQ: MatrixDelaySymmetrization}.} of a matrix $M$)
	\begin{equation}
		\label{EQ: MatrixL0LtauDelaySymLyapExpsEstimate}
		L_{0}(\wp)+L^{T}_{0}(\wp) + e^{\varkappa_{J}\tau}L_{-\tau}(\wp)L^{T}_{-\tau}(\wp) + \sum_{j=1}^{J} \frac{L_{-\tau_{j}}(\wp)L^{T}_{-\tau_{j}}(\wp)}{e^{-\varkappa_{j-1}\tau_{j}} - e^{-\varkappa_{j}\tau_{j}} } + I_{n}.
	\end{equation}
	Let $0 \leq K(\wp) \leq n$ be the largest number $k$ such that $\lambda_{k}(\wp) \geq -\varkappa_{0}$. For any $m \geq 1$, define the quantity
	\begin{equation}
		\label{EQ: DelaySymmetrizationCocycleApplicationAlphaPlusDef}
		\alpha^{+}(m) \coloneq \sup_{\wp \in \mathcal{P}}\left( \dot{V}(\wp) + \frac{1}{2}\sum_{k=1}^{\min\{m,K(\wp)\}}\lambda_{k}(\wp) - \frac{\varkappa_{0}}{2} \cdot \max\{0,m-K(\wp)\}\right).
	\end{equation}
	
	Then the cocycle $\Xi$ given by \eqref{EQ: DelayLinearizationCocycleDef} over the semiflow $(\mathcal{P},\pi)$ satisfies
	\begin{equation}
		\label{EQ: DelaySymThLyapExpEstimatesAlphaPlus}
		\lambda_{1}(\Xi) + \cdots + \lambda_{m}(\Xi) \leq \alpha^{+}(m).
	\end{equation}
\end{theorem}
\begin{proof}
	We apply Theorem \ref{TH: LiouvilleSymmetrizationSummary} with $S(\wp)$ being the additive symmetrization of $A(\wp)$. In terms of the theorem, we have $\alpha^{+}_{\mathfrak{n}^{\wedge m}}(\Xi_{m}) = \alpha^{+}(m)$, where the metric $\mathfrak{n}^{\wedge m}$ on $\mathbb{H}^{\wedge m}$ is given by \eqref{EQ: AdaptedMetricOnExteriorProductViaLyapFunc} over $\mathcal{P}$ when $\mathbb{H}$ is endowed with the metric $\langle \cdot, \cdot \rangle_{\rho}$ from \eqref{EQ: DelayEqsSymWeightedInnerProduct}, where $\langle \cdot, \cdot \rangle_{\mathbb{R}^{n}}$ is the standard Euclidean product. Since $V$ is bounded, $\mathfrak{n}^{\wedge m}$ is equivalent to the standard metric. Then \eqref{EQ: DelaySymThLyapExpEstimatesAlphaPlus} follows from $\alpha^{+}_{\mathfrak{n}^{\wedge m}}(\Xi_{m}) \geq \lambda_{1}(\Xi_{m}) = \lambda_{1}(\Xi) + \cdots + \lambda_{m}(\Xi)$.
\end{proof}

\begin{remark}
	\label{REM: DelayDimEstimatesScaling}
	It is important to note that the estimate in \eqref{EQ: DelaySymThLyapExpEstimatesAlphaPlus} depends (in a nontrivial way) on the spatio-temporal scaling $x(t) \mapsto x(\kappa t)$, $\phi(t,\theta) \mapsto \phi(\kappa t,\kappa \theta)$ in \eqref{EQ: DelayLinearCocAbsract} and \eqref{EQ: LinearizedEquationDelayInRn} with some $\kappa > 0$. Although uniform Lyapunov exponents depend on such a change (they scale by $\kappa$), both variants of the Lyapunov dimension do not, that is clear from their definitions in \eqref{EQ: LyapunovDimensionDefinition} and \eqref{EQ: Kaplan-YorkeFormula}. In applications, this suggests considering a family of estimates for the Lyapunov dimension, which depend on $\kappa > 0$. Then minimizing with respect to $\kappa$ will give a sharper estimate (see above \eqref{EQ: ExampleSuarezSchopfPertLyapDimEst} for an illustrative example).
\end{remark}

Below we are going to apply Theorem \ref{TH: TheoremLyapExponentsEstimateDelayEqs} to provide dimension estimates for cocycles over global attractors in the Mackey--Glass equations and the perturbed Suarez--Schopf delayed oscillator (see \cite{AnikushinRomanov2024EffEst} for applications to the Nicholson blowflies model). For such and other scalar equations, it is convenient to establish the following lemma.
\begin{lemma}
	\label{LEM: ScalarDelayDimEstLemma}
	Let $a, b \in \mathbb{R}$ be such that $a + b \geq 0$ and $b > 0$. For $\tau > 0$, consider the function of $\varkappa > 0$ defined by
	\begin{equation}
		d^{*}(\varkappa) \coloneq \frac{a + b e^{\varkappa \tau}}{\varkappa} + 1.
	\end{equation}
	Then $d^{*}(\varkappa)$ admits a global minimum given by $\tau b e^{p^{*}+1} + 1$ at $\varkappa = \tau^{-1}(p^{*}+1)$, where $p^{*}$ is the unique root $p \geq -1$ of
	\begin{equation}
		\label{EQ: DimEstPRootGeneral}
		pe^{p+1} = \frac{a}{b}.
	\end{equation}
\end{lemma}
\begin{proof}
	Clearly, for $a + b > 0$, the minimum point $\varkappa$ is positive and satisfies $a + be^{\varkappa \tau} = b \varkappa \tau e^{\varkappa \tau}$ or, equivalently, $(\tau \varkappa - 1) e^{\varkappa \tau} = a/b$. Putting in the latter $\varkappa = \tau^{-1}(p+1)$ for $p \geq - 1$, we obtain \eqref{EQ: DimEstPRootGeneral}. Note that $p^{*}=-1$ corresponds to $\varkappa = 0$ and $a + b = 0$, in which case $d(\varkappa)$ is defined in the limit $\varkappa \to 0+$ by $b\tau + 1$, and this is indeed the global minimum.
\end{proof}

%% file: MackeyGlassEqs.tex
\subsection{Mackey--Glass equations}
\label{SEC: Mackey-GlassDimEst}
In this section, we consider a class of nonlinear scalar delay equations suggested by Mackey and Glass in \cite{MackeyGlass1977} as a model for certain physiological processes. It is described by
\begin{equation}
	\label{EQ: MackeyGlassExample}
	\dot{x}(t) =  - \gamma x(t) + \beta F(x(t-\tau)),
\end{equation}
where $\tau,\beta, \gamma>0$ are real parameters, and the nonlinearity $F$ is given by
\begin{equation}
	F(y) = \frac{y}{1+|y|^{k}} \qquad \text{for any} \quad y \in \mathbb{R},
\end{equation}
where $k>1$ is a real parameter.

Straightforward calculations show that the derivative $F'$ of $F$ satisfies
\begin{equation}
	-\frac{(k-1)^{2}}{4k} \leq F'(y) \leq 1 \qquad \text{for any} \quad y \in \mathbb{R}.
\end{equation}
In particular, since \eqref{EQ: MackeyGlassExample} is autonomous and belongs to the class described by \eqref{EQ: DelayRnNonlinearNonautonomous}, it generates a semiflow $\pi$ in $\mathbb{H} = \mathbb{R} \times L_{2}(-\tau,0;\mathbb{R})$ or $\mathbb{E} = C([-\tau,0];\mathbb{R})$ (see \eqref{EQ: NonlinCocycleDelayEqsDefinition} and below).

Let us establish the existence of invariant regions and the global attractor for $\pi$ in $\mathbb{E}$. Since $\pi^{\tau}(\mathbb{H}) \subset \mathbb{E}$ when $\mathbb{E}$ is embedded into $\mathbb{H}$ as in \eqref{EQ: DelayCocNonlinearSmoothingEst}, the same properties can be translated for the semiflow $\pi$ in $\mathbb{H}$.

Let $\mathcal{B}_{R}(0)$ be the closed ball in $\mathbb{E}$ of radius $R \geq 0$ centered at $0$. By $\mathring{\mathcal{B}}_{R}(0)$ we denote its interior.
\begin{lemma}
	\label{LEM: MackeyGlassAttractorBound}
	The set $\mathcal{B}_{R}(0)$ is positively invariant, i.e., $\pi^{t}(\mathcal{B}_{R}(0)) \subset \mathcal{B}_{R}(0)$ for any $t \geq 0$, provided that $R \geq R_{0}$, where
	\begin{equation}
		\label{EQ: MackeyGlassR0Bound}
		R_{0} = \begin{cases}
			0 \qquad &\text{if} \quad \beta \leq \gamma,\\
			\beta \gamma^{-1} k^{-1}(k - 1)^{\frac{k-1}{k}} \qquad &\text{if} \quad \beta > \gamma.
		\end{cases}
	\end{equation}
\end{lemma}
\begin{proof}
	Let $\phi_{0} \in \mathring{\mathcal{B}}_{R}(0)$ for some $R>0$, and suppose that the trajectory $\pi^{t}(\phi_{0})$ leaves $\mathring{\mathcal{B}}_{R}(0)$ for some $t > 0$. Then for the corresponding solution $x(t)=x(t;\phi_{0})$ to \eqref{EQ: MackeyGlassExample}, there exists $t_{0} > 0$ such that $|x(t_{0})| = R$ and $|x(t)| < R$ for any $t \in [-\tau,t_{0})$. Since $-x(t)=x(t;-\phi_{0})$ is also a solution to \eqref{EQ: MackeyGlassExample}, we may assume that $x(t_{0}) = R$. Then
	\begin{equation}
		\label{EQ: MackeyGlassInvariantRegionsLemmaEq}
		\dot{x}(t_{0}) = -\gamma R + \beta \frac{x(t_{0}-\tau)}{1+|x(t_{0}-\tau)|^{k}}.
	\end{equation}
    First, we consider the rough estimate
    \begin{equation}
    	\dot{x}(t_{0}) < -\gamma R + \beta R = (\beta - \gamma) R.
    \end{equation}
    For $\beta \leq \gamma$, this gives $\dot{x}(t_{0}) < 0$ which contradicts our assumptions for any $R > 0$.
    
    Since the maximum of the function $y/(1+y^{k})$ on $[0,\infty)$ is given by  $k^{-1}(k - 1)^{\frac{k-1}{k}}$, from \eqref{EQ: MackeyGlassInvariantRegionsLemmaEq} we obtain
    \begin{equation}
    	\dot{x}(t_{0}) \leq -\gamma R + \beta k^{-1}(k - 1)^{\frac{k-1}{k}}.
    \end{equation}
    Thus, $\dot{x}(t_{0}) < 0$ for $R > \beta \gamma^{-1} k^{-1}(k - 1)^{\frac{k-1}{k}}$ that also leads to a contradiction.
    
    Consequently, $\mathring{\mathcal{B}}_{R}(0)$ is positively invariant for any $R > R_{0}$. Then $\mathcal{B}_{R}(0)$ is positively invariant as the closure of $\mathring{\mathcal{B}}_{R}(0)$, and $\mathcal{B}_{R_{0}}(0)$ is positively invariant as the intersection of $\mathcal{B}_{R}(0)$ over $R > R_{0}$.
\end{proof}

\begin{lemma}
	\label{LEM: MackeyGlassAbsorbingBall}
	The set $\mathcal{B}_{R}(0)$ is $\mathcal{B}$-absorbing for any $R > R_{0}$, where $R_{0}$ is given by \eqref{EQ: MackeyGlassR0Bound}, i.e., for any bounded subset $\mathcal{B} \subset \mathbb{E}$, there exists $T=T(\mathcal{B}) \geq 0$ such that $\pi^{t}(\mathcal{B}) \subset \mathcal{B}_{R}(0)$ for every $t \geq T$.
\end{lemma}
\begin{proof}
	Since $\pi^{\tau}$ is a compact mapping of $\mathbb{E}$, it is sufficient to prove that $\mathcal{B}_{R}(0)$ is point-absorbing and then apply the standard compactness argument as in \cite[Proposition 1.7, Chapter I]{KuzReit2020}. 
	
	Any point $\phi_{0} \in \mathbb{E}$ lies in a sufficiently large closed ball centered at $0$, which is positively invariant due to Lemma \ref{LEM: MackeyGlassAttractorBound}. Consequently, the trajectory of $\phi_{0}$ has compact closure. Let us show that the $\omega$-limit set $\omega(\phi_{0})$ of $\phi_{0}$ lies in $\mathring{\mathcal{B}}_{R}(0)$. 
	
	Let $\bar{R}$ be the maximum of the supremum norm $\| \phi \|_{\infty}$ taken over all $\phi \in \omega(\phi_{0})$. Since any point from $\omega(\phi_{0})$ has an infinite past, there exists $\bar{\phi} \in \omega(\phi_{0})$ such that $|\bar{\phi}(0)| = \bar{R}$. As above, we may assume that $\bar{\phi}(0) = \bar{R}$. Then for the corresponding solution $x(t) = x(t;\bar{\phi})$ to \eqref{EQ: MackeyGlassExample}, we have
	\begin{equation}
		\dot{x}(0) = -\gamma \bar{R} + \beta \frac{x(-\tau)}{1+|x(-\tau)|^{k}}.
	\end{equation} 
    In the case $\bar{R} \geq R > R_{0}$, by the same argument as in the proof of Lemma \ref{LEM: MackeyGlassAttractorBound}, we obtain $\dot{x}(0) < 0$. Since $x(\cdot)$ has a local maximum at $0$, this leads to a contradiction.
\end{proof}

From Lemma \ref{LEM: MackeyGlassAbsorbingBall}, we immediately obtain the following.
\begin{corollary}
	\label{COR: MackeyGlassGlobalAttractor}
	There exists a global $\mathcal{B}$-attractor $\mathcal{A}$ for the semiflow $\pi$ generated by \eqref{EQ: MackeyGlassExample}. It is given by
	\begin{equation}
		\mathcal{A} = \bigcap_{t \geq 0} \pi^{t}(\mathcal{B}_{R}(0))
	\end{equation}
    for any $R \geq R_{0}$, where $R_{0}$ is defined by \eqref{EQ: MackeyGlassR0Bound}.  In particular, for $0 \leq \beta \leq \gamma$, the attractor coincides with the zero equilibrium, and for $\beta > \gamma \geq 0$, it is contained in the ball of radius $\beta \gamma^{-1} k^{-1}(k - 1)^{\frac{k-1}{k}}$.
\end{corollary}

For $\beta < \gamma$, the unique zero equilibrium $\phi^{0}(\cdot) \equiv 0$ has a negative leading real eigenvalue. For $\beta=\gamma$, the leading eigenvalue becomes zero, and $\phi^{0}$ undergoes a pitchfork bifurcation, i.e., the leading eigenvalue becomes positive for $\beta > \gamma$, and a pair of symmetric equilibria $\phi^{+}$ and $\phi^{-}$ given by $\phi^{\pm} \equiv \pm (\beta \gamma^{-1} - 1)^{1/k}$ is born.

Linearization of \eqref{EQ: MackeyGlassExample} along a given solution $y_{0} \colon [-\tau,\infty) \to \mathbb{R}$ gives
\begin{equation}
	\label{EQ: MackeyGlassLinearization}
	\dot{x}(t) = -\gamma x(t) + \beta F'( y_{0}(t-\tau) )x(t-\tau) 
\end{equation}
This is the boundary part of the corresponding equation in the abstract form \eqref{EQ: LinearizedEquationDelayInRn}. In terms of Theorem \ref{TH: TheoremLyapExponentsEstimateDelayEqs}, for $\wp \in C([-\tau,0];\mathbb{R})$ we have (note that $J=0$)
\begin{equation}
	\label{EQ: MackeyGlassOperatorsL0andLtau}
	L_{0}(\wp) = -\gamma \quad \text{and} \quad L_{-\tau}(\wp) = \beta F'(\wp(-\tau)).
\end{equation}
Thus, for a given $\varkappa = \varkappa_{0} = \varkappa_{J} \geq 0$, the $(1 \times 1)$-matrix from \eqref{EQ: MatrixL0LtauDelaySymLyapExpsEstimate} has the eigenvalue $\lambda_{1}(\wp) = 1-2\gamma + \beta^{2} e^{\varkappa \tau} |F'(\wp(-\tau))|^{2}$. For $\beta > \gamma$, let $\Lambda=\Lambda(\gamma,\beta,k)$ be the maximum of $|F'(y)|$ over $|y| \leq R_{0} = \beta \gamma^{-1} k^{-1}(k - 1)^{\frac{k-1}{k}}$.

It is convenient to introduce the Lyapunov dimension $\operatorname{dim}_{\operatorname{L}}\mathcal{A}$ of $\mathcal{A}$ as the Lyapunov dimension $\operatorname{dim}_{\operatorname{L}}\Xi$ of the derivative cocycle $\Xi$ in $\mathbb{H}$ generated by \eqref{EQ: MackeyGlassLinearization} over the semiflow $\pi$ restricted to $\mathcal{A}$, as in \eqref{EQ: DelayLinearizationCocycleDef}. Note that restricting $\pi$ to any positively invariant ball $\mathcal{B}_{R}(0)$ yields the same value of $\operatorname{dim}_{\operatorname{L}}\Xi$ (see Lemma \ref{LEM: WeakConcetrationSubadditive} and \eqref{EQ: LyapDimTopErgCharacterization}), and the presence of such positively invariant subsets, which localize $\mathcal{A}$, gives the robustness of $\operatorname{dim}_{\operatorname{L}}\mathcal{A}$ in terms of Corollary \ref{COR: UpperSemiContUniformLyapExpAndLyapDim}.
\begin{theorem}
	\label{TH: MackeyGlassDimensionEstimate}
	Suppose that $\beta > \gamma \geq 0$. Then the Lyapunov dimension $\operatorname{dim}_{\operatorname{L}}\mathcal{A}$ of the global attractor $\mathcal{A}$ of the semiflow $\pi$ generated by \eqref{EQ: MackeyGlassExample} admits the estimate
	\begin{equation}
		\label{EQ: MackeyGlassLyapDimEstimate}
		\operatorname{dim}_{\operatorname{L}}\mathcal{A} \coloneq \operatorname{dim}_{\operatorname{L}}\Xi \leq (\beta \Lambda)^{2}e^{p^{*}+1} \cdot \tau + 1,
	\end{equation}
	where $p^{*}$ is the unique root $p \geq -1$ of 
	\begin{equation}
		\label{EQ: MackeyGlassLyapDimEstimateEqForP}
		p e^{p + 1} = \frac{1-2\gamma}{\beta^{2} \Lambda^{2}}. 
	\end{equation}
	Here $\Lambda$ can be taken as $\Lambda(\gamma,\beta,k)$ defined above or, more roughly, as $\max\left\{ 1, \frac{(k-1)^{2}}{4k}\right\}$.
\end{theorem}
\begin{proof}
	First, note that under $\beta > \gamma \geq 0$, we must always have $1-2\gamma + \beta^{2} \geq 0$ and, consequently, $1-2\gamma + \beta^{2} \Lambda^{2} \geq 0$ since $\Lambda \geq 1$. In the context of the theorem, from \eqref{EQ: MackeyGlassOperatorsL0andLtau}, it is clear that the value $\alpha^{+}(m)$ given by \eqref{EQ: DelaySymmetrizationCocycleApplicationAlphaPlusDef} for $\mathcal{P} = \mathcal{A}$, $V \equiv 0$, $\varkappa_{0} = \varkappa > 0$, and any integer $m \geq 1$ satisfies
	\begin{equation}
		\label{EQ: AlphaPlusMackeyGlassEstimate}
		\alpha^{+}(m) \leq \frac{1}{2}\left(1-2\gamma + \beta^{2} \Lambda^{2} e^{\varkappa \tau} - \varkappa (m-1) \right) \eqcolon \sigma(m),
	\end{equation}
	owing to Corollary \ref{COR: MackeyGlassGlobalAttractor}. Note that the formula for $\sigma(m)$ defines a decreasing and concave function of the real variable $m \geq 1$. By Remark \ref{REM: LyapDimEstApplications}, if a real number $d^{*} \geq 1$ satisfies $\sigma(d^{*}) = 0$, then we obtain the estimate $\operatorname{dim}_{\operatorname{L}}\mathcal{A} \leq d^{*}$. Since such $d^{*}$ is given by
	\begin{equation}
		d^{*} = \varkappa^{-1} ( 1 - 2 \gamma  + \beta^{2} \Lambda^{2} e^{\varkappa \tau} ) + 1,
	\end{equation}
	it is reasonable to minimize its value over $\varkappa > 0$. Here Lemma \ref{LEM: ScalarDelayDimEstLemma} with $a\coloneq1-2\gamma$ and $b\coloneq\beta^{2}\Lambda^{2}$ gives the desired. 
\end{proof}

According to Remark \ref{REM: DelayDimEstimatesScaling}, we can sharpen \eqref{EQ: MackeyGlassLyapDimEstimate} as follows.
\begin{corollary}
	\label{COR: MackeyGlassSharperDimEst}
	In terms of Theorem \ref{TH: MackeyGlassDimensionEstimate}, we have the sharper estimate
	\begin{equation}
		\label{EQ: MackeyGlassDimEstSharper}
		\operatorname{dim}_{\operatorname{L}}\Xi \leq  (\beta \Lambda)^{2} \inf_{\kappa > 0}  \left(\kappa e^{p^{*}(\kappa)+1}\right) \cdot \tau + 1,
	\end{equation}
	where $p^{*}(\kappa)$ is the unique root $p \geq -1$ of
	\begin{equation}
		pe^{p+1} = \frac{1-2\kappa \gamma}{(\kappa \beta \Lambda)^{2}},
	\end{equation}
	and $\Lambda$ can be taken as in \eqref{EQ: MackeyGlassLyapDimEstimateEqForP}.
\end{corollary}
\begin{proof}
	Using the scaling $x(t) \mapsto x(\kappa t)$, $\phi(t,\theta) \mapsto \phi(\kappa t,\kappa \theta)$, we have the following transformations in the parameters of \eqref{EQ: MackeyGlassExample}: $\gamma \mapsto \kappa \gamma$, $\beta \mapsto \kappa \beta$, $\tau \mapsto \kappa^{-1} \tau$, and $\Lambda \to \Lambda$. Then, applying \eqref{EQ: MackeyGlassLyapDimEstimate} with these new parameters and minimizing over $\kappa > 0$ yield \eqref{EQ: MackeyGlassDimEstSharper}.
\end{proof}

It is numerically shown in \cite{MackeyGlass1977} that the model \eqref{EQ: MackeyGlassExample} may exhibit chaotic behavior and, in particular, the attractor $\mathcal{A}$ may have a rich structure. For example, such behavior can be observed for the parameters $\beta = 0.2$, $\gamma = 0.1$, $k=10$, and $\tau > 10$. For these parameters, in terms of \eqref{EQ: MackeyGlassLyapDimEstimate}, we have $\Lambda = \frac{(k-1)^{2}}{4k} = \frac{81}{40}$, $p^{*} \approx 0.8034$, and $(\beta \Lambda)^{2}e^{p^{*}+1} \approx 0.9957$. Thus, \eqref{EQ: MackeyGlassLyapDimEstimate} yields the estimate
\begin{equation}
	\label{EQ: MackeyGlassClassicalParamteresEstimate}
	\operatorname{dim}_{\operatorname{L}}\mathcal{A} \leq 0.9958 \tau  + 1.
\end{equation}
As to \eqref{EQ: MackeyGlassDimEstSharper}, here the infimum is achieved at $\kappa \approx 1.00431$, and the result does not significantly differ from \eqref{EQ: MackeyGlassClassicalParamteresEstimate}, which corresponds to $\kappa = 1$.

Note that the Lyapunov dimension $\operatorname{dim}_{\operatorname{L}}\mathcal{A}$ always bounds the fractal dimension\footnote{Due to the smoothing estimate \eqref{EQ: DelayCocNonlinearSmoothingEst} and the invariance of $\mathcal{A}$, its fractal and Hausdorff dimensions do not depend on the choice of the metric from $\mathbb{E}$ or $\mathbb{H}$.} $\operatorname{dim}_{\operatorname{F}}\mathcal{A}$ of $\mathcal{A}$, see \cite{ChepyzhovIlyin2004}. However, the value of $\operatorname{dim}_{\operatorname{L}}\mathcal{A}$ usually does not reflect any geometric dimensions, such as $\operatorname{dim}_{\operatorname{F}}\mathcal{A}$, but rather their possible expansions under bifurcations and perturbations of the system. Nonetheless, such estimates as \eqref{EQ: MackeyGlassLyapDimEstimate} seem to be asymptotically sharp as $\tau \to \infty$. We are going to discuss this.

For $\beta = 0.2$, $\gamma = 0.1$, $k=10$, and $\tau = 22$, numerically computing eigenvalues at the symmetric equilibria $\phi^{\pm}$ shows that
\begin{equation}
	\sum_{j=1}^{14} \operatorname{Re}\lambda_{j} \geq 0.03 \quad \text{and} \quad \sum_{j=1}^{15} \operatorname{Re}\lambda_{j} < 0,
\end{equation}
where $\lambda_{j}$ are given by the roots $p \in \mathbb{C}$ of
\begin{equation}
	\label{EQ: MackeyGlassLinearizationSymmetricStates}
	-\gamma + \beta F'((\beta\gamma^{-1} - 1)^{1/k}) e^{-\tau p} - p = 0,
\end{equation}
which are arranged by decreasing their real parts. Thus, the local Lyapunov dimension at $\phi^{\pm}$ satisfies $\operatorname{dim}_{\operatorname{L}}\phi^{\pm} \in (14, 15)$. Moreover, $\operatorname{Re}\lambda_{j} > 0$ for $j = \{1,\ldots, 6\}$. Consequently, $\mathcal{A}$ contains the $6$-dimensional unstable manifolds of $\phi^{\pm}$. On the other hand, $\operatorname{dim}_{\operatorname{L}}\phi^{0} \in (3,4)$, and there is a one-dimensional unstable manifold at $\phi^{0}$. This indicates that the symmetric equilibria may have the prospect of generating geometric structures with a dimension proportional to $\tau$ as $\tau \to \infty$. 

We conjecture the following asymptotic behavior of the local Lyapunov dimension $\operatorname{dim}_{\operatorname{L}}\phi^{\pm}$ and the dimension $\operatorname{dim}\mathcal{W}^{u}(\phi^{\pm})$ of unstable manifolds at the symmetric equilibria $\phi^{\pm}$.
\begin{conjecture}
	\label{CONJ: MackeyGlassLocalDims}
	For any $\beta > \gamma > 0$, there are constants $C^{L}_{\gamma,\beta}>0$ and $C^{u}_{\gamma,\beta}>0$ such that
	\begin{equation}
		\operatorname{dim}_{\operatorname{L}}\phi^{\pm} \sim C^{L}_{\gamma, \beta} \cdot \tau \quad \text{and} \quad \operatorname{dim}\mathcal{W}^{u}(\phi^{\pm}) \sim C^{u}_{\gamma,\beta} \cdot \tau \qquad \text{as} \quad \tau \to \infty.
	\end{equation}
    In other words, the numbers $N_{L}=N_{L}(\gamma,\beta,\tau)$ and $N^{u}=N^{u}(\gamma,\beta,\tau)$ of roots $\{\lambda_{j} \}_{j \geq 1}$ of \eqref{EQ: MackeyGlassLinearizationSymmetricStates} (arranged by nonincreasing their reals parts) such that
    \begin{equation}
    	\begin{split}
    		\sum_{j=1}^{N_{L}}\operatorname{Re}\lambda_{j} \geq 0 \quad &\text{and} \quad \sum_{j=1}^{N_{L}+1}\operatorname{Re}\lambda_{j} < 0,\\
    		\operatorname{Re}\lambda_{N^{u}} > 0 \quad &\text{and} \quad \operatorname{Re}\lambda_{N^{u}+1} \leq 0
    	\end{split}
    \end{equation}
    satisfy $N_{L} \sim C^{L}_{\gamma,\beta} \cdot \tau$ and $N^{u} \sim C^{u}_{\gamma,\beta} \cdot \tau$ as $\tau \to \infty$.
\end{conjecture}
For $\gamma = 0.1$ and $\beta = 0.2$, numerical experiments conducted for $\tau \in [10,5000]$ suggest that $C^{L}_{\gamma,\beta} \approx 0.2936$ and $C^{u}_{\gamma,\beta} \approx 0.1232$. This indicates that the estimate \eqref{EQ: MackeyGlassLyapDimEstimate} is asymptotically sharp in the sense of any dimensions (Hausdorff, fractal or Lyapunov) of $\mathcal{A}$.

In \cite{Farmer1982ChaoticAtt}, Farmer numerically computed the fractal and Kaplan--Yorke dimensions of the observable part of the global attractor $\mathcal{A}$ for $\beta = 0.2$, $\gamma = 0.1$, $k=10$, and certain values of $\tau \in [17,30]$. His results provide values in the interval $[2,3.6]$. Combining this with the previous paragraphs indicates that the observable part is significantly smaller than the entire $\mathcal{A}$. This is not a surprising phenomenon.

%% file: SuarezSchopfOscill.tex
\subsection{Periodically forced Suarez--Schopf delayed oscillator}
\label{SEC: PertSuarezSchopfDimEstimate}
In this section, we consider the delayed oscillator proposed by Suarez and Schopf in \cite{Suarez1988} as a model for the El Ni\~{n}o--Southern Oscillation. Here we study the model under the influence of periodic forcing, as considered in our paper \cite{AnikushinRom2023SS}. This model is given by
\begin{equation}
	\label{EQ: SuarezSchopfPerturbedSin}
	\dot{x}(t) = x(t) - \alpha x(t-\tau) - x^{3}(t) + A\sin(t),
\end{equation}
where $\tau, A > 0$, and $\alpha \in (0,1)$ are real parameters. 

However, it will be convenient to study the following general class of scalar nonautonomous equations over a semiflow $(\mathcal{Q},\vartheta)$ on a compact metric space $\mathcal{Q}$, which is described over $q \in \mathcal{Q}$ by
\begin{equation}
	\label{EQ: SuarezSchopfPerturbedGeneral}
	\dot{x}(t) = \gamma x(t) - \alpha x(t-\tau) - \beta F(\vartheta^{t}(q),x(t)) + \widetilde{W}(\vartheta^{t}(q)),
\end{equation}
where $\alpha, \beta, \gamma$, and $\tau$ are positive real parameters; $\widetilde{W} \colon \mathcal{Q} \to \mathbb{R}$ is continuous; $F(q,y)$ is $C^{1}$-differentiable in $y$ and satisfies
\begin{equation}
	\label{EQ: GeneralSuarezSchopfNonlinearityDissipativity}
	\frac{F(q,y)}{y} \to +\infty \qquad \text{as} \quad |y| \to \infty \quad \text{ uniformly in} \ q \in \mathcal{Q},
\end{equation}
and the derivative $F'$ of $F$ in $y$ satisfies
\begin{equation}
	\label{EQ: GeneralSuarezSchopfNonlinearityDerivative}
	F'(q,y) \geq 0 \qquad \text{for all} \quad q \in \mathcal{Q} \quad \text{and} \quad y \in \mathbb{R}.
\end{equation} 

Note that \eqref{EQ: SuarezSchopfPerturbedSin} corresponds to the case of $\mathcal{Q} \coloneq \mathbb{R}/2\pi \mathbb{Z}$, and $\vartheta$ is the flow of translates on $\mathcal{Q}$, i.e., $\vartheta^{t}(q) \coloneq q + t$ for all $q \in \mathcal{Q}$ and $t \in \mathbb{R}$.

Similarly to \cite[Lemma 4.1]{Anikushin2022Semigroups} or Lemma \ref{LEM: MackeyGlassAttractorBound}, we have the following.
\begin{lemma}
	\label{LEM: SuarezSchopfPerturbedInvariantSets}
	Let $R>0$ and $\varepsilon>0$ be such that
	\begin{equation}
		\sup_{q \in \mathcal{Q}} \frac{F(q,y)}{y} \geq \frac{\alpha + \gamma  + \varepsilon}{\beta} \qquad \text{for all} \quad |y| \geq R \quad \text{and} \quad \varepsilon R > \sup_{q \in \mathcal{Q}}|\widetilde{W}(q)|.
	\end{equation}
	Then any classical solution to \eqref{EQ: SuarezSchopfPerturbedGeneral} starting in the ball $\mathcal{B}_{R}(0)$ remains in the ball in the future.
\end{lemma}

Thanks to \eqref{EQ: GeneralSuarezSchopfNonlinearityDissipativity}, the statement of Lemma \ref{LEM: SuarezSchopfPerturbedInvariantSets} is satisfied for any sufficiently large $R$. Consequently (see \cite{HaleLunel1993}), \eqref{EQ: SuarezSchopfPerturbedGeneral} generates a cocycle $\psi$ in $\mathbb{E} = C([-\tau,0];\mathbb{R})$ over $(\mathcal{Q},\vartheta)$ such that $\psi^{t}(q,\phi_{0}) = x_{t}$, where $x \colon [-\tau,+\infty) \to \mathbb{R}$ is the classical solution to \eqref{EQ: SuarezSchopfPerturbedGeneral} satisfying $x(\theta) = \phi_{0}(\theta)$ for $\theta \in [-\tau,0]$.

Let $\pi$ be the skew-product semiflow on $\mathcal{Q} \times \mathbb{E}$ associated with $\psi$, i.e., $\pi^{t}(q,\phi) \coloneq (\vartheta^{t}(q), \psi^{t}(q,\phi))$ for all $\phi \in \mathbb{E}$, $q \in \mathcal{Q}$, and $t \geq 0$. Similarly to Corollary \ref{COR: MackeyGlassGlobalAttractor}, we obtain the following complement to Lemma \ref{LEM: SuarezSchopfPerturbedInvariantSets}.
\begin{lemma}
	\label{LEM: SuarezSchopfPertSkewProductSemiflow}
	The skew-product semiflow $\pi$ admits a global $\mathcal{B}$-attractor $\mathcal{A} \subset \mathcal{Q} \times \mathbb{E}$. Its fiber $\mathcal{A}_{q} \coloneq \mathcal{A} \cap ( \{q\} \times \mathbb{E} )$ over $q \in \mathcal{Q}$ lies in the ball $\mathcal{B}_{R}(0)$ for any $R > 0$ satisfying the conditions of Lemma \ref{LEM: SuarezSchopfPerturbedInvariantSets}.
\end{lemma}

Given $\wp = (q,\phi) \in \mathcal{A}$, let $y_{0}(\cdot) \colon [-\tau,+\infty) \to \mathbb{R}$ be the corresponding classical solution to \eqref{EQ: SuarezSchopfPerturbedGeneral} with initial data $y_{0}(\theta) = \phi(\theta)$ for $\theta \in [-\tau,0]$. Then the boundary part of the linearized equations \eqref{EQ: LinearizedEquationDelayInRn} over $\wp$ is given by
\begin{equation}
	\label{EQ: SuarezSchopfPertLinearized}
	\dot{x}(t) = (\gamma - \beta F'(\vartheta^{t}(q),y_{0}(t)) ) x(t) - \alpha x(t-\tau).
\end{equation}
Thus, in terms of Theorem \ref{TH: TheoremLyapExponentsEstimateDelayEqs}, for $\wp = (q,\phi) \in \mathcal{P} =  \mathcal{A}$, we have (note that $J=0$)
\begin{equation}
	\label{EQ: SuarezSchopfOperatorsL0andLtau}
	L_{0}(\wp) = \gamma - \beta F'(q,\phi(0)) \quad \text{and} \quad L_{-\tau}(\wp) = -\alpha.
\end{equation}
Consequently, for a given $\varkappa = \varkappa_{0} = \varkappa_{J} \geq 0$, the $(1 \times 1)$-matrix from \eqref{EQ: MatrixL0LtauDelaySymLyapExpsEstimate} has the eigenvalue $\lambda_{1}(\wp) = 1 + 2 \gamma - 2\beta F'(q,\phi(0)) + \alpha^{2} e^{\varkappa \tau}$.

Let $\Xi$ be the derivative cocycle in $\mathbb{H} = \mathbb{R} \times L_{2}(-\tau,0;\mathbb{R})$ generated by \eqref{EQ: SuarezSchopfPertLinearized} over $(\mathcal{A},\pi)$, as in \eqref{EQ: DelayLinearizationCocycleDef}. We have the following theorem.
\begin{theorem}
	In the above context, for the Lyapunov dimension $\operatorname{dim}_{\operatorname{L}}\Xi$ of the cocycle $\Xi$ in $\mathbb{H}$ over $(\mathcal{A},\pi)$, we have the estimate
	\begin{equation}
		\label{EQ: SuarezSchopfPertLyapDimEstimate}
		\operatorname{dim}_{\operatorname{L}}\Xi \leq \alpha^{2} e^{p^{*} +1} \cdot \tau + 1,
	\end{equation}
	where $p^{*}$ is the unique root $p > 0$ of 
	\begin{equation}
		\label{EQ: SuarezSchopfPertLyapDimEstimateEqForP}
		p e^{p + 1} = \frac{1+2\gamma}{\alpha^{2}}.
	\end{equation}
\end{theorem}
\begin{proof}
	Consider the value $\alpha^{+}(m)$ from \eqref{EQ: DelaySymmetrizationCocycleApplicationAlphaPlusDef} for $\mathcal{P} = \mathcal{A}$, $V \equiv 0$, $\varkappa_{0} = \varkappa > 0$, and any integer $m \geq 1$. By \eqref{EQ: GeneralSuarezSchopfNonlinearityDerivative}, we have
	\begin{equation}
		\alpha^{+}(m) \leq \frac{1}{2}\left(1 + 2 \gamma  + \alpha^{2} e^{\varkappa \tau} - \varkappa (m-1) \right) \eqcolon \sigma(m).
	\end{equation}
    As in the proof of Theorem \ref{TH: MackeyGlassDimensionEstimate}, we consider
    $d^{*} \geq 1$ such that $\sigma(d^{*}) = 0$, i.e.,
    \begin{equation}
    	d^{*} = \varkappa^{-1}(1+2\gamma + \alpha^{2}e^{\varkappa \tau}) + 1.
    \end{equation}
    Minimizing $d^{*}$ over $\varkappa > 0 $ according to Lemma \ref{LEM: ScalarDelayDimEstLemma} with $a = 1+2\gamma$ and $b = \alpha^{2}$, we obtain the desired.
\end{proof}

Following Remark \ref{REM: DelayDimEstimatesScaling} and introducing the scaling $x(t) \mapsto x(\kappa t)$, $\phi(t,\theta) \mapsto \phi(\kappa t,\kappa \theta)$ into \eqref{EQ: SuarezSchopfPerturbedGeneral}, we may sharpen \eqref{EQ: SuarezSchopfPertLyapDimEstimate} analogously to Corollary \ref{COR: MackeyGlassSharperDimEst} as follows.
\begin{corollary}
	In terms of Theorem \ref{TH: MackeyGlassDimensionEstimate}, we have the sharper estimate
	\begin{equation}
		\label{EQ: SuarezSchopfPertLyapDimEstimateSharper}
		\operatorname{dim}_{\operatorname{L}}\Xi \leq \alpha^{2} \inf_{\kappa > 0}\left(\kappa e^{p^{*}(\kappa) +1}\right) \cdot \tau + 1,
	\end{equation}
	where $p^{*}(\kappa)$ is the unique root $p \geq -1$ of
	\begin{equation}
		pe^{p+1} = \frac{1+2\kappa \gamma}{(\kappa \alpha)^{2}}.
	\end{equation}
\end{corollary}

An extension of the result from \cite{ChepyzhovIlyin2004} to the context of cocycles would imply that any fiber $\mathcal{A}_{q}$, where $q \in \mathcal{Q}$, admits the same bound as in \eqref{EQ: SuarezSchopfPertLyapDimEstimateSharper} for its fractal dimension.

In \cite{AnikushinRom2023SS}, there was numerically observed a chaotic behavior in the model \eqref{EQ: SuarezSchopfPerturbedSin} for the parameters $\tau = 1.596$, $\alpha = 0.75$, and $A=0.073$. For such parameters, in terms of \eqref{EQ: SuarezSchopfPertLyapDimEstimate} we have $p^{*} \approx 0.843807$ and $\operatorname{dim}_{\operatorname{L}}\Xi \leq 6.675$. Moreover, the infimum in \eqref{EQ: SuarezSchopfPertLyapDimEstimateSharper} is achieved at $\kappa \approx 0.346771$ with $\kappa e^{p^{*}(\kappa)+1} \approx 5.1267$, and it gives the improved estimate
\begin{equation}
	\label{EQ: ExampleSuarezSchopfPertLyapDimEst}
	\operatorname{dim}_{\operatorname{L}}\Xi \leq 5.603.
\end{equation}
However, the analytical-numerical evidence presented in \cite{AnikushinRom2023SS} shows that each fiber $\mathcal{A}_{q}$ is contained in a two-dimensional $C^{1}$-smooth inertial manifold. Here, the numerical part only assures that $\mathcal{A}_{q}$ lies in the ball of radius $1$, and the analytical part is developed in our adjacent works \cite{Anikushin2020FreqDelay, Anikushin2020Geom} (see also \cite{Anikushin2022Semigroups,AnikushinAADyn2021}).

It is also interesting to determine whether estimates such as \eqref{EQ: SuarezSchopfPertLyapDimEstimate} are asymptotically sharp as $\tau \to \infty$ in the considered class. For this, note that for $A = 0$, the system \eqref{EQ: SuarezSchopfPerturbedSin} is autonomous, and there exist three equilibria, namely, $\phi^{0} \equiv 0$ and $\phi^{\pm}\equiv \pm \sqrt{1+\alpha}$. For $\alpha \in (0,1)$, we always have a one-dimensional unstable manifold at $\phi^{0}$ for any $\tau > 0$, i.e., $\operatorname{dim}\mathcal{W}^{u}(\phi^{0}) = 1$. For the symmetric equilibria $\phi^{\pm}$, one may propose a statement similar to Conjecture \ref{CONJ: MackeyGlassLocalDims}. For $\alpha = 0.75$, numerical experiments suggest that $\operatorname{dim}_{\operatorname{L}}\phi^{\pm} \sim C^{L}_{\alpha} \cdot \tau$ with $C^{L}_{\alpha} \approx 0.961$, and $\operatorname{dim}\mathcal{W}^{u}(\phi^{\pm}) \sim C^{u}_{\alpha} \cdot \tau$ with $C^{u}_{\alpha} \approx 0.225$. Furthermore, $\operatorname{dim}_{\operatorname{L}}\phi^{0} \sim C^{L}_{\alpha} \cdot \tau$ with $C^{L}_{\alpha} \approx 1.645$. This indicates that \eqref{EQ: SuarezSchopfPertLyapDimEstimate} as an upper estimate for the Hausdorff, fractal, or Lyapunov dimensions is asymptotically sharp as $\tau \to \infty$ in the considered class.

%% file: ErgVarPrinc.tex
\section{Ergodic Variational Principle for subadditive families}
\label{APP: EVPNonCompact}
In this section, we deal with a dynamical system $(\mathcal{Q},\vartheta)$ on a complete metric space $\mathcal{Q}$ for which the mappings $\vartheta^{t}$ are defined for $t \in \mathbb{T}$ from a time set $\mathbb{T} \in \{ \mathbb{Z}_{+}, \mathbb{Z}, \mathbb{R}, \mathbb{R}_{+} \}$, see Section \ref{SEC: VariationalDescriptionOfLargUnifLyapExp}. Recall $\mathbb{T}_{+} \coloneq \mathbb{T} \cap [0,+\infty)$. Below, when we write $t \to +\infty$ or $t > 0$, it is always understood that $t \in \mathbb{T}_{+}$.

Our aim is to extend \cite[Theorem A.3]{Morris2013}, which we call the \textit{Ergodic Variational Principle} for subadditive families (for brevity, EVP). We relax the requirement of $\mathcal{Q}$ to be compact by the existence of a $2^{\mathcal{Q}}$-attractor $\mathcal{A}$ for $(\mathcal{Q},\vartheta)$, see \eqref{EQ: MinimalAttractorInQ}. This is the usual situation in applications to dissipative dynamical systems, where $\mathcal{A}$ is a global attractor, and $\mathcal{Q}$ is a bounded positively invariant region localizing $\mathcal{A}$. We also provide a proof for continuous-time systems, deducing it from the discrete-time case considered in \cite{Morris2013}.

Such a generalization has useful consequences for theory and applications. Based on this, we motivate and establish the following three principles for uniformly continuous linear cocycles $\Xi$ over $(\mathcal{Q},\vartheta)$.
\begin{description}
	\item[(\textit{Concentration Principle})] Uniform Lyapunov exponents and the Lyapunov dimension are concentrated on the attractor and characterized through ergodic measures, see Corollary \ref{COR: DVolumesUniformAsDVolumesOverMeasure};
	\item[(\textit{Computation Principle})] Uniform Lyapunov exponents can be approximated by computing infinitesimal growth exponents in adapted (possibly noncoercive) metrics, see Theorem \ref{TH: VariationalDescriptionUniformLyapunovExponents};
	\item[(\textit{Robustness Principle})] Consecutive sums of uniform Lyapunov exponents and the Lyapunov dimension depend upper semicontinuously on $\Xi$ under natural conditions, see Theorem \ref{TH: UpperSemiContLargLyapExp} and Corollary \ref{COR: UpperSemiContUniformLyapExpAndLyapDim}.
\end{description}
In fact, robustness is more elementary and does not require any kind of the EVP.

Let us first recall the Kingman subadditive ergodic theorem with a clarification for continuous-time systems. It is formulated for a measurable dynamical system\footnote{That is, \nameref{DESC: DS1} holds, and instead of \nameref{DESC: DS2}, we require the mapping $\mathbb{T} \times \mathcal{Q} \ni (t,q) \mapsto \vartheta^{t}(q) \in \mathcal{Q}$ to be measurable.} $\vartheta$ on a measurable space $\mathcal{Q}$. For a measure $\mu$ on $\mathcal{Q}$, a family $f = \{ f^{t} \}_{t \in \mathbb{T}_{+}}$ of measurable functions $f^{t}$ on $\mathcal{Q}$ with values in $[-\infty,+\infty)$ is called \textit{$\mu$-subadditive} over $(\mathcal{Q},\vartheta)$ if for $\mu$-almost all $q \in \mathcal{Q}$ it satisfies
\begin{equation}
	\label{EQ: SubbAddFamilyDefinition}
	f^{t+s}(q) \leq f^{t}(\vartheta^{s}(q)) + f^{s}(q) \qquad \text{for all} \quad t,s \in \mathbb{T}_{+}.
\end{equation}
Furthermore, if there exists a $\mu$-integrable function $M$ on $\mathcal{Q}$ such that (here $\theta,s \in \mathbb{T}_{+}$)
\begin{equation}
	\label{EQ: MuSummableBoundedness}
	\sup_{0 \leq \theta \leq s \leq 1} f^{s-\theta}(\vartheta^{\theta}(q)) \leq M(q) \qquad \text{for} \quad \text{$\mu$-almost all } q \in \mathcal{Q},
\end{equation}
we say that $f$ is \textit{$\mu$-summably bounded from above} (on finite times). For the case of discrete time, the just introduced condition simply means that $\max\{ f^{1}(q), 0 \}$ is integrable.

\begin{theorem}[Kingman's subadditive ergodic theorem]
	\label{TH: KingmanSETDiscreteTime}
	Suppose a measure $\mu$ is invariant with respect to a measurable dynamical system $(\mathcal{Q},\vartheta)$, and a $\mu$-subadditive over $(\mathcal{Q},\vartheta)$ family $f = \{ f^{t} \}_{t \in \mathbb{T}_{+}}$ is $\mu$-summably bounded from above. Then there exists the limit
	\begin{equation}
		\label{EQ: KingmanPointwiseLimit}
		\lim_{t \to +\infty} \frac{1}{t}f^{t}(q) \coloneq \bar{f}(q) \in [-\infty, +\infty)
	\end{equation}
	for $\mu$-almost all $q \in \mathcal{Q}$, and it defines an invariant function $\bar{f}$, i.e., for any $t \in \mathbb{T}_{+}$, we have $\bar{f}(\vartheta^{t}(q))=\bar{f}(q)$ for $\mu$-almost all $q \in \mathcal{Q}$. Furthermore,
	\begin{equation}
		\label{EQ: KingmanLimitTrans}
		\int_{\mathcal{Q}}\bar{f}d\mu = \inf_{t > 0} \frac{1}{t} \int_{\mathcal{Q}}f^{t}d\mu = \lim_{t \to +\infty}\frac{1}{t} \int_{\mathcal{Q}}f^{t}d\mu \in [-\infty, +\infty).
	\end{equation}
\end{theorem}
\begin{proof}
	This is due to Kingman \cite{Kingman1973}, except that we refined his condition for continuous time, where he requires the similar supremum of the modulus to be summable, see \cite[(1.4.7)]{Kingman1973}. Let us emphasize this and show that under \eqref{EQ: MuSummableBoundedness} we have (here $k \in \mathbb{Z}_{+}$)
	\begin{equation}
		\label{EQ: KingmanRefine0}
		\lim_{k \to \infty} \frac{1}{k}f^{k}(q) = \lim_{t \to +\infty}\frac{1}{t}f^{t}(q)
	\end{equation}
	for $\mu$-almost all $q \in \mathcal{Q}$. With \eqref{EQ: KingmanRefine0} established, the theorem for continuous time immediately follows from that of the discrete case. For the following, we set $T := \vartheta^{1}$.
	
	Supposing the left-hand limit exists, let us show the existence of the right-hand limit and their coincidence for $\mu$-almost all $q \in \mathcal{Q}$. Indeed, consider $t \geq 1$ and an integer $k=k(t)$ such that $k \leq t < k+1$. Let $s = k+1 - t \in [0,1]$. Then $f^{k+1}(q) = f^{t+s}(q) \leq f^{t}(q) + f^{s}(\vartheta^{t}(q))$. For $M$ from \eqref{EQ: MuSummableBoundedness}, we have $f^{s}(\vartheta^{t}(q)) \leq M(\vartheta^{k}(q))$, and, consequently,
	\begin{equation}
		\label{EQ: KingmanRefine1}
		\frac{1}{t}f^{t}(q) \geq \frac{k+1}{t}\cdot\frac{1}{k+1}f^{k+1}(q) - \frac{k}{t} \cdot \frac{1}{k} M(T^{k}(q)).
	\end{equation}
	
	On the other hand, let $s = t - k$ in the above context. Then $f^{t}(q) = f^{k + s}(q) \leq f^{k}(q) + f^{s}(\vartheta^{k}(q))$, and, consequently,
	\begin{equation}
		\label{EQ: KingmanRefine2}
		\frac{1}{t} f^{t}(q)  \leq \frac{k}{t} \cdot \frac{1}{k}f^{k}(q) + \frac{k}{t} \cdot \frac{1}{k} M(T^{k}(q)).
	\end{equation}
	Since $T$ preserves $\mu$ and $M$ is $\mu$-integrable, it is a standard fact that $M(T^{k}(q))/k \to 0$ as $k \to \infty$ for $\mu$-almost all $q$. By taking the limit inferior and limit superior as $t \to \infty$ in \eqref{EQ: KingmanRefine1} and \eqref{EQ: KingmanRefine2}, respectively, we obtain \eqref{EQ: KingmanRefine0}.
\end{proof}

Under the conditions of Theorem \ref{TH: KingmanSETDiscreteTime}, if $\mu$ is ergodic, then the limit $\bar{f}$ is constant $\mu$-almost everywhere. We set $\lambda(f;\mu)$ to be this constant value and call it the \textit{growth exponent} of $f$ with respect to $\mu$.

Now we proceed to continuous dynamical systems $(\mathcal{Q},\vartheta)$. A family $f = \{ f^{t} \}_{t \in \mathbb{T}_{+}}$ of functions $f^{t} \colon \mathcal{Q} \to [-\infty,+\infty)$ is called \textit{subadditive} over $(\mathcal{Q},\vartheta)$ if \eqref{EQ: SubbAddFamilyDefinition} is satisfied for any $q \in \mathcal{Q}$.  Moreover, we say that $f$ is \textit{bounded from above} (on finite times) if 
\begin{equation}
	\label{EQ: SubbAddBoundedness}
	\sup_{t \in [0,1] \cap \mathbb{T}_{+}}\sup_{q \in \mathcal{Q}} f^{t}(q) < +\infty.
\end{equation}
Clearly, the analogous supremum, as in \eqref{EQ: SubbAddBoundedness}, but taken over $t \in [0,T] \cap \mathbb{T}_{+}$ for any $T>0$, is finite, thanks to subadditivity.

We call a subadditive family $f$ \textit{proper} if it is bounded from above and any $f^{t}$ is an upper semicontinuous function. Clearly, such $f$ can be considered in the context of Theorem \ref{TH: KingmanSETDiscreteTime}, so the growth exponent $\lambda(f;\mu)$ of $f$ is well defined with respect to any ergodic Borel probability measure $\mu$.

For what follows, let us fix a proper subadditive family $f$ over $(\mathcal{Q},\vartheta)$.

We define the \textit{uniform growth exponent} $\lambda(f)$ of $f$ by the limit
\begin{equation}
	\label{EQ: UniformGrowthExpSubbF}
	\lambda(f) \coloneq \lim_{t \to +\infty} \frac{1}{t} \sup_{q \in \mathcal{Q}} f^{t}(q) = \inf_{t > 0 } \frac{1}{t} \sup_{q \in \mathcal{Q}} f^{t}(q),
\end{equation}
which is well defined by Lemma \ref{LEM: FeketeLemma}.

We say that a subset $\mathcal{A} \subset \mathcal{Q}$ is a \textit{$2^{\mathcal{Q}}$-attractor} for $(\mathcal{Q},\vartheta)$ if $\mathcal{A}$ is compact, $\mathcal{A}$ attracts $\mathcal{Q}$ in the sense
\begin{equation}
	\label{EQ: SubsetA0AttractsQ}
	\sup_{ q \in \mathcal{Q}}\operatorname{dist}(\vartheta^{t}(q),\mathcal{A}) \to 0 \qquad \text{as} \quad t \to +\infty,
\end{equation}
where $\operatorname{dist}$ denotes the distance in $\mathcal{Q}$ from a point to a subset, and $\mathcal{A}$ is the smallest (by inclusion) set satisfying these two properties.

If there exists a compact set that attracts $\mathcal{Q}$, then there exists a $2^{\mathcal{Q}}$-attractor $\mathcal{A}$, and it is given by
\begin{equation}
	\label{EQ: MinimalAttractorInQ}
	\mathcal{A} \coloneq \bigcap_{t \in \mathbb{T}_{+}} \vartheta^{t}(\mathcal{Q}).
\end{equation}
Note that $\mathcal{A}$ is necessarily invariant, i.e., $\vartheta^{t}(\mathcal{A}) = \mathcal{A}$ for any $t \in \mathbb{T}_{+}$.

We reduce the case of noncompact $\mathcal{Q}$ to a compact one using the following simple lemma.
\begin{lemma}[Weak Concentration Principle]
	\label{LEM: WeakConcetrationSubadditive}
	For $f = \{ f^{t} \}_{t \in \mathbb{T}_{+}}$ as above, suppose there exists a $2^{\mathcal{Q}}$-attractor $\mathcal{A}$ for $(\mathcal{Q},\vartheta)$. Then
	\begin{equation}
		\label{EQ: WeakConcetrationSubadditive}
		\lambda(f) = \lim_{t \to +\infty} \frac{1}{t} \sup_{q \in \mathcal{A}} f^{t}(q) = \inf_{t > 0 } \frac{1}{t} \sup_{q \in \mathcal{A}} f^{t}(q).
	\end{equation}
\end{lemma}
\begin{proof}
	Let $\lambda$ denote the right-hand side of \eqref{EQ: WeakConcetrationSubadditive}. It is clear that $\lambda(f) \geq \lambda$, so it is required to show the inverse inequality. Fix any $\nu < -\lambda$, and set $f^{t}_{\nu}(q) \coloneq f^{t}(q) + \nu t$ for all $q \in \mathcal{Q}$ and $t \in \mathbb{T}_{+}$. Clearly, $f_{\nu} = \{ f^{t}_{\nu} \}_{t \in \mathbb{T}_{+}}$ is a proper subadditive family. From the definition of $\lambda$, there exist $t_{0} \in \mathbb{T}_{+}$ and $\delta > 0 $ such that
	\begin{equation}
		\label{EQ: SubadditiveWeakConcLemma1}
		f^{t}_{\nu}(q) < -\delta t \qquad \text{for any} \quad t \in \mathbb{T}_{+} \cap [t_{0},+\infty) \quad \text{and} \quad q \in \mathcal{A}.
	\end{equation}
	Since $f^{t_{0}}(\cdot)$ is upper semicontinuous and $\mathcal{A}$ is compact, there exists a neighborhood $\mathcal{U}$ of $\mathcal{A}$ in $\mathcal{Q}$ such that 
	\begin{equation}
		\label{EQ: SubadditiveWeakConcLemma2}
		f^{t_{0}}_{\nu}(q) \leq -\delta t_{0} \qquad \text{for any} \quad q \in \mathcal{U}.
	\end{equation}
	
	Furthermore, since $\mathcal{A}$ is a $2^{\mathcal{Q}}$-attractor, there exists $T \in \mathbb{T}_{+}$ such that
	\begin{equation}
		\label{EQ: SubadditiveWeakConcLemma3}
		\vartheta^{t}(\mathcal{Q}) \subset \mathcal{U} \qquad \text{for any} \quad t \in \mathbb{T}_{+} \cap [T,+\infty).
	\end{equation}
	
	By \eqref{EQ: SubadditiveWeakConcLemma2} and \eqref{EQ: SubadditiveWeakConcLemma3}, for any positive integer $k$, we obtain
	\begin{equation}
		\label{EQ: SubadditiveWeakConcLemma4}
		f^{kt_{0}}_{\nu}(\vartheta^{T}(q)) \leq \sum_{j=0}^{k-1}f^{t_{0}}_{\nu}(\vartheta^{T + j t_{0}}(q)) \leq -\delta k t_{0} \qquad \text{for any} \quad q \in \mathcal{Q}.
	\end{equation}
	
	From the subadditivity and boundedness from above on finite times, we have for any $s \in \mathbb{T}_{+}$ that
	\begin{equation}
		\lambda(f) = \lim_{t \to +\infty} \frac{1}{t} \sup_{q \in \mathcal{Q}}f^{t}(\vartheta^{s}(q)).
	\end{equation}
	By combining this with \eqref{EQ: SubadditiveWeakConcLemma4}, we obtain $\lambda(f_{\nu}) = \lambda(f) + \nu < 0$ for any $\nu < -\lambda$. Taking it to the limit as $\nu \to -\lambda$ yields $\lambda(f) \leq \lambda$.
\end{proof}

For $q \in \mathcal{Q}$, let $\lambda(f;q)$ denote the pointwise limit as in \eqref{EQ: KingmanPointwiseLimit}, provided that it exists. We need the following elementary lemma.
\begin{proposition}
	\label{PROP: ElementaryPropertiesOfSAF}
	For $f=\{f^{t}\}_{t \in \mathbb{T}_{+}}$ as above, suppose $\mathbb{T}_{+} = [0,+\infty)$, define $f_{1} \coloneq \{ f^{k} \}_{k \in \mathbb{Z}_{+}}$, and let $q \in \mathcal{Q}$ be given. Then
	\begin{enumerate}
		\item[1)] $\lambda(f_{1};q) = \lambda(f;q)$, i.e., the limits either exist and coincide or do not exist;
		\item[2)] If $\lambda(f;q)  = \lambda(f)$, then $\lambda(f;\vartheta^{s}(q)) = \lambda(f)$ uniformly in $s$ from bounded intervals.
	\end{enumerate}
\end{proposition}
\begin{proof}
	For item 1), the proof follows from \eqref{EQ: KingmanRefine1} and \eqref{EQ: KingmanRefine2} without appealing to almost every point because $M$ is constant, and the estimates hold everywhere. For item 2), we use the inequality $f^{t+s}(q) \leq f^{t}(\vartheta^{s}(q)) + f^{s}(q)$ and apply similar arguments, noting that the limit superior of $f^{t}(\vartheta^{s}(q))/t$ as $t \to +\infty$ cannot exceed $\lambda(f)$.
\end{proof}

Assuming that $\mathcal{A}$ is a $2^{\mathcal{Q}}$-attractor for $(\mathcal{Q},\vartheta)$, let $\mathfrak{M}^{erg}(\vartheta) = \mathfrak{M}^{erg}(\mathcal{A};\vartheta)$ denote the space of all ergodic invariant Borel probability measures on $\mathcal{A}$.

\begin{theorem}[Ergodic Variational Principle for subadditive families]
	\label{TH: ErgodicVarPrinciple}
	Suppose $f = \{ f^{t} \}_{t \in \mathbb{T}_{+}}$ is a proper subadditive family over $(\mathcal{Q},\vartheta)$ as above. Then
	\begin{equation}
	\label{EQ: ErgVarPrincMeasureIden}
		\begin{split}
			\lambda(f) = \max_{\mu \in \mathfrak{M}^{erg}(\vartheta)}\lambda(f;\mu) = \lim_{t \to +\infty} \max_{ \mu \in \mathfrak{M}^{erg}(\vartheta)} \frac{1}{t}\int_{\mathcal{A}} f^{t}(q)d\mu(q).
		\end{split}
	\end{equation}
    Furthermore,
    \begin{equation}
    \label{EQ: ErgVarPrincPointIden}
   		\lambda(f) = \inf_{t > 0 } \frac{1}{t} \sup_{q \in \mathcal{Q}} f^{t}(q) = \sup_{q \in \mathcal{Q}} \inf_{t > 0} \frac{1}{t}f^{t}(q) = \max_{q \in \mathcal{A}} \inf_{t > 0} \frac{1}{t}f^{t}(q).
    \end{equation}
\end{theorem}
\begin{proof}
	By virtue of Lemma \ref{LEM: WeakConcetrationSubadditive}, we may assume that $\mathcal{Q} = \mathcal{A}$. For systems with discrete time, the statement is exactly \cite[Theorem A.3]{Morris2013}. We explain how the statement for continuous time can be deduced from that of the discrete case.
	
	As to \eqref{EQ: ErgVarPrincPointIden}, its only nontrivial part is \cite[Proposition A.7]{Morris2013}, whose proof works in continuous time with only cosmetic changes at the beginning\footnote{In terms of \cite{Morris2013}, we change $n(x) \geq 1$ to $n(x) > 0$ and $r_{1} \colon X \to \mathbb{N}$ to $r_{1} \colon X \to (0,+\infty)$ and take $m,M > 0$ such that $m \leq r_{1} \leq M-1$. Then the proof works in both discrete and continuous time.}. Therefore, it only requires to show \eqref{EQ: ErgVarPrincMeasureIden}.
	
	For this, we apply the discrete-time case to the mapping $T \coloneq \vartheta^{1}$ and the family $f_{1} \coloneq \{ f^{k} \}_{k \in \mathbb{Z}_{+}}$. Let $\mu \in \mathfrak{M}^{erg}(T)$ be any maximizing measure for $f_{1}$, i.e., such that $\lambda(f_{1}) = \lambda(f_{1};\mu)$. Define a measure $\bar{\mu}$ by (here $\chi_{\mathcal{C}}$ is the characteristic function of $\mathcal{C}$)
	\begin{equation}
		\bar{\mu}(\mathcal{C}) = \int_{\mathcal{Q}}\int_{0}^{1} \chi_{\mathcal{C}}(\vartheta^{t}(q))dt d\mu(q)
	\end{equation}
	for any Borel measurable subset $\mathcal{C}$. This is the suspension (or the mapping torus) measure, which is known to be invariant and ergodic with respect to $\vartheta$. We claim that $\bar{\mu}$ is maximizing for $\vartheta$, i.e., $\lambda(f) = \lambda(f;\bar{\mu})$. 
	
	By the definition of $\mu$ and since $\lambda(f_{1}) = \lambda(f)$, we obtain (here $k \in \mathbb{Z}_{+}$)
	\begin{equation}
		\label{EQ: EVPLiftedMeasureMaximal1}
		\lim_{t \to +\infty} \frac{1}{k} f^{k}(q) = \inf_{k > 0} \frac{1}{k} \int_{\mathcal{X}}f^{k} d\mu = \lambda(f_{1}) = \lambda(f)
	\end{equation}
	for $\mu$-almost all $q \in \mathcal{Q}$.
	
	Consider $g^{t}(q) \coloneq \int_{0}^{1}f^{t}(\vartheta^{s}(q))ds$. Clearly, $g$ is a proper subadditive over $(\mathcal{Q},\vartheta)$ family. By combining \eqref{EQ: EVPLiftedMeasureMaximal1} with Proposition \ref{PROP: ElementaryPropertiesOfSAF} applied to $f$, we obtain
	\begin{equation}
		 \lambda(f) = \lim_{k \to \infty}\frac{1}{k}\int_{0}^{1}f^{k}(\vartheta^{s}(q))ds = \lim_{k \to \infty}\frac{1}{k}g^{k}(q)
	\end{equation}
	for $\mu$-almost all $q \in \mathcal{Q}$. This being combined with \eqref{EQ: KingmanLimitTrans} applied to $g_{1} = \{g^{k}\}_{k \in \mathbb{Z}_{+}}$, $T$, and $\mu$, and then to $f$, $\vartheta$, and $\bar{\mu}$ yields
	\begin{equation}
		\lambda(f) = \inf_{k > 0} \frac{1}{k}\int_{\mathcal{X}}g^{k}d\mu = \inf_{k > 0} \frac{1}{k}\int_{\mathcal{X}}f^{k}d\bar{\mu} = \inf_{t > 0} \frac{1}{t}\int_{\mathcal{X}}f^{t}d\bar{\mu} = \lambda(f;\bar{\mu}),
	\end{equation}
	 as claimed. Since the second equality in \eqref{EQ: ErgVarPrincMeasureIden} is obvious now, the proof is finished.
\end{proof}

To the best of our knowledge, the first variant of the EVP was formulated by Thieullen in \cite[Lemma 2.3.5]{Thieullen1992}, who referred to some ideas of Ledrappier \cite{Ledrappier1981}. In \cite{Thieullen1992}, it was only the first equality from \eqref{EQ: ErgVarPrincMeasureIden}, and it also lacks the possibility for $f^{t}$ to take values in the extended reals $[-\infty,+\infty)$. The latter is important for applications to noninvertible cocycles. We refer to the paper by Morris \cite{Morris2013} for a more detailed historical discussion of the EVP.

Let $\Xi$ be a uniformly continuous linear cocycle in a separable Hilbert space $\mathbb{H}$ over the dynamical system $(\mathcal{Q},\vartheta)$. Given $d \geq 0$, we define $f^{t}(q) \coloneq \ln \omega_{d}(\Xi^{t}(q,\cdot))$ for all $t \in \mathbb{T}_{+}$ and $q \in \mathcal{Q}$. By Lemma \ref{COR: HornInequality} and the uniform continuity of $\Xi$, we obtain that $f = \{f^{t}\}_{t \in \mathbb{T}_{+}}$ is a proper subadditive family. Then, Theorem \ref{TH: ErgodicVarPrinciple} applied to such $f$ yields the following result.
\begin{corollary}(Concentration Principle)
	\label{COR: DVolumesUniformAsDVolumesOverMeasure}
	In the above context, assume that there exists a $2^{\mathcal{Q}}$-attractor $\mathcal{A}$ for $(\mathcal{Q},\vartheta)$. Then for any $d \geq 0$, there exists an ergodic measure $\mu \in \mathfrak{M}^{erg}(\vartheta)$ such that for $\mu$-almost all $q \in \mathcal{A}$ we have
	\begin{equation}
		\label{EQ: ErgodicDVolumes}
		\omega_{d}(\Xi) = \lim_{t \to +\infty} \left[\omega_{d}(\Xi^{t}(q,\cdot))\right]^{1/t} \eqcolon \omega_{d}(\Xi;\mu).
	\end{equation}
\end{corollary}

One may also introduce the local Lyapunov dimension $\dim_{\operatorname{L}}(\Xi;q)$ of $\Xi$ at $q \in \mathcal{Q}$ and the Lyapunov dimension $\dim_{\operatorname{L}}(\Xi;\mu)$ of $\Xi$ over an ergodic invariant measure $\mu$ and deduce from Theorem \ref{TH: ErgodicVarPrinciple} that
\begin{equation}
	\label{EQ: LyapDimTopErgCharacterization}
	\dim_{\operatorname{L}}\Xi = \max_{q \in \mathcal{A}}\dim_{\operatorname{L}}(\Xi;q) = \max_{\mu \in \mathfrak{M}^{erg}(\mathcal{A};\vartheta)}\dim_{\operatorname{L}}(\Xi;\mu),
\end{equation}
where the first maximum is achieved on $\mu$-almost all points $q \in \mathcal{A}$ for any $\mu$ on which the second maximum is achieved\footnote{This is any $\mu$ from Corollary \ref{COR: DVolumesUniformAsDVolumesOverMeasure} with $d = \dim_{\operatorname{L}}\Xi$.}. A particular result of this kind first appeared in the work of Ledrappier \cite{Ledrappier1981} for derivative cocycles of differentiable mappings on manifolds. In the topological context concerned with the first identity in \eqref{EQ: LyapDimTopErgCharacterization}, such questions were studied by Eden \cite{EdenLocalEstimates1990} for cocycles in infinite dimensions. We refer to the end of Section \ref{SEC: VariationalDescUnifLyapExpCompCoc} for further related discussion.

%% file: UpperSemiCont.tex
\section{Upper semicontinuity of the uniform growth exponent for subadditive families}
\label{SEC: UpperSemiContLyapDim}

Let $\mathbb{T}_{+} \in \{ \mathbb{Z}_{+}, \mathbb{R}_{+} \}$. A number family $a=\{a^{t}\}_{t \in \mathbb{T}_{+}}$, where $a^{t} \in [-\infty,+\infty)$, is called \textit{subadditive} if $a^{t+s} \leq a^{t} + a^{s}$ for any $t,s \in \mathbb{T}_{+}$. It is said to be \textit{bounded from above} (on finite times) if $\sup_{t \in [0,1] \cap \mathbb{T}_{+}}a^{t} < +\infty$.
 
The following lemma is well-known.
\begin{lemma}[Fekete's lemma]
	\label{LEM: FeketeLemma}
	Let $a=\{a^{t}\}_{t \in \mathbb{T}_{+}}$ be a subadditive number family which is bounded from above. Then
	\begin{equation}
		\label{EQ: FeketeLemmaInfDescr}
		\lambda(a) \coloneq \lim_{t \to +\infty} \frac{1}{t} a^{t} = \inf_{t > 0} \frac{1}{t}a^{t}.
	\end{equation}
	Furthermore, if $\mathbb{T}_{+} = [0,\infty)$ and $a_{1} \coloneq \{ a^{k} \}_{k \in \mathbb{Z}_{+}}$, then $\lambda(a_{1}) = \lambda(a)$.
\end{lemma}

It is important to note that the description through the infimum in \eqref{EQ: FeketeLemmaInfDescr} immediately gives the upper semicontinuity\footnote{Similarly, the limit description \eqref{EQ: KingmanLimitTrans} gives the upper semicontinuity of $\lambda(f;\mu)$, which in turn can be applied to Lyapunov exponents and Lyapunov dimension over an ergodic measure $\mu$.} of $\lambda(a)$. For what follows, $\Gamma$ denotes a metric space.
\begin{corollary}
	\label{COR: SubAddNumberFamilyUpperSemiCont}
	Suppose there are bounded from above subadditive number families $a_{\gamma}=\{ a^{t}_{\gamma} \}_{t \in \mathbb{T}_{+}}$ given for any $\gamma \in \Gamma$. Assume for some $\gamma_{0} \in \Gamma$ that
	\begin{equation}
		\label{EQ: SubAddNumberFamilyConvergence}
		\limsup_{\gamma \to \gamma_{0}} a^{t}_{\gamma} \leq a^{t}_{\gamma_{0}} \qquad \text{for any} \quad t > 0.
	\end{equation}
	Then
	\begin{equation}
		\label{EQ: SubAddNumberFamilyUpperSemiCont}
		\limsup_{\gamma \to \gamma_{0}} \lambda(a_{\gamma}) \leq \lambda(a_{\gamma_{0}}).
	\end{equation}
\end{corollary}
\begin{proof}
	Indeed, by \eqref{EQ: FeketeLemmaInfDescr} and \eqref{EQ: SubAddNumberFamilyConvergence}, for any $t_{0} > 0$ we have
	\begin{equation}
		\limsup_{\gamma \to \gamma_{0}} \lambda(a_{\gamma}) \leq \limsup_{\gamma \to \gamma_{0}} \frac{1}{t_{0}}a^{t_{0}}_{\gamma} \leq \frac{1}{t_{0}}a^{t_{0}}_{\gamma_{0}}.
	\end{equation}
	Taking the infimum over $t_{0} > 0$ yields \eqref{EQ: SubAddNumberFamilyUpperSemiCont}.
\end{proof}

Now suppose that for any $\gamma \in \Gamma$, there is a dynamical system $(\mathcal{Q},\vartheta_{\gamma})$ with the same phase space $\mathcal{Q}$ and the same time set $\mathbb{T}$ for each $\gamma$; a $2^{\mathcal{Q}}$-attractor $\mathcal{A}_{\gamma}$ for $(\mathcal{Q},\vartheta_{\gamma})$, see \eqref{EQ: SubsetA0AttractsQ}; and a proper subadditive family $f_{\gamma}=\{ f^{t}_{\gamma} \}_{t \in \mathbb{T}_{+}}$ over $(\mathcal{Q},\vartheta_{\gamma})$, see below \eqref{EQ: SubbAddBoundedness}. 

\begin{corollary}
	\label{COR: UpperSemiContSubadditive}
	In the above context, suppose that for some $\gamma_{0} \in \Gamma$, we have
	\begin{equation}
		\label{EQ: SubAddNumberFamilyFromFunctionsConvergence}
		\limsup_{\gamma \to \gamma_{0}}\sup_{q \in \mathcal{A}_{\gamma}}f^{t}_{\gamma}(q) \leq \sup_{q \in \mathcal{Q}}f^{t}_{\gamma_{0}}(q) \qquad \text{for any} \quad t > 0.
	\end{equation}
 	Then
	\begin{equation}
		\label{EQ: SubadditiveGrowthExpUpperSemiCont}
		\limsup_{\gamma \to \gamma_{0}} \lambda(f_{\gamma}) \leq \lambda(f_{\gamma_{0}}).
	\end{equation}
\end{corollary}
\begin{proof}
	By Lemma \ref{LEM: WeakConcetrationSubadditive} and \eqref{EQ: SubAddNumberFamilyFromFunctionsConvergence}, for any $t_{0} > 0$ we have
	\begin{equation}
		\limsup_{\gamma \to \gamma_{0}} \lambda(f_{\gamma}) = \limsup_{\gamma \to \gamma_{0}} \inf_{t > 0} \frac{1}{t} \sup_{q \in \mathcal{A}_{\gamma}}f^{t}_{\gamma}(q) \leq \limsup_{\gamma \to \gamma_{0}} \frac{1}{t_{0}} \sup_{q \in \mathcal{A}_{\gamma}}f^{t_{0}}_{\gamma}(q) \leq \frac{1}{t_{0}} \sup_{q \in \mathcal{Q}}f^{t_{0}}_{\gamma_{0}}(q).
	\end{equation}
	By taking the infimum over $t_{0} > 0$ and using Lemma \ref{LEM: WeakConcetrationSubadditive}, we obtain \eqref{EQ: SubadditiveGrowthExpUpperSemiCont}.
\end{proof} 

Given a Banach space $\mathbb{E}$, suppose that for and each $\gamma \in \Gamma$, there is a uniformly continuous linear cocycle $\Xi_{(\gamma)}$ in $\mathbb{E}$ over $(\mathcal{Q},\vartheta_{\gamma})$ as above. We say that $\Xi_{(\gamma)}$ \textit{eventually converges} to $\Xi_{(\gamma_{0})}$ on a subset $\mathcal{B}_{0} \subset \mathcal{Q}$ as $\gamma \to \gamma_{0} \in \Gamma$, if for some positive $\tau \in \mathbb{T}_{+}$ we have\footnote{For Theorem \ref{TH: UpperSemiContLargLyapExp}, it is sufficient to require the convergence of norms rather than convergence in norm as in \eqref{EQ: EventualConvergenceCocycles}. However, this relaxation appears to be impractical.}
\begin{equation}
	\label{EQ: EventualConvergenceCocycles}
	\lim_{\gamma \to \gamma_{0}}\sup_{q \in \mathcal{B}_{0}}\| \Xi^{k\tau}_{(\gamma)}(q,\cdot) - \Xi^{k\tau}_{(\gamma_{0})}(q,\cdot)\|_{\mathcal{L}(\mathbb{E})} \to 0 \qquad \text{for any} \quad k=1,2,\ldots.
\end{equation}
\begin{remark}
	For the eventual convergence, one may formulate sufficient conditions posed only in terms of the convergence of $\Xi^{\tau}_{\gamma}$ to $\Xi^{\tau}_{\gamma_{0}}$ and $\vartheta^{\tau}_{\gamma}$ to $\vartheta^{\tau}_{\gamma_{0}}$ for \textit{some} $\tau > 0$ uniformly on $\mathcal{B}_{0}$. For example, if $\mathcal{B}_{0}$ is positively invariant with respect to $\vartheta^{\tau}_{\gamma}$ for all $\gamma$, and $\Xi^{\tau}_{\gamma_{0}}(q,\cdot) \in \mathcal{L}(\mathbb{E})$ and $\vartheta^{\tau}_{\gamma_{0}}(q)$ are uniformly continuous in $q \in \mathcal{B}_{0}$ (in particular, if $\mathcal{B}_{0}$ is compact), this will imply \eqref{EQ: EventualConvergenceCocycles}.
\end{remark}

Recall the largest uniform Lyapunov exponent $\lambda_{1}(\Xi_{(\gamma)})$ of $\Xi_{(\gamma)}$ defined below \eqref{EQ: LargetLyapunovExponentWRTMetric}.

\begin{theorem}
	\label{TH: UpperSemiContLargLyapExp}
	In the above context, suppose there exists a subset $\mathcal{B}_{0} \subset \mathcal{Q}$ which contains $\mathcal{A}_{\gamma}$ for any $\gamma \in \Gamma$, and let $\Xi_{(\gamma)}$ eventually converge to $\Xi_{(\gamma_{0})}$ on $\mathcal{B}_{0}$ as $\gamma \to \gamma_{0}$. Then
	\begin{equation}
		\label{EQ: LargestUniformLyapExpSemiContIneq}
		\limsup_{\gamma \to \gamma_{0}} \lambda_{1}(\Xi_{(\gamma)}) \leq \lambda_{1}(\Xi_{(\gamma_{0})}).
	\end{equation}
\end{theorem}
\begin{proof}
	By rescaling the time if necessary, we may assume that $\tau = 1$. Furthermore, we may consider $\mathbb{T}_{+} = \mathbb{Z}_{+}$, thanks to Lemma \ref{LEM: FeketeLemma}. For any $\gamma \in \Gamma$, define a proper subadditive family $f_{\gamma}$ by $f^{t}_{\gamma}(q) \coloneq \ln\|\Xi^{t}_{(\gamma)}(q,\cdot)\|_{\mathcal{L}(\mathbb{E})}$ for all $t \in \mathbb{T}_{+}$ and $q \in \mathcal{Q}$. Let us show that such families $f_{\gamma}$ satisfy all the assumptions of Corollary \ref{COR: UpperSemiContSubadditive}. First, by \eqref{EQ: EventualConvergenceCocycles}, for all $k =1,2,\ldots$ and $\varepsilon > 0$, there exists a neighborhood $\mathcal{U}$ of $\gamma_{0}$ such that
	\begin{equation}
		\sup_{q \in \mathcal{B}_{0}}\| \Xi^{k}_{\gamma}(q,\cdot) \|_{\mathcal{L}(\mathbb{E})} \leq \sup_{q \in \mathcal{B}_{0}}\|\Xi^{k}_{\gamma_{0}}(q,\cdot)\|_{\mathcal{L}(\mathbb{E})} + \varepsilon \qquad \text{for any} \quad \gamma \in \mathcal{U}.
	\end{equation}
	By taking logarithms and then the limit superior as $\gamma \to \gamma_{0}$, we obtain
	\begin{equation}
		\limsup_{\gamma \to \gamma_{0}}\sup_{q \in \mathcal{A}_{\gamma}}f^{k}_{\gamma}(q) \leq \sup_{q \in \mathcal{Q}}\ln(\|\Xi^{k}_{\gamma_{0}}(q,\cdot)\|_{\mathcal{L}(\mathbb{E})} + \varepsilon)
	\end{equation}
	for any $\varepsilon>0$. Passing to the limit as $\varepsilon \to 0+$ yields \eqref{EQ: SubAddNumberFamilyFromFunctionsConvergence}.
\end{proof}

Recall here the quantities $\omega_{d}(\Xi_{(\gamma)})$, $\dim_{\operatorname{L}}\Xi_{(\gamma)}$, and $\dim^{KY}_{\operatorname{L}}\Xi_{(\gamma)}$ defined in \eqref{EQ: AveragedFunctionOfSingularValues}, \eqref{EQ: LyapunovDimensionDefinition}, and \eqref{EQ: Kaplan-YorkeFormula}, respectively. Similarly, we have the following corollary.
\begin{corollary}(Robustness Principle)
	\label{COR: UpperSemiContUniformLyapExpAndLyapDim}
	Under the conditions of Theorem \ref{TH: UpperSemiContLargLyapExp}, suppose that $\mathbb{E} = \mathbb{H}$ is a Hilbert space. Then for any $d \geq 0$, we have
	\begin{equation}
		\label{EQ: UpperSemiContSumUniformLE}
		\limsup_{\gamma \to \gamma_{0}}\omega_{d}(\Xi_{(\gamma)}) \leq \omega_{d}(\Xi_{(\gamma_{0})}).
	\end{equation}
    Furthermore,
    \begin{equation}
		\label{EQ: UpperSemiContLyapDimens}
		\begin{split}
			&\limsup\limits_{\gamma \to \gamma_{0}}\dim^{KY}_{\operatorname{L}}\Xi_{(\gamma)} \leq \dim^{KY}_{\operatorname{L}}\Xi_{(\gamma_{0})},\\
			&\limsup\limits_{\gamma \to \gamma_{0}}\dim_{\operatorname{L}}\Xi_{(\gamma)} \leq \dim_{\operatorname{L}}\Xi_{(\gamma_{0})}.
		\end{split}
    \end{equation}
\end{corollary}
\begin{proof}
	For \eqref{EQ: UpperSemiContSumUniformLE}, we apply Corollary \ref{COR: UpperSemiContSubadditive} to $f_{\gamma}$ given by $f^{t}_{\gamma}(q) \coloneq \ln \omega_{d}(\Xi^{t}_{(\gamma)}(q,\cdot))$. By Lemma \ref{COR: HornInequality} and the properties of $\Xi_{\gamma}$, this is a proper subadditive family. From the eventual convergence of cocycles, \eqref{EQ: SingularValueDef}, and \eqref{EQ: FuncSingularValueDef}, we have for any $k=1,2,\ldots$ that
	\begin{equation}
		\label{EQ: UpperSemiContConvergenceOmega}
		\lim_{\gamma \to \gamma_{0}}\sup_{q \in \mathcal{B}_{0}}|\omega_{d}(\Xi^{\tau k}_{\gamma}(q,\cdot)) - \omega_{d}(\Xi^{\tau k}_{\gamma_{0}}(q,\cdot))| = 0.
	\end{equation}
	By applying similar arguments as in the proof of Theorem \ref{TH: UpperSemiContLargLyapExp}, we obtain \eqref{EQ: UpperSemiContSumUniformLE}.
	
	For the first inequality in \eqref{EQ: UpperSemiContLyapDimens}, omitting the trivial case $\dim^{KY}_{\operatorname{L}}\Xi_{(\gamma_{0})} = \infty$, we act as follows. First, for any cocycle $\Xi$ and for any $d = m + \theta$ with $m=0,1,\ldots$ and $\theta \in (0,1]$, we define $\Xi_{d}$ as the cocycle in $\mathbb{H}^{\wedge (m+1)}$ with the mappings $\Xi^{t}_{d}(q,\cdot) \coloneq e^{(\theta-1)\lambda_{m+1}(\Xi)t} \cdot \Xi^{t}_{m+1}(q,\cdot)$, where $q \in \mathcal{Q}$, $t \geq 0$, and $\Xi_{m+1}$ is the $(m+1)$-fold multiplicative compound of $\Xi$. Clearly, $\lambda_{1}(\Xi_{d}) = \sum_{j=1}^{m}\lambda_{j}(\Xi) + \theta \lambda_{m+1}(\Xi)$. From this and \eqref{EQ: Kaplan-YorkeFormula}, we have 
	\begin{equation}
		\dim^{KY}_{\operatorname{L}}\Xi = \inf\{ d > 0 \ | \ \lambda_{1}(\Xi_{d}) < 0 \}.
	\end{equation}
    By applying Theorem \ref{TH: UpperSemiContLargLyapExp} to such extensions of $\Xi_{(\gamma)}$, for which the eventual convergence follows from \eqref{EQ: UpperSemiContConvergenceOmega} and \eqref{EQ: ExteriorNormSingularValuesIdentity}, we obtain that
    \begin{equation}
    	\limsup_{\gamma \to \gamma_{0}}\dim^{KY}_{\operatorname{L}}\Xi_{(\gamma)} < d \qquad \text{for any} \quad d > \dim^{KY}_{\operatorname{L}}\Xi_{(\gamma_{0})}.
    \end{equation}
    Taking it to the limit as $d$ tends to $\dim^{KY}_{\operatorname{L}}\Xi_{(\gamma_{0})}$ yields the desired inequality. 
    
    For the second inequality in \eqref{EQ: UpperSemiContLyapDimens}, we again omit the trivial case $\dim_{\operatorname{L}}\Xi_{(\gamma_{0})} = \infty$. By definition, for any $d > \dim_{\operatorname{L}}\Xi_{(\gamma_{0})}$ we have $\omega_{d}(\Xi_{(\gamma_{0})}) < 1$. Consequently, \eqref{EQ: UpperSemiContSumUniformLE} gives
    \begin{equation}
    	\limsup_{\gamma \to \gamma_{0}} \omega_{d}(\Xi_{(\gamma)}) \leq \omega_{d}(\Xi_{(\gamma_{0})}) < 1 \quad \text{and} \quad
    	\limsup\limits_{\gamma \to \gamma_{0}}\dim_{\operatorname{L}}\Xi_{(\gamma)} < d.
    \end{equation}
    Taking it to the limit as $d$ tends to $\dim_{\operatorname{L}}\Xi_{(\gamma_{0})}$ yields the desired result.
\end{proof}

Particular cases of the above result appear in the works of Eden \cite[Proposition 3.8]{EdenLocalEstimates1990} and Kawan, Matveev, and Pogromsky \cite[Corollary 8]{KawanPogromsky2021}, but these are proved using redundant properties. Specifically, in \cite{EdenLocalEstimates1990}, the proof relies on the existence of a maximizing point, while in \cite{KawanPogromsky2021}, a variational description via adapted metrics is employed. Our proof demonstrates that it is the elementary subadditivity --- which, in the compact case, concerns only number families --- that yields the upper semicontinuity.

Moreover, some works exhibit misunderstandings regarding the presence of upper semicontinuity. In these studies, the inequality $\lambda_{1}(\Xi) + \lambda_{2}(\Xi) < 0$ and its preservation under small perturbations of $\Xi$ (robustness), where \(\Xi\) is the derivative cocycle of $\vartheta$ in $\mathcal{Q}$, is of special interest due to the generalized Bendixson criterion of Smith \cite{Smith1986HD} and its extension to infinite dimensions proposed by Li and Muldowney \cite{LiMuldowney1995LowBounds}. If the conditions for such a criterion are robust, then variants of Pugh's closing lemma can be applied to establish the convergence of all trajectories to equilibria. A notable example in this direction is \cite[Proposition 3.3]{LiMuldowney1996SIAMGlobStab}, which establishes a form of upper semicontinuity of the averaged exponent $\bar{\alpha}_{\mathfrak{n}}(\Xi)$ (see Remark \ref{REM: AveragedExponents}) computed with respect to a particular metric $\mathfrak{n}$ on $\mathbb{H}^{\wedge 2}$. Although this result is formally different, it is ultimately used to ensure that the contraction condition $\lambda_{1}(\Xi) + \lambda_{2}(\Xi) < 0$ is preserved under small perturbations. From Corollary \ref{COR: UpperSemiContUniformLyapExpAndLyapDim}, we observe that $\lambda_{1}(\Xi) + \lambda_{2}(\Xi) < 0$ is itself a robust condition and does not require any additional assumptions.

Another misunderstanding is evident in the discussion at the end of \cite{LiMuldowney1995LowBounds}, which includes a statement regarding the insufficiency (for robustness) of a similar inequality, $\lambda_{1}(\Xi;q) + \lambda_{2}(\Xi;q) < 0$ for any $q \in \mathcal{Q}$, in terms of local Lyapunov exponents\footnote{Here, \(\exp(\lambda_{1}(\Xi;q))\) is defined as the limit superior in \eqref{EQ: ErgodicDVolumes} for $d=1$, and $\exp(\lambda_{1}(\Xi;q) + \lambda_{2}(\Xi;q))$ is obtained from the same operation for $d=2$.} $\lambda_{1}(\Xi;q)$ and $\lambda_{2}(\Xi;q)$. A more precise formulation appears in the subsequent work by the same authors. Specifically, in \cite[Remark on p.~1076]{LiMuldowney1996SIAMGlobStab}, they state\footnote{For convenience, we have omitted some notations from the original text.}: ``Since the global attractor is not necessarily preserved under a local $C^1$ perturbation at a nonwandering point, the criterion may not be robust under such perturbations''. This clearly indicates the presence of a misconception.

From Corollary \ref{COR: DVolumesUniformAsDVolumesOverMeasure}, it follows that $\lambda_{1}(\Xi) + \lambda_{2}(\Xi) = \lambda_{1}(\Xi; q) + \lambda_{2}(\Xi; q)$ for some $q \in \mathcal{A}$. Consequently, the local Lyapunov exponents provide the same inequality. It is also evident that the condition $\lambda_{1}(\Xi) + \lambda_{2}(\Xi) < 0$ represents the strongest inequality attainable through the analysis of linearization cocycles without additional assumptions. Therefore, it is unsurprising that known applications of the generalized Bendixson criterion focus implicitly on verifying such conditions by computing maximized or averaged exponents in adapted metrics. For further discussion, we refer to \cite{AnikushinRomanov2023FreqConds}.

%% file: Symmetrization.tex
\section{Additive symmetrization of operators}
\label{SEC: SymmetrizationOfOperators}
Let $A$ be a linear operator in a Hilbert space $\mathbb{H}$ with domain $\mathcal{D}(A)$. We introduce the \textit{trace numbers} $\beta_{1}(A) \geq \beta_{2}(A) \geq \ldots$ of $A$ by induction from the relations for any $k=1,2,\ldots$
\begin{equation}
	\label{EQ: TraceNumbersDef}
	\beta_{1}(A) + \cdots + \beta_{k}(A) = \sup_{\substack{\mathbb{L} \subset \mathcal{D}(A), \\ \dim\mathbb{L}=k}} \operatorname{Re}\operatorname{Tr}\left(\Pi_{\mathbb{L}} \circ A \circ \Pi_{\mathbb{L}}\right),
\end{equation}
where the supremum is taken over all $k$-dimensional subspaces $\mathbb{L}$, $\Pi_{\mathbb{L}}$ is the orthogonal projector onto $\mathbb{L}$, and $\operatorname{Tr}$ denotes the trace functional. If $\dim \mathbb{H} <\infty$, it is convenient to set $\beta_{k}(A) \coloneq -\infty$ for any $k > \dim \mathbb{H}$.

By virtue of the Liouville trace formula \eqref{EQ: GeneralizedTraceFormula}, trace numbers represent infinitesimal analogs of singular values. In particular, they provide upper bounds for Lyapunov exponents by applying certain maximization or averaging procedures, as discussed in Section \ref{SEC: ComputationOfInfExponents}.

Let $e_{1}, \ldots, e_{k} \in \mathcal{D}(A)$ form an orthonormal basis for a $k$-dimensional subspace $\mathbb{L}$. Then
\begin{equation}
	\label{EQ: TraceComputationTroughBasis}
	\operatorname{Re}\operatorname{Tr}\left(\Pi_{\mathbb{L}} \circ A \circ \Pi_{\mathbb{L}}\right) = \operatorname{Re}\sum_{j=1}^{k}\left\langle Ae_{j},e_{j}\right\rangle_{\mathbb{H}}.
\end{equation}
From this, it can be seen that $\beta_{k}(A)$ is nonincreasing in $k$.

A self-adjoint operator $S$ in $\mathbb{H}$ is called a \textit{symmetrization} of $A$ if $\beta_{k}(A) = \beta_{k}(S)$ for any $k=1,2,\ldots$. We address the following problem.
\begin{problem}
	For a given operator $A$, determine whether it admits effective symmetrization.
\end{problem}

This problem is related to the computation of the trace numbers of $A$. Below, we will establish Theorem \ref{TH: ComputationSwedgeSpectralBound}, which describes the trace numbers of a self-adjoint operator $S$ in terms of eigenvalues above the essential spectral bound and the essential spectral bound itself (we call these \textit{characteristic numbers} of $S$).

Clearly, we can always construct a self-adjoint operator having all $\beta_{k}(A)$ as eigenvalues, provided that each $\beta_{k}(A)$ is finite. However, this symmetrization is not effective. In Theorem \ref{TH: AdditiveSymmetrizationAndFriedrichsExtension}, we provide a partial answer to the problem in terms of Friedrichs extensions.

The standard approach for constructing an effective symmetrization of $A$ is as follows. Let $A^{*}$ be the adjoint of $A$ defined on domain $\mathcal{D}(A^{*})$. Suppose $\mathcal{D}(A) \cap \mathcal{D}(A^{*})$ is dense in $\mathbb{H}$, and let the symmetric operator $S_{0}(A) \coloneq (A+A^{*})/2$, which is defined on the intersection of the domains, admit an extension to a self-adjoint operator $S_{A}$ on $\mathbb{H}$. In this case, $S_{A}$ is called an \textit{additive symmetrization} of $A$. For $\mathbb{L} \subset \mathcal{D}(A) \cap \mathcal{D}(A^{*})$, we can proceed in \eqref{EQ: TraceComputationTroughBasis} as
\begin{equation}
	\label{EQ: AdditiveSymmetrizationBasicRelation}
	\begin{split}
		\operatorname{Re}\operatorname{Tr}\left(\Pi_{\mathbb{L}} \circ A \circ \Pi_{\mathbb{L}}\right) = \operatorname{Re}\sum_{j=1}^{k}\left\langle Ae_{j},e_{j}\right\rangle_{\mathbb{H}} =\\= \sum_{j=1}^{k}\left\langle S_{0}(A)e_{j},e_{j}\right\rangle_{\mathbb{H}} = \operatorname{Tr}\left(\Pi_{\mathbb{L}} \circ S_{A} \circ \Pi_{\mathbb{L}}\right).
	\end{split}
\end{equation}
We say that $S_{A}$ is (or provides) a \textit{proper} additive symmetrization of $A$ if $S_{A}$ is a symmetrization of $A$, i.e., $\beta_{k}(A) = \beta_{k}(S_{A})$ for any $k=1,2,\ldots$. It is crucial that \eqref{EQ: AdditiveSymmetrizationBasicRelation} is not sufficient\footnote{In \cite[Theorem 7]{Anikushin2022Semigroups}, the conditions for a closed operator to admit a proper additive symmetrization are given, but the statement and proof are incorrect. Here, the mistake was caused by a misunderstanding of the mentioned insufficiency. Precisely, the proof breaks at the phrase ``as a consequence of the symmetry'' showing the generally wrong identity \cite[(3.28)]{Anikushin2022Semigroups}.} for an additive symmetrization $S_{A}$ to be proper.

In view of the above, the problem of interest is as follows: suppose we have a class of operators $A$ and a class of metrics (inner products) in $\mathbb{H}$. We then want to determine the metrics for which $A$ admits a proper additive symmetrization and compute (or at least estimate from above) the corresponding trace numbers.

In Section \ref{SEC: DelayEqsSymmetrization}, we investigate the problem of symmetrization for delay operators $A$ in adapted metrics. In particular, we show that $S_{0}(A)$ may not be densely defined, in which case it also does not admit densely defined symmetric extensions, and $\beta_{1}(A) = \infty$. This corresponds to the degeneracy of the metric in the considered class. In the nondegenerate case, there exists a unique additive symmetrization $S_{A}$ (given by a bounded self-adjoint operator) that is not necessarily proper, see Theorem \ref{TH: DelayOperatorSymmetrizationExistence} and Remark \ref{REM: DelayOperatorsNoProperSymm}. Furthermore, $S_{A}$ is proper if and only if for any $k=1,2,\ldots$
\begin{equation}
	\label{EQ: TraceNumbersIdentityMaxOnAdjointDomain}
	\sup_{\substack{\mathbb{L} \subset \mathcal{D}(A), \\ \dim\mathbb{L}=k}} \operatorname{Re}\operatorname{Tr}\left(\Pi_{\mathbb{L}} \circ A \circ \Pi_{\mathbb{L}}\right) = \sup_{\substack{\mathbb{L} \subset \mathcal{D}(A) \cap \mathcal{D}(A^{*}), \\ \dim\mathbb{L}=k}} \operatorname{Re}\operatorname{Tr}\left(\Pi_{\mathbb{L}} \circ A \circ \Pi_{\mathbb{L}}\right),
\end{equation}
and necessary and sufficient conditions for this are expressed in Theorem \ref{TH: TraceNumbersAdmitSymmetrizationDelayOperator} and Remark \ref{REM: DelayOperatorsNoProperSymm}. Note that \eqref{EQ: TraceNumbersIdentityMaxOnAdjointDomain} is an interesting property that shows that the addition of adjoint boundary conditions does not change the variational problem.

Recall that $S_{0}(A)$, as a symmetric operator, is closable. If its closure is self-adjoint, the operator is called \textit{essentially self-adjoint}. In fact, the above discussed property reflects the essential self-adjointness of $S_{0}(A)$.
\begin{proposition}
	\label{PROP: AdditiveSymmEssAdj}
	Suppose $S_{0}(A)$ is essentially self-adjoint. Then the closure $S_{A}$ of $S_{0}(A)$ provides a proper additive symmetrization of $A$ if and only if \eqref{EQ: TraceNumbersIdentityMaxOnAdjointDomain} is satisfied.
\end{proposition}
\begin{proof}
	Indeed, since $S_{A}$ is the closure of $S_{0}(A)$, we obtain $\beta_{k}(S_{0}(A)) = \beta_{k}(S_{A})$ for any $k=1,2,\ldots$. From this, the conclusion follows.
\end{proof}

A more general criterion can be obtained in terms of the Friedrichs extension of $S_{0}(A)$. For this, we need to introduce some standard terminology, see \cite[Section 1.2.1]{FrankLaptevWeidl2022} for the related theory.

Let $B$ be a symmetric bilinear form (for brevity, form) in $\mathbb{H}$ defined on a dense domain $\mathcal{D}(B)$. It is called \textit{bounded from above} if there exists $M>0$ such that
\begin{equation}
	\label{EQ: BoundedFromAboveFormDefinition}
	B(v,v) \leq M |v|^{2}_{\mathbb{H}} \qquad \text{for any} \quad v \in \mathcal{D}(B).
\end{equation}
Then $\mathcal{D}(B)$ can be endowed with the \textit{energetic norm}
\begin{equation}
	|v|^{2}_{B} \coloneq -B(v,v) + (M+1) |v|^{2}_{\mathbb{H}} \qquad \text{for any} \quad v \in \mathcal{D}(B).
\end{equation}
Note that $\mathcal{D}(B) \ni v \mapsto B(v,v)$ is continuous in the norm $|\cdot|_{B}$. Moreover, different choices of $M$ in \eqref{EQ: BoundedFromAboveFormDefinition} result in equivalent norms.

We say that $B$ is \textit{closed} if $\mathcal{D}(B)$ endowed with the norm $|\cdot|_{B}$ is a Hilbert space. If this is not the case, but there exists a closed form $\bar{B}$ defined on the domain $\mathcal{D}(\bar{B})$ which is the completion of $\mathcal{D}(B)$ with respect to $|\cdot|_{B}$, then $B$ is called \textit{closable}, and $\bar{B}$ is called the \textit{closure} of $B$.

Introduce the trace numbers $\beta_{1}(B) \geq \beta_{2}(B) \geq \ldots$ of $B$ by induction from the relations for any $k=1,2,\ldots$
\begin{equation}
	\beta_{1}(B) + \cdots + \beta_{k}(B) = \sup_{e_{1}, \ldots, e_{j} \in \mathcal{D}(B)}\sum_{j=1}^{k}B(e_{j},e_{j})_{\mathbb{H}},
\end{equation}
where the supremum is taken over all orthonormal families $e_{1},\ldots,e_{k} \in \mathcal{D}(B)$. If $B$ is closable, the continuity in the energetic norm yields $\beta_{k}(B) = \beta_{k}(\bar{B})$ for any $k=1,2,\ldots$.

With any symmetric operator $S$ in $\mathbb{H}$ defined on a dense domain $\mathcal{D}(S)$, we can associate the form $B_{S}(v,w) \coloneq \langle Sv, w \rangle_{\mathbb{H}}$ defined for all $v,w \in \mathcal{D}(S)$. We call $S$ \textit{bounded from above} if $B_{S}$ is bounded from above, or, equivalently, $\beta_{1}(S) < \infty$. An important fact is that for any bounded from above symmetric operator, the form $B_{S}$ is closable. In this case, the domain $\mathcal{Q}(S) \coloneq \mathcal{D}(\bar{B}_{S})$ is called the \textit{form-domain} of $S$. Moreover, there exists the \textit{Friedrichs extension} $S_{F}$ of $S$ that is a unique self-adjoint operator whose domain $\mathcal{D}(S_{F})$ satisfies
\begin{equation}
	\mathcal{D}(S) \subset \mathcal{D}(S_{F}) \subset \mathcal{Q}(S).
\end{equation}
In addition, we have $\bar{B}_{S_{F}} = \bar{B}_{S}$. Clearly, this gives that $\beta_{k}(S) = \beta_{k}(S_{F}) = \beta_{k}(\bar{B}_{S})$ for any $k=1,2,\ldots$.

Finally, with $A$ as above, we associate the form $B_{A}$ on $\mathcal{D}(A)$ defined by
\begin{equation}
	B_{A}(v,w) \coloneq \frac{1}{2} \left( \langle Av, w \rangle_{\mathbb{H}} + \langle v, Aw \rangle_{\mathbb{H}} \right) \qquad \text{for all} \quad v,w \in \mathcal{D}(A).
\end{equation}
Note that $B_{A}(v,v) = \operatorname{Re}\langle Av, v\rangle_{\mathbb{H}}$. Consequently, $B_{A}$ is bounded from above if and only if $\beta_{1}(A) < \infty$.

One can strengthen Proposition \ref{PROP: AdditiveSymmEssAdj} in terms of Friedrichs extensions as follows.
\begin{theorem}
	\label{TH: AdditiveSymmetrizationAndFriedrichsExtension}
	Suppose $\mathcal{D}(A) \cap \mathcal{D}(A^{*})$ is dense in $\mathbb{H}$, and let $\beta_{1}(S_{0}(A)) < \infty$. Then the Friedrichs extension $S_{A}$ of $S_{0}(A)$ provides a proper additive symmetrization of $A$ if and only if \eqref{EQ: TraceNumbersIdentityMaxOnAdjointDomain} is satisfied.
\end{theorem}
\begin{proof}
	Under the assumptions, the Friedrichs extension $S_{A}$ is well-defined. Since it satisfies $\beta_{k}(S_{0}(A)) = \beta_{k}(S_{A})$ for any $k=1,2,\ldots$, we immediately have the conclusion.
\end{proof}

A more naive approach would lead to the following result.
\begin{proposition}
	Suppose that  $\beta_{1}(A) < \infty$, and let $\mathcal{D}(A) \cap \mathcal{D}(A^{*})$ be dense in $\mathcal{D}(A)$ in the energetic norm $|\cdot|_{B_{A}}$. Then the Friedrichs extension $S_{A}$ of $S_{0}$ provides a proper additive symmetrization of $A$.
\end{proposition}
\begin{proof}
	Since $\mathcal{D}(A) \cap \mathcal{D}(A^{*})$ is dense in $\mathbb{H}$, the symmetric operator $S_{0}(A)$ is densely defined. From $\beta_{1}(S_{0}(A)) \leq \beta_{1}(A) < \infty$, we obtain that $S_{0}(A)$ is bounded from above. Thus, there exists the Friedrichs extension $S_{A}$ of $S_{0}(A)$.
	
	By the density in the norm $|\cdot|_{B_{A}}$, it is clear that $\mathcal{Q}(S_{0}) \supset \mathcal{D}(A)$, and $\bar{B}_{S_{0}(A)}$ coincides with $B_{A}$ on $\mathcal{D}(A)$. Consequently, for any $k=1,2,\ldots$, we have
	\begin{equation}
		\beta_{k}(A) \geq \beta_{k}(S_{0}(A)) = \beta_{k}(S_{A}) = \beta_{k}(\bar{B}_{S_{0}(A)}) \geq \beta_{k}(A)
	\end{equation}
	that shows the statement. 
\end{proof}

However, the density in the energetic norm does not hold in the problems of our interest. Indeed, by taking a look at the proof of Theorem \ref{TH: TraceNumbersAdmitSymmetrizationDelayOperator}, one can see that all the significant (unbounded in $\mathbb{H}$) terms of $B_{A}$ shown in \eqref{EQ: QuadraticFormDelayOperatorDecomposition} become nonnegative on $\mathcal{D}(A) \cap \mathcal{D}(A^{*})$, although $B_{A}$ may take negative values on $\mathcal{D}(A)$. Thus, there is no density in the energetic norm. On the other hand, \eqref{EQ: TraceNumbersIdentityMaxOnAdjointDomain} is satisfied in this case.

Note that the violation of \eqref{EQ: TraceNumbersIdentityMaxOnAdjointDomain} in the studied case of delay operators is adjacent to $\beta_{k}(A) = +\infty$ for any $k$. From this, the following problem is interesting.
\begin{problem}
	Does there exist an operator $A$, being the generator of a $C_{0}$-semigroup in a Hilbert space $\mathbb{H}$, for which $\beta_{1}(A)$ is finite, but $A$ does not admit a proper additive symmetrization?
\end{problem}

\begin{remark}
	Often, instead of computing the trace numbers, one just estimates them directly from \eqref{EQ: TraceNumbersDef} by utilizing some specificity of the problem. For example, in the case of 2D Navier-Stokes equations, this is usually done with the aid of Lieb-Thirring inequalities \cite{ChepyzhovIlyin2004}. However, it is possible to justify the presence of a unique proper additive symmetrization $S_{A}$ in this case, as we deal with a dominated perturbation $A$ of the Stokes operator for which $A^{*}$ is defined on the same domain, so $\beta_{k}(A)=\beta_{k}(S_{A})$ is trivially satisfied for any $k$.
\end{remark}

In the remaining part of the section, we aim to express the trace numbers of a self-adjoint operator in terms of its spectrum.

For a self-adjoint operator $S$, by $s_{e}(S) \in [-\infty;+\infty]$ we denote the \textit{essential spectral bound} of $S$, which is given by the supremum of the essential spectrum (see \cite[Section 1.1.7]{FrankLaptevWeidl2022}). If $S$ is bounded from above, then $s_{e}(S) < +\infty$, and there are at most countably many, say $N \geq 0$ or $N = \infty$, eigenvalues $\lambda > s_{e}(S)$, and any such eigenvalue has finite multiplicity. Let $\lambda_{1}(S) \geq \lambda_{2}(S) \geq \ldots$ denote the arranged eigenvalues counted with multiplicities. Define the \textit{characteristics numbers} $\alpha_{1}(S) \geq \alpha_{2}(S) \geq \ldots$ of $S$ by
\begin{equation}
	\label{EQ: CharacteristicNumbersS}
	\alpha_{j}(S) \coloneq \begin{cases}
		\lambda_{j}(S) \qquad &\text{if} \quad j \leq N,\\
		s_{e}(S) \qquad &\text{if} \quad j > N.
	\end{cases}
\end{equation}
\begin{remark}
	If $S$ is not bounded from above, all characteristic numbers are set to be $+\infty$ by definition.
	
	If $S$ is bounded from above, then it generates a $C_{0}$-semigroup, and the characteristics numbers are precisely the Lyapunov exponents of the semigroup, see \cite[Theorem A.2.5]{Thieullen1992}. \qed
\end{remark}

\begin{lemma}
	\label{LEM: ComparisonLemma}
	For a self-adjoint operator $S$, let $\Pi$ be an orthogonal projector with $m$-dimensional range lying in $\mathcal{D}(S)$. Then, for the truncated operator $S_{\Pi} \coloneq \Pi S \Pi$, we have $\lambda_{j}(S_{\Pi}) \leq \alpha_{j}(S)$ for each $j \in \{1,\ldots, m\}$.
\end{lemma}
\begin{proof}
	We assume that $S$ is bounded from above, otherwise there is nothing to prove.
	
	By the Courant--Fischer--Weyl min–max principle \cite[Theorem 1.27]{FrankLaptevWeidl2022}, we have
	\begin{equation}
		\label{EQ: MinmaxWeilLemma}
		\alpha_{j}(S) = \inf_{v_{1},\ldots, v_{j-1} \in \mathbb{H}} \sup_{\substack{v \bot v_{1},\ldots, v_{j-1},\\ 0\not=v \in \mathcal{D}(S)} } \frac{\langle S v, v\rangle_{\mathbb{H}}}{|v|^{2}_{\mathbb{H}}}
	\end{equation}
	for any $j = 1, 2, \ldots$.
	
	Let $e_{1},\ldots, e_{m} \in \mathcal{D}(S)$ be the eigenvectors of $S_{\Pi}$ such that $S_{\Pi} e_{j} = \lambda_{j}(S_{\Pi}) e_{j}$ for each $j \in \{1,\ldots,m\}$. Then, for any $v_{1},\ldots,v_{j-1} \in \mathbb{H}$, there exists a nonzero $v \in \operatorname{Span}\{ e_{1},\ldots,e_{j} \}$ which is orthogonal to $v_{1},\ldots,v_{j-1}$. In particular, $v \in \mathcal{D}(S)$ and
	\begin{equation}
		(Sv,v)_{\mathbb{H}} = (S\Pi v, \Pi v)_{\mathbb{H}} = (S_{\Pi}v,v)_{\mathbb{H}} \geq \lambda_{j}(S_{\Pi}) |v|^{2}_{\mathbb{H}}.
	\end{equation} 
	This and \eqref{EQ: MinmaxWeilLemma} yield the desired inequality $\lambda_{j}(S_{\Pi}) \leq \alpha_{j}(S)$ for each $j \in \{1,\ldots, m\}$. 
\end{proof}

Finally, we establish the following key result.
\begin{theorem}
	\label{TH: ComputationSwedgeSpectralBound}
	For a self-adjoint operator $S$ in a separable Hilbert space $\mathbb{H}$, the trace numbers (see \eqref{EQ: TraceNumbersDef}) coincide with the characteristic numbers (see \eqref{EQ: CharacteristicNumbersS}), that is, $\beta_{k}(S) = \alpha_{k}(S)$ for any $k=1,2,\ldots$.
\end{theorem}
\begin{proof}
	We assume that $S$ is bounded from above; otherwise, all the numbers are just $+\infty$.
	
	By taking subspaces $\mathbb{L}$ spanned by the eigenvectors or approximate eigenvectors corresponding to consecutive characteristic numbers, from \eqref{EQ: TraceNumbersDef} we see that 
	$\sum_{j=1}^{m}\beta_{j}(S) \geq \sum_{j=1}^{m}\alpha_{j}(S)$ for any $m=1,2,\ldots$.
	
	Conversely, let $\mathbb{L} \subset \mathcal{D}(S)$ be a fixed $m$-dimensional subspace with an orthonormal basis $e_{1}, \ldots, e_{m}$. Let $\Pi$ be the orthogonal projector onto $\mathbb{L}$. Applying Lemma \ref{LEM: ComparisonLemma} to $S$ and such $\Pi$ gives
	\begin{equation}
		\sum_{j=1}^{m}(S e_{j}, e_{j})_{\mathbb{H}} = \sum_{j=1}^{m}(S_{\Pi} e_{j}, e_{j})_{\mathbb{H}} = \sum_{j=1}^{m}\lambda_{j}(S_{\Pi}) \leq \sum_{j=1}^{m} \alpha_{j}(S).
	\end{equation}
	Since this holds for any $\mathbb{L}$, we obtain $\sum_{j=1}^{m}\beta_{j}(S) \leq \sum_{j=1}^{m}\alpha_{j}(S)$ for any $m=1,2,\ldots$. 
	
	Finally, from $\sum_{j=1}^{m}\beta_{j}(S) = \sum_{j=1}^{m}\alpha_{j}(S)$ by induction on $m$, one establishes the coincidence of $\beta_{k}(S)$ and $\alpha_{k}(S)$ for any $k=1,2,\ldots$. 
\end{proof}

\begin{remark}
	\label{REM: TraceNumbersThParticularCases}
	There are particular cases of Theorem \ref{TH: ComputationSwedgeSpectralBound}, such as \cite[Lemmas 2.1 and 2.2, Chapter 6]{Temam1997} or \cite[Proposition 1.33]{FrankLaptevWeidl2022}, although the statement in full generality can be easily proven using the methods presented in these monographs. Moreover, the validity of Theorem \ref{TH: ComputationSwedgeSpectralBound} is also noted in \cite[Remark 2.13]{SavostianovZelik2016}, which refers to \cite{Temam1997} for the case of bounded operators. Therefore, it is somewhat surprising that a complete proof in full generality has not been published previously.
\end{remark}

\begin{remark}
	\label{REM: TraceNumbersCriticalSubSpace}
	Let $S$ be as above. For any $m=1,2,\ldots$, it can be shown (via the spectral theorem) that the algebraic $m$-fold additive compound $S^{[\wedge m]}$ of $S$, which is defined by \eqref{EQ: AlgebraicAdditiveCompoundDef}, is essentially self-adjoint, i.e., it admits a unique self-adjoint extension. Let us denote it by the same symbol $S^{[\wedge m]}$. Then it can be shown that
	\begin{equation}
		\sum_{j=1}^{m}\alpha_{j}(S) = \sup \operatorname{spec}S^{[\wedge m]}.
	\end{equation}
	From the proof of Theorem \ref{TH: ComputationSwedgeSpectralBound}, it is seen that if $\alpha_{1}(S), \ldots, \alpha_{m}(S)$ are eigenvalues of $S$, then any corresponding to them eigenvectors $e_{1},\ldots,e_{m}$ span a subspace $\mathbb{L}$ on which the supremum in \eqref{EQ: TraceNumbersDef} with $A = S$ is achieved. Furthermore, $\sum_{j=1}^{m}\alpha_{j}(S)$ is an eigenvalue of $S^{[\wedge m]}$ corresponding to the eigenvector $e_{1}\wedge \cdots \wedge e_{m}$.
\end{remark}